\documentclass[11pt,a4paper]{amsart}

\usepackage[utf8]{inputenc}
\usepackage[T1]{fontenc}      % Police contenant les caractères français
\usepackage{geometry}         % Définir les marges
\usepackage[french, english]{babel}  % Placez ici une liste de langues, la
\usepackage{amsmath}
\usepackage{amsfonts}
\usepackage{amsthm}
\usepackage{pifont}

\usepackage{amscd}
\usepackage{hyperref}
\usepackage{xcolor}
\usepackage{enumitem}

\usepackage{graphicx}

\usepackage{pstricks,pst-plot}

   \usepackage{amssymb}
   \usepackage{verbatim}
   \newif\ifpdf
   \ifpdf
     \usepackage[pdftex]{graphicx}
     \usepackage[pdftex]{hyperref}
   \else
     \usepackage{graphicx}
   \fi

\usepackage[all]{xy}

\usepackage{tikz}

\newcommand{\R}{\mathbb{R}}

\newcommand{\N}{\mathbb{N}}
\newcommand{\Z}{\mathbb{Z}}
\newcommand{\C}{\mathbb{C}}

\newcommand{\D}{\mathbb{D}}

\newcommand{\f}{\mathrm{f}}

\newtheorem{thm}{Theorem}[section]
\newtheorem{prop}[thm]{Proposition}
\newtheorem{coro}[thm]{Corollary}
\newtheorem{defi}[thm]{Definition}
\newtheorem{lem}[thm]{Lemma}

\newtheorem{rem}[thm]{Remark}
\newtheorem{claim}{Claim}

\newcommand{\rs}{\mathbb{P}^1}

\newcommand{\limn}{\lim_{n \rightarrow \infty}}

\newcommand{\bcal}{\mathcal{B}}

\newcommand{\lcal}{\mathcal{L}}

\newcommand{\re}{\mathrm{Re}}
\newcommand{\im}{\mathrm{Im}}

\newcommand{\fcal}{\mathcal{F}}
\renewcommand{\H}{\mathbb{H}}

\newcommand{\eps}{\epsilon}
\newcommand{\teps}{\tilde{\epsilon}}

\newcommand{\ecal}{\mathcal{E}}
\newcommand{\tbold}{\mathbf{t}}
\newcommand{\psev}{\mathcal{P}_7}

\author{Matthieu Astorg}
\thanks{This research was partially supported by the ANR grant Fatou ANR-17-CE40-0002-01. }

\author{Luka Boc Thaler$^{\dag}$}
\author{Han Peters}

\address{M. Astorg: Institut Denis Poisson, Collegium Sciences et Techniques, Université d'Orléans
	Rue de Chartres B.P. 6759
	45067 Orléans cedex 2 France.}
\email{matthieu.astorg@univ-orleans.fr}
\address{L. Boc Thaler: Faculty of Education, University of Ljubljana, SI--1000 Ljubljana, Slovenia. Institute of Mathematics, Physics and Mechanics, Jadranska 19, 1000 Ljubljana, Slovenia.} \email{luka.boc@pef.uni-lj.si}

\address{ H. Peters: Korteweg de Vries Institute for Mathematics\\
University of Amsterdam\\
the Netherlands} \email{hanpeters77@gmail.com}

\thanks{$^{\dag}$  Supported by the SIR grant ``NEWHOLITE - New methods in holomorphic iteration'' no. RBSI14CFME and by the research program P1-0291 from ARRS, Republic of Slovenia}

\title{Wandering domains arising from Lavaurs maps with Siegel disks}

\begin{document}

\maketitle
\selectlanguage{english}
\begin{abstract}
The first example of polynomial maps with wandering domains were constructed in 2016 by the first and last authors, together with Buff, Dujardin and Raissy.
In this paper, we construct a second example with different dynamics, using a Lavaurs map with a Siegel disk
instead of an attracting fixed point.
We prove a general necessary and sufficient condition for the existence of a trapping domain for non-autonomous compositions of maps converging parabolically towards a Siegel type limit map.
Constructing a skew-product satisfying this condition requires precise estimates on the convergence to the Lavaurs map, which we obtain
by a new approach.
We also give a self-contained construction of parabolic curves, which are integral to this new method.
\end{abstract}

\section{Introduction}
%\subsection{Description of results}
Rational functions do not have wandering domains, a classical result due to Sullivan \cite{Sullivan}. Recently in \cite{ABDPR} it was shown that there do exist polynomial maps in two complex variables with wandering Fatou components. The maps constructed in \cite{ABDPR} are polynomial skew products of the form
$$
(z,w) \mapsto (f_w(z), g(w)),
$$
where $g(w)$ and $f_w(z) = f(z,w)$ are polynomials in respectively one and two variables. While the construction holds for families of maps with arbitrarily many parameters, the constructed examples are essentially unique: they all arise from similar behavior and cannot easily be distinguished in terms of the geometry of the components or qualitative behavior of the orbits in the components. The goal in this paper is to modify the construction in \cite{ABDPR} to obtain quite different examples of wandering Fatou components. Our construction requires much more precise convergence estimates, forcing us to revisit and clarify the original proof, obtaining a better understanding of the methodology.

The maps considered in \cite{ABDPR} are of the specific form
\begin{equation}\label{equation:form of P}
P: (z,w) \mapsto (f(z) + \frac{\pi^2}{4} w, g(w)),
\end{equation}
where $f(z) = z + z^2 + O(z^3)$ and $g(w) = w - w^2 + O(w^3)$. Recall that the constant $\frac{\pi^4}{4}$ is essential to guarantee the following key result in \cite{ABDPR}:

\medskip

\noindent {\bf Proposition A.} \emph{
As $n\to +\infty$, the sequence of maps
$$
(z,w) \mapsto P^{\circ 2n+1}\left(z,g^{\circ n^2}(w)\right)
$$
converges locally uniformly in $\mathcal{B}_f \times \mathcal{B}_g$ to the map
$$
(z,w)\mapsto \left( \mathcal{L}_f(z),0\right).
$$}

\medskip

Here and later $\mathcal{B}_f$ and $\mathcal{B}_g$ refer to the \emph{parabolic basins} of respectively $f$ and $g$, and $\mathcal{L}_f$ refers to the \emph{Lavaurs map} of $f$ with phase $0$, see for example \cite{Lavaurs, Shishikura}. By carefully choosing the higher order terms of $f$, one can select Lavaurs maps with desired dynamical behavior.

In Proposition B of \cite{ABDPR} it was shown that $\mathcal{L}_f$ can have an attracting fixed point. The fact that $P$ has a wandering Fatou component is then a quick corollary of Proposition A. It seems very likely that one can similarly construct wandering domains when $\mathcal{L}_f$ has a parabolic fixed point, using the refinement of Proposition A presented here.

In this paper we will construct wandering domains arising when $\mathcal{L}_f$ has a Siegel fixed point: an irrationally indifferent fixed point with Diophantine rotation number. Compositions of small perturbations of $\mathcal{L}_f$ behave so subtly  that it is far from clear that Lavaurs maps with Siegel disks can produce wandering domains.

In order to control the behavior of successive perturbations, we prove a refinement of Proposition A with precise convergence estimates, showing that the convergence towards the Lavaurs map is ``parabolic''. Moreover, we study the behavior of non-autonomous systems given by maps converging parabolically to a limit map with a Siegel fixed point. We introduce an easily computable index characterizing the behavior of the non-autonomous systems.

In the next section we give more precise statements of our results, and prove how the combination of these results provides a new construction of wandering domains.

%If the convergence to the Lavaurs map would be sufficiently quick, so that the successive errors introduced are summable, then a Siegel disk would indeed give rise to a \emph{trapping domain}, which in turn would induce a wandering domain. However, as we will see, the successive errors are not summable. The reason that Siegel disks still can give rise to wandering domains is due to the fact that the errors converge to $0$ along a \emph{real} direction, which in combination with the rotation in the Siegel disk can still give convergent series, and in turn, trapping domains.

\section{Background and overview of results}

\subsection{Polynomial skew products and Fatou components} There is more than one possible interpretation of Fatou and Julia sets for polynomial skew products, see for example the paper \cite{Jonsson} for a thorough discussion. When we discuss Fatou components of skew products here, we consider open connected sets in $\mathbb C^2$ whose orbits are uniformly bounded, which of course implies equicontinuity. Since the degrees of $f$ and $g$ in \eqref{equation:form of P} are at least $2$, the complement of a sufficiently large bidisk is contained in the escape locus, which is connected, all other Fatou components are therefore bounded and have bounded orbits.

Given a Fatou component $U$ of $P$, normality implies that its projection onto the second coordinate $\pi_w(U)$ is contained in a Fatou component of $g$, which must therefore be periodic or preperiodic. Without loss of generality we may assume that this component of $g$ is invariant, and thus either an attracting basin, a parabolic basin or a Siegel disk.

The behavior of $P$ inside a Siegel disk of $g$ may be very complicated and has received little attention in the literature, but see \cite{PetersRaissy} for the treatment of a special case.

There have been a number of results proving the non-existence of wandering domains inside attracting basins of $g$. The non-existence of wandering domains in the super-attracting case was proved by Lilov in \cite{Lilov}, but it was shown in \cite{PetersVivas} that the arguments from Lilov cannot hold in the geometrically attracting case. The non-existence of wandering domains under progressively weaker conditions were proved in \cite{PetersSmit, Ji1}.

Here, as in \cite{ABDPR}, we will consider components $U$ for which $\pi_w(U)$ is contained in a parabolic basin of $g$. We assume that the fixed point of $g$ lies at the origin, and that $g$ is of the form $g(w) = w - w^2 + h.o.t.$, so that orbits approach $0$ tangent to the positive real axis. We will in fact make the stronger assumption $g(w) = w - w^2 + w^3 + h.o.t.$.

\subsection{Fatou coordinates and Lavaurs Theorem}

Consider a polynomial $f(z) = z - z^2 + a z^3 + h.o.t.$. For $r>0$ small enough we define incoming and outgoing petals
$$
P_f^\iota = \{|z+r| < r\} \; \; \mathrm{and} \; \; P_f^o = \{|z-r| < r\}.
$$
The incoming petal $P_f^\iota$ is forward invariant, and all orbits in $P_f^\iota$ converge to $0$. Moreover, any orbit which converges to $0$ but never lands at $0$ must eventually be contained in $P_f^\iota$. Therefore we can define the parabolic basin as
$$
\mathcal{B}_f = \bigcup f^{-n} P_f^\iota.
$$
The outgoing petal $P_f^o$ is backwards invariant, with backwards orbits converging to $0$.

On $P_f^\iota$ and $P_f^o$ one can define incoming and outgoing Fatou coordinates $\phi_f^\iota: P_f^\iota \rightarrow \mathbb C$ and $\phi_f^o: P_f^o \rightarrow \mathbb C$, solving the functional equations
$$
\phi_f^\iota \circ f(z) = \phi_f^\iota(z) + 1 \; \; \mathrm{and} \; \;  \phi_f^o \circ f(z) = \phi_f^o(z) + 1,
$$
where $\phi_f^\iota (P_f^\iota)$ contains a right half plane and $\phi_f^o (P_f^o)$ contains a left half plane. By the first functional equation the incoming Fatou coordinates can be uniquely extended to the attracting basin $\mathcal{B}_f$. On the other hand, the inverse of $\phi_f^o$, denoted by $\psi_f^o$, can be extended to the entire complex plane, still satisfying the functional equation
$$
f \circ \psi_f^o (Z) = \psi_f^o(z + 1).
$$
The fact that the exceptional set of $f$ is empty implies that $\psi_f^o : \mathbb C \rightarrow \mathbb C$ is surjective. We note that both incoming and outgoing Fatou coordinates are (on the corresponding petals) of the form $Z = -\frac{1}{z} + b \log(z) + o(1)$, where the coefficient $b$ vanishes when $a = 1$. This is one reason for working with maps $f$ of the form $f(z) = z + z^2 + z^3 + h.o.t.$.

Let us now consider small perturbations of the map $f$. For $\epsilon \in \mathbb C$ we write $f_\epsilon(z) = f(z) + \epsilon^2$, and consider the behavior as $\epsilon \rightarrow 0$. The most interesting behavior occurs when $\epsilon$ approaches $0$ tangent to the positive real axis.

\medskip

\noindent {\bf Lavaurs Theorem \cite{Lavaurs}.} \emph{Let $\epsilon_j \rightarrow 0$, $n_j \in \mathbb N$ and $\alpha \in \mathbb C$ satisfy
$$
n_j - \frac{\pi}{\epsilon_j} \rightarrow \alpha \; \; \mathrm{as} \; \; j \rightarrow \infty.
$$
Then
$$
f_{\epsilon_j}^{n_j} \rightarrow \mathcal{L}_f(\alpha) = \psi_f^o \circ \tau_{\alpha} \circ \phi_f^\iota,
$$
where $\tau_\alpha(Z) = Z+ \alpha$.
}

\medskip

The map $\mathcal{L}_f(\alpha)$ is called the \emph{Lavaurs map}, and $\alpha$ is called the \emph{phase}. In this paper we will only consider phase $\alpha = 0$, and write $\mathcal{L}_f$ instead of $\mathcal{L}_f(0)$.

\subsection{Propositions A and B}

	The construction of wandering domains in \cite{ABDPR} follows quickly from two key propositions, the aforementioned Propositions A and B. In this paper we will prove a variation to Proposition B, and a refinement to Proposition A, which we will both state here.

	Our main technical result is the following refinement of Proposition A. As before we write $P(z,w) = (f(z) + \frac{\pi^2}{4}w, g(w))$, with $f(z) = z + z^2 + z^3 + bz^4 + h.o.t.$, and $g(w) = w - w^2 + w^3 + h.o.t.$.

	\medskip
	
	\noindent {\bf Proposition A'} \emph{
		There exists a holomorphic function $h:\mathcal{B}_f \times \mathcal{B}_g \rightarrow \mathbb C$ such that
		$$
		P^{2n+1}(z, g^{n^2}(w)) = (\mathcal{L}_f(z), 0) + \left(\frac{h(z,w)}{n},0\right) + O\left(\frac{\log n}{n^2}\right),
		$$
		uniformly on compact subsets of $\mathcal{B}_f \times \mathcal{B}_g$. The function $h(z,w)$ is given by
		$$
		h(z,w) = \frac{\mathcal{L}_f^\prime(z)}{(\phi^\iota_f)^\prime(z)} \cdot \left(C + \phi_f^\iota(z) - \phi_g^\iota(w)\right),
		$$
		where the constant $C \in  \mathbb C$ depends on $b$.
	}	
	\medskip

Proposition A' will be proved in section \ref{section:ConvergenceEstimates}, {\black see  Theorem \ref{thm:lav}.}

	Proposition B in \cite{ABDPR} states that the Lavaurs map $\mathcal{L}_f$ of a polynomial $f(z) = z + z^2 + az^3 + O(z^4)$ has an attracting fixed point for suitable choices of the constant $a \in \mathbb C$. We recall very briefly the main idea in the proof of Proposition B: For $a = 1$ the ``horn map'' has a parabolic fixed point at infinity. By perturbing $a \simeq 1$, the parabolic fixed point bifurcates, and for appropriate perturbations this guarantees the existence of an attracting fixed point for the horn map, and thus also for the Lavaurs map.
	
	In this paper we will consider a more restrictive family of polynomials of the form $f(z) = z + z^2 + z^3 + O(z^4)$, which means that we cannot use the above bifurcation argument. Using a different line of reasoning, using small perturbations of a suitably chosen degree $7$ real polynomial, {\black we will prove a variation to Proposition B, namely Proposition B' below}. The proof of Proposition B' will be given in section \ref{section: proposition B'}.	

\medskip

{\black Before stating the proposition we recall that a fixed point $z_0 = \mathcal{L}_f(z_0)$ is said to be of Siegel type if $\lambda = \mathcal{L}_f^\prime(z_0) =e^{2\pi i \zeta}$, where $\zeta\in \mathbb{R}\backslash\mathbb{Q}$ is Diophantine, i.e. if there exist $c,r>0$ such that $|\lambda^n-1|\geq cn^{-r}$ for all integers $n>0$. Recall that neutral fixed points with Diophantine rotation numbers are always locally linearizable:}
	
	\begin{thm}[Siegel, \cite{Siegel}] Let $p(z) = e^{2\pi i\zeta} z + O(z^2)$ be a holomorphic germ. If $\zeta$ is Diophantine
		then there exist a neighborhood of the origin $\Omega_p$ and a biholomorphic map $\varphi:\Omega_p\rightarrow D_r(0)$ of the form $\varphi(z)=z+a_2z^2+O(z^3)$ satisfying
		$$
		\varphi(p(z))= e^{2\pi i \zeta} \varphi(z).
		$$
	\end{thm}

	\medskip

\noindent {\bf Proposition B'} \emph{There exist polynomials of the form $f(z) = z + z^2 + z^3 + O(z^4)$ for which the Lavaurs map $\mathcal{L}_f$ has a \emph{Siegel fixed point} {\black $z_0$}, with $\lambda = \mathcal{L}_f^\prime(z_0)$. Moreover we can guarantee that
\begin{equation}\label{non-degenerate index}
\frac{\mathcal{L}_f''(z_0)(\phi_f^\iota)'(z_0)}{\lambda(1-\lambda)}-(\phi_f^\iota)''(z_0) \neq 0.
\end{equation}
}	
	\medskip
	
	Condition \eqref{non-degenerate index} is necessary to guarantee the existence of wandering domains, see the discussion of the index $\kappa$ later in this section, and the discussion in subsection \ref{subsection:choice of index}.

	A more precise description of the derivatives $\lambda$ for which $p$ is locally linearizable was given by Brjuno \cite{Brjuno} and Yoccoz \cite{Yoccoz}. As we are only concerned with constructing examples of maps with wandering Fatou components, we find it convenient to work with the stronger Diophantine condition. Proposition B' will be proved in section \ref{section: proposition B'}.

\subsection{Perturbations of Siegel disks}

A key element in our study is the following question:

\medskip

\emph{Let  $f_1, f_2, \ldots $ be a sequence of holomorphic germs, converging locally uniformly to a holomorphic function $f$ having a Siegel fixed point at $0$. Under which conditions does there exist a trapping region?}

\medskip

By a trapping region we mean the existence of arbitrarily small neighborhoods $U, V$ of $0$ and $n_0 \in \mathbb N$ such that
$$
f_m \circ \cdots \circ f_n (z) \in V
$$
for all $z \in U$ and ${\black m \ge n} \ge n_0$. In other words, any orbit $(z_n)_{n\ge 0}$ that intersects $U$ for sufficiently large $n$ will afterwards be contained in a small neighborhood of the origin. Note that this in particular guarantees normality of the sequence of compositions $f_m \circ \ldots \circ f_0$ in a neighborhood of $z_0$, which is the reason for our interest in trapping regions.

We are particularly interested in the case where the differences $f_n - f$ are not absolutely summable, i.e. when
$$
\sum_{n \ge n_0} \|f_n - f\|_U = \infty
$$
for any $n_0$ and $U$. In this situation one generally does not expect a trapping region. However, motivated by Proposition A', we will assume that $f_n - f$ is roughly of size $\frac{1}{n}$, and converges to zero along some real direction. More precisely, we assume that
\begin{equation}\label{function f_n}
f_n(z) - f(z) = \frac{h(z)}{n} + O(\frac{1}{z^{1+\epsilon}}),
\end{equation}

where $h$ is a holomorphic germ, defined in a neighborhood of the origin.

\begin{thm}\label{thm:siegel-simple} There exists an \emph{index} $\kappa$, a rational expression in the coefficients of $f$ and $h$, such that the following holds:
\begin{enumerate}
\item If $\mathrm{Re}(\kappa)= 0$, then there is a trapping region, and all limit maps have rank $1$.
\item If $\mathrm{Re}(\kappa)< 0$, then there is a trapping region, and all orbits
%the sequence $\{f_{(n,j)}|_U\}_n$
 converge uniformly to the origin.
\item If $\mathrm{Re}(\kappa)> 0$, then there is no trapping region. In fact, there can be at most one orbit that remains in a sufficiently small neighborhood of the origin.

%for every $j$ the family $\{f_{(n,j)|_U}\}_n$ is unbounded for any neighborhood $U= U(0)$.
\end{enumerate}
\end{thm}

Theorem \ref{thm:siegel-simple} holds under more general assumptions regarding the convergence towards the limit map, but the above statement is sufficient for our purposes. An example of a more general statement is given in Remark \ref{remark:generalstatement}. An explicit formula for the index $\kappa$ is given in section \ref{section:perturbations}, which contains the proof of Theorem \ref{thm:siegel-simple}.

\begin{rem}{\black The non-autonomous dynamics of the functions $f_n$ satisfying \eqref{function f_n} is closely related to the autonomous dynamics of the quasi-parabolic map
$$
F(z,w)=(f(z)+wh(z)+O(w^2),w-w^2+O(w^2)).
$$
The case $\mathrm{Re}(\kappa) < 0$ in Theorem \ref{thm:siegel-simple}  corresponds to $F$ being dynamically separating and parabolically attracting, using the terminology of \cite{BZ}, hence by \cite[Corollary 6.3]{BZ} the map $F$ has a connected basin of attraction at the origin. In particular this implies the existence of a trapping region for the sequence $(f_n)$.}
\end{rem}

\subsection{Parabolic Curves}
An important idea in the proof of Lavaurs Theorem is that in a sufficiently small neighborhood of the origin, the function $f_{\epsilon} = f + \epsilon^2$ can be interpreted as a near-translation in the ``almost Fatou coordinates'': functions that converge to the ingoing and outgoing Fatou coordinates as $\epsilon \rightarrow \infty$. This idea is especially apparent in the treatment given in \cite{BSU}. The almost Fatou coordinates are defined using the pair of fixed points $\zeta_{\pm}(\epsilon)$ ``splitting'' from the parabolic fixed point.

When iterating two-dimensional skew products $P(z,w) = (f_w(z), g(w))$ it does not make sense to base the almost Fatou coordinates on the pair of fixed points of the maps $f_w(z) = f(z) + \frac{\pi^2}{4}w$, as the parameter $w$ changes after every iteration of $P$. Instead, the natural idea would be to base these coordinates on a pair of invariant curves $\{z = \zeta_{\pm}(w)\}$, so-called \emph{parabolic curves}, defined over a forward invariant parabolic petal in the $w$-plane. The invariance of these parabolic curves is equivalent to the functional equations
$$
\zeta_{\pm}(g(w)) = f_w(\zeta_{\pm}(w)).
$$

In  \cite{ABDPR}, it is asked whether such parabolic curves exists. Instead, in \cite{ABDPR} it was shown that there exist \emph{almost parabolic curves}, approximate solutions to the above functional equation with explicit error estimates. The proof of Proposition A relies to a great extent on these almost parabolic curves, and the fact that these are not exact solutions causes significant extra work.

In the recent paper \cite{LH-R} by Lopez-Hernanz and Rosas it is shown that the parabolic curves indeed exist, in fact, the authors prove the existence of parabolic curves for any characteristic direction for diffeomorphisms in two complex dimensions. However, to be used in the proof of Proposition A, it is necessary to also obtain control over the domain of definition of the two parabolic curves. The result from \cite{LH-R} does not give the needed control.

In section \ref{section:parabolic}, Proposition \ref{prop:curve}, we give an alternative proof of the existence of parabolic curves, with control over the domains of definition. The availability of these parabolic curves forms an important ingredient in the proof of Proposition A'. The method of proof is a variation to the well known graph transform method, and can likely be used to prove the existence of parabolic curves in greater generality.

\subsection{Wandering domains} Let us conclude this section by proving how Propositions A' and B' together imply the existence of wandering Fatou components. As before we let
$$
P(z,w) = (f(z) + \frac{\pi^2}{4} w, g(w)),
$$
where $g(w) = w - w^2 + w^3 + h.o.t.$ and the function $f(z) = z + z^2 + z^3 + h.o.t.$ is chosen such that $\mathcal{L}_f$ has a neutral fixed point $z_0$ with Diophantine rotation number. The existence of such $f$ is given by Proposition B'.

Proposition A' states that
$$
P^{2n+1}(z, g^{n^2}(w)) = (\mathcal{L}_f(z), 0) + \left(\frac{h(z,w)}{n},0\right) + O\left(\frac{\log n}{n^2}\right),
$$
uniformly on compact subsets of $\mathcal{B}_f \times \mathcal{B}_g$.

Recall from Proposition A' that the function $h(z,w)$ is given by
$$
h(z,w) = \frac{\mathcal{L}_{\black f}^\prime(z)}{(\phi^\iota_f)^\prime(z)} \cdot \left(C + \phi_f^\iota(z) - \phi_g^\iota(w)\right),
$$
from which it follows directly that the index $\kappa$ depends affinely on $\phi_g^\iota(w)$, although it is conceivable that the multiplicative constant in this dependence vanishes.

As will be explained in detail in subsection \ref{subsection:choice of index}, the index $\kappa$ is independent from $w$ if and only if, {\black denoting the fixed point of $\mathcal{L}_f$ again by $z_0$}, writing
\begin{equation}\label{equation: repeatedfromlater}
{\black \frac{\mathcal{L}_f''(z_0)(\phi^\iota_f)'(z_0)}{\lambda(1-\lambda)}-(\phi^\iota_f)''(z_0) = 0,}
\end{equation}
in which case $\kappa$ is constantly equal to $+1$. The second statement in Proposition B' therefore implies that $f$ can be chosen in order to obtain an inequality in \eqref{equation: repeatedfromlater}, which implies that the affine dependence of $\kappa$ on $\phi_g^\iota(w)$ is non-constant.

It follows that there exists an open subset of $\mathcal{B}_g$ where the $w$-values are such that $\mathrm{Re}(\kappa)$ is strictly negative. Let $D_2 \subset \mathcal{B}_g$ be a small disk centered contained in this open subset, so that $\mathrm{Re}(\kappa)$ is negative for all $w \in D_2$.

Let $D_1$ be a small disk centered at ${\black z_0}$, the Siegel type fixed point of $\mathcal{L}_f$. We claim that for $n \in \mathbb N$ large enough, the open set $D_1 \times g^{n^2}(D_2)$ is contained in a wandering Fatou component.

Indeed, it follows from Proposition A' that the non-autonomous one-dimensional system given by compositions of the maps $z \mapsto \pi_z \circ P^{2n+1}(z, g^{n^2}(w))$ satisfy case (2) of Theorem \ref{thm:siegel-simple}, where $\pi_z$ is the projection onto the $z$-coordinate. Thus Theorem \ref{thm:siegel-simple} implies that
$$
P^{m^2 -n^2}(z, w) \rightarrow ({\black z_0}, 0),
$$
uniformly for all $(z,w) \in D_1 \times g^{n^2}(D_2)$. {\black The remainder of the proof follows the argument from \cite{ABDPR}}. Since the complement of the escape locus of $P$ is bounded, it follows that the entire orbits $P^m(z,w)$ must remain uniformly bounded, which implies normality of $(P^m)$ on $D_1 \times g^{n^2}(D_2)$, which is therefore contained in a Fatou component, say $U$. The fact that on an open subset of $U$ the subsequence $P^{m^2 -n^2}$ converges to the constant $({\black z_0},0)$ implies convergence of this subsequence to $({\black z_0},0)$ on all of $U$, since limit maps of convergent subsequences are holomorphic. But if $U$ was periodic or preperiodic, the limit set would have been periodic. The point $({\black z_0},0)$ is however not periodic: its orbit converges to $(0,0)$. Thus $U$ is wandering, which completes the proof.

\begin{rem}
From the above discussion we can conclude that all possible limit maps of convergent subsequence $P^{n_j}|_U$ are points. In fact these points form (the closure of) a bi-infinite orbit of $({\black z_0},0)$, converging to $(0,0)$ in both backward in forward time.

We note however that there are fibers $\{w=w_0\}$, with $w_0 \in \mathcal{B}_g$ for which $\mathrm{Re}(\kappa) = 0$. Let $D_1$ again be a sufficiently small disk centered at ${\black z_0}$, the Siegel type fixed point of $\mathcal{L}_f$. Theorem A' together with case (1) of Theorem \ref{thm:siegel-simple} implies that for sufficiently large $n$ the disk $D_1 \times \{g^{n^2}(w_0)\}$ is a Fatou disk for $P$, i.e. the restriction of the iterates $P^n$ to the disk form a normal family. For this Fatou disk the sequence of iterates $P^{m^2 - n^2}$ converges to a rank $1$ limit map, whose image is a holomorphic disk containing $({\black z_0},0)$. All the limit sets together form (the closure of) a bi-infinite sequence of disks, converging in backward and forward time to the point $(0,0)$.
\end{rem}

\section{Perturbations of Siegel disks}\label{section:perturbations}

\subsection{Notation} The following conventions will be used throughout this section:
\begin{enumerate}[label=(\roman*)]
\item Given a holomorphic function $f$, we will write $\hat{f}$ for the non-linear part of $f$.

\item  For a sequence of constants $\lambda_n \in \mathbb C$ we will write
$$
\lambda_{n,m}  = \prod_{j=m+1}^n \lambda_j \quad \mathrm{and} \quad \lambda(n) = \lambda_{n,0} = \prod_{j=1}^n \lambda_j,
$$
and similarly for a sequence of functions $(f_n)$
$$
f_{n,m} = f_n \circ \cdots \circ f_{m+1}.
$$
\item Given two sequences of holomorphic functions $(f_n)$ and $(g_n)$ defined on some uniform neighborhood of the origin, we will write $f_n \asymp g_n$ if the norms of the sequence of differences $(f_n-g_n)$ is summable on some uniform neighborhood of the origin.

\end{enumerate}
\medskip

\subsection{Preparation}

{\black In this section we introduce non-autonomous analogies of attracting, repelling, and locally linearizable indifferent fixed points, and make a few initial observations. In the next subsection we introduce the index $\kappa$, and show that the local behavior of the non-autonomous systems we consider can be deduced from the real part of the index.}

\begin{defi}\normalfont
Two sequences of functions $(f_n)$ and $(g_n)$ are said to be \emph{non-autonomously conjugate} if there exist a uniformly bounded sequence of local coordinate changes $(\psi_n)_{n \ge n_0}$, all defined in a uniform neighborhood of the origin, satisfying
$$f_n\circ \psi_n = \psi_{n+1}\circ g_n$$
for all $n\geq n_0$.
\end{defi}

\begin{defi} \normalfont
A sequence of functions $(f_n)$ %of the form
%$$
%f_n(z) = \lambda_n z + \hat{f}_n(z)
%$$
is said to be non-autonomously \emph{linearizable} if there exists a sequence $(\lambda_n)_{n \ge n_0}$ in $\mathbb C\setminus \{0\}$ and a sequence of coordinate changes $(\psi_n)_{n \ge n_0}$, defined and uniformly bounded in a uniform neighborhood of the origin, and with derivative $\psi_n^\prime(0)$ uniformly bounded away from zero, so that
$$
f_n\circ \psi_n(z) = \psi_{n+1}(\lambda_n \cdot z)
$$
for all $n \ge n_0$. If the sequence $|\lambda(n)|$ is bounded, both from above and away from $0$, then we say that $(f_n)$ is  \emph{rotationally linearizable}.
\end{defi}

\begin{defi} \normalfont
A sequence of functions $(f_n)$ is said to be \emph{collapsing} if there is a neighborhood of the origin $U$ and an $n_0 \in \mathbb N$ such that $f_{n,m} \rightarrow 0$ on $U$ {\black as $n \rightarrow \infty$} for any $m \ge n_0$.
\end{defi}

An example of a collapsing sequence is given by a sequence of functions $f_n$ converging to a function $f$ with an attracting fixed point at the origin.

\begin{defi} \normalfont
We say that sequence $(f_n)$ is \emph{expulsive} if there exists $r>0$ such that for every $m\geq 0$ there exists at most one exceptional point $\hat{z}$ such that for every $z\in D_r(0) \setminus \{\hat{z}\}$ there exist $n>m$ for which $f_{n,m}(z)\notin D_r(0)$. {\black Here $D_r(0)$ denotes the disk of radius $r$ centered at the origin}
\end{defi}

An example of an expulsive sequence can be obtained by considering a sequence of maps $(f_n)$ converging locally uniformly to a map with a repelling fixed point. Since $f$ maps a small disk around the origin to a strictly larger holomorphic disk, the same holds for sufficiently small perturbations. A nested sequence argument shows that, starting at a sufficiently large time $n_0$, there is a unique orbit which remains in the small disk.

%\begin{rem}
%More generally we say that the sequence $(f_n)_{n>0}$ is  collapsing/ expulsive/rotationally linearizable if there exist $n_0>0$ for which the sequence $(f_n)_{n>n_0}$ is collapsing/expulsive/rotationally linearizable in the above sense.
%\end{rem}

\begin{lem}\label{normal->rotation}
Consider a sequence $(f_n)$ of univalent holomorphic functions, defined in a uniform neighborhood of the origin. Suppose the compositions $f_{n,0}$ are all defined in a possibly smaller neighborhood of the origin, and form a normal family. Then the sequence $(f_n)$ is either rotationally linearizable, or there exist subsequences $(n_j)$ for which $f_{n_j,0}$ converges to a constant.
\end{lem}
\begin{proof}
By normality the orbit $f_{n,0}(0)$ stays bounded. By non-autonomously conjugating with a sequence of translations we may therefore assume that $f_n(0) = 0$ for all $n$. Note that normality is preserved under non-autonomous conjugation by bounded translations.

Write $\lambda_n = f_n^\prime(0)$. Normality implies that $|\lambda(n)|$ is bounded from above. The functions
$$
\psi_{n+1}(z) := f_{n,0} (\lambda(n)^{-1}\cdot z)
$$
are tangent to the identity, and they satisfy the functional equation
$$
f \circ \psi_n(z) = \psi_{n+1}(\lambda \cdot z).
$$
If the sequence $|\lambda(n)|$ is bounded away from the origin then the maps $\psi_n$ are uniformly bounded, and the sequence $(f(n))$ is rotationally linearizable. Suppose that the sequence $\lambda(n)$ is not bounded from below, in which case there is a subsequence $\lambda(n_j)$ converging to $0$. By Hurwitz Theorem the sequence of maps $f_{n_j,0}$ converges to a constant.
\end{proof}

\begin{lem}\label{lemma:reduction}
If the sequence $(f_n)$ is rotationally linearizable, and $(\zeta_n)$ is a sequence of absolutely summable holomorphic functions, i.e.
$$
\sum \|\zeta_n\|_{D_r{\black (0)}} < \infty
$$
for some $r>0$, then the sequence $(f_n + \zeta_n)$ is also rotationally linearizable.
\end{lem}
\begin{proof}
Write $g_n = f_n + \zeta_n$. We consider the errors due to the perturbations in linearization coordinates, i.e.
$$
\psi_{n+1}^{-1} \circ g_n \circ \psi_n(z) - \psi_{n+1}^{-1} \circ f_n \circ \psi_n(z) = \psi_{n+1}^{-1} \circ g_n \circ \psi_n(z) - \lambda_n \cdot z.
$$
By definition of the non-autonomous linearization, it follows that after restricting to a smaller neighborhood of the origin the derivatives of the maps $\psi_n$  and their inverses are uniformly bounded.  It follows that the above errors are also absolutely summable, which guarantees normality of the sequence $\psi_{n+1}^{-1} \circ g_{n,0}$ in a small neighborhood of the origin, and hence normality of the sequence $g_{n,0}$. It follows from Lemma~\ref{normal->rotation} that $(f_n + \zeta_n)$ is either rotationally linearizable, or has subsequences converging to the origin. It follows from the summability of the errors that the latter is impossible.
\end{proof}

\subsection{Introduction of the index}

Let $f(z)=\lambda z+b_2z^2+O(z^3)$ be a holomorphic function with $\lambda=e^{2\pi i\zeta}$ and $\zeta\in\mathbb{R}\backslash\mathbb{{\black Q}}$ Diophantine.  Let $h(z)=c_{0}+c_{1}z+O(z^2)$ be a holomorphic function defined in a neighborhood of the origin. Let $(\zeta_n(z))$  be a sequence of holomorphic functions that is defined and absolutely summable on some uniform neighborhood of the origin.  We consider the non-autonomous dynamical system given by compositions of the maps
$$
f_n(z)=f(z)+ \frac{1}{n} h(z)+\zeta_n(z).
$$

We introduce the index $\kappa$, depending rationally on the two-jet of $f$ at the origin, and the one-jet of $h$ at the origin, by
\begin{equation}\label{kappa}
\kappa:=\frac{2b_2c_{0}}{\lambda(1-\lambda)}+\frac{c_{1}}{\lambda}.
\end{equation}
{\black We claim that the index $\kappa$ is invariant under local autonomous changes of coordinates, i.e. when all the maps $f_n$ are conjugated by a single analytic transformation. One easily observes that the index is invariant under affine changes of coordinates, and is unaffected by terms of order $3$ and higher. It is therefore sufficient to only consider local changes of the form $z \mapsto z + \alpha z^2$. It is clear that $\lambda$ and $c_0$ are unaffected by such a coordinate change, while computation shows that $b_2$ is replaced by $b_2 + \alpha \lambda - \alpha \lambda^2$ and $c_1$ is replaced by $c_1 - \alpha \lambda c_0$. Indeed, $\kappa$ is invariant under these changes.}

\medskip

Since $\zeta$ is Diophantine the function $f$ is linearizable. Let us write $\phi(z) = z + h.o.t.$ for the linearization map of $f$, i.e. $f\circ\phi(z)=\phi(\lambda z)$.

We define
$$
\theta_n(z):=z+ \frac{1}{n}\frac{c_{0}}{1-\lambda}.
$$

\begin{lem}
With the above definitions we can write
\begin{equation}\label{form}
\begin{aligned}
\f_n:=& \phi^{-1}\circ\theta_{n+1}^{-1}\circ f_n\circ \theta_n\circ\phi \\
=& \lambda \cdot e^{\frac{\kappa}{n}} \cdot z + \frac{1}{n} \sum_{{\black k}=2}^{\infty}d_{\black k}z^{\black k} +\xi_n(z),
\end{aligned}
\end{equation}
where $(\xi_n)$  is a sequence of holomorphic functions that are defined and whose norms are summable on a uniform neighborhood of the origin.

\end{lem}
\begin{proof}
First observe that
\begin{align*}
\theta_{n+1}^{-1}\circ f_n\circ \theta_n &\asymp f(z+ \frac{1}{n}\frac{c_{0}}{1-\lambda})+ \frac{1}{n}h(z+ \frac{1}{n}\frac{c_{0}}{1-\lambda})-\frac{1}{n+1}\frac{c_{0}}{1-\lambda}\\
&\asymp f(z)+f'(z) \frac{1}{n}\frac{c_{0}}{1-\lambda}+ \frac{1}{n}h(z)-\frac{1}{n+1}\frac{c_{0}}{1-\lambda}.
\end{align*}
Using the power series expansions of $f^\prime$ and $h$ we can therefore write
\begin{align*}
\theta_{n+1}^{-1}\circ f_n\circ \theta_n &\asymp f(z)+ (\frac{1}{n}\frac{\lambda  c_{0}}{1-\lambda}+\frac{1}{n}c_0-\frac{1}{n+1}\frac{c_{0}}{1-\lambda})+\frac{1}{n}(\frac{2b_2c_{0}}{1-\lambda}+c_1)z+ \frac{1}{n}\sum_{k=2}^{\infty}\beta_kz^k\\
&\asymp f(z)+ \frac{c_{0}}{1-\lambda}(\frac{1}{n}-\frac{1}{n+1})+ \frac{1}{n}\lambda\kappa z+  \frac{1}{n}\sum_{k=2}^{\infty}\beta_kz^k\\
&\asymp f(z)+\frac{1}{n}\lambda\kappa z+ \frac{1}{n}\sum_{k=2}^{\infty}\beta_kz^k.
\end{align*}
It follows that

\begin{align*}
%\f_n&\asymp\phi^{-1}\left(f(\phi(z))+\frac{1}{n}\lambda\kappa \phi(z)+ \frac{1}{n}\sum_{k=2}^{\infty}\beta_k\phi(z)^k+O(n^{-2})\right) \\
\f_n&\asymp\phi^{-1}(f(\phi(z)))+(\phi^{-1})'(f(\phi(z)))\left(\frac{1}{n}\lambda\kappa \phi(z)+ \frac{1}{n}\sum_{k=2}^{\infty}\beta_k\phi(z)^k\right)\\
&\asymp\lambda(1+\frac{\kappa}{n})z + \frac{1}{n}\sum_{k=2}^{\infty}d_kz^k\\
&\asymp\lambda e^{\frac{\kappa}{n}}z+ \frac{1}{n}\sum_{k=2}^{\infty}d_kz^k.
\end{align*}
For the last equality we used that $1+\frac{\kappa}{n} = e^{\frac{\kappa}{n}} +O(\frac{1}{n^2})$.
\end{proof}

\begin{coro}
If $\mathrm{Re}(\kappa) < 0$ the sequence $\f_n$ is collapsing.
\end{coro}
\begin{proof}
Observe that $\f_n'(z) \asymp \lambda e^{\frac{\kappa}{n}}+O(\frac{z}{n})$ and
note that there is a small disk $D_r(0)$ such that for $n$ sufficiently large
$$
\|\f_n^\prime\|_{D_r(0)}< e^{\frac{\re(\kappa)}{2n}},
$$
and thus
\begin{equation}\label{go away}
|\f_n(z) - \f_n(w)| \le e^{\frac{\re(\kappa)}{2n}} |z - w|.
\end{equation}
Since $\re(\kappa)<0$ it follows that $\prod_{n \ge 1} e^{\frac{\re(\kappa)}{2n}} = 0$.

Let us write
$$
\varphi_n(z) = \lambda \cdot e^{\frac{\kappa}{n}} \cdot z+\hat{\f}_n(z),
$$
i.e. dropping the term $\xi_n$ from $\f_n$.  By decreasing the radius $r$ if necessary we can choose $m_0$ such that
$$
\sum_{j \ge m_0} \|\xi_j\|_{D_r(0)} < \frac{r}{2}.
$$
By increasing $m_0$ if necessary we can also guarantee that $\varphi_{n,m}(z) \in D_{r/2}(0)$ for all $z \in D_{r/4}(0)$ and $m \ge m_0$. Using \eqref{go away} it follows by induction on $n$ that whenever $z \in D_{r/4}(0)$ and $m \ge m_0$ then
$$
\| \f_{n,m}(z) - \varphi_{n,m}(z) \| \le \sum_{j = m}^n (\prod_{k=j+1}^n e^{\frac{\re(\kappa)}{2k}}) \|\xi_j\|_{D_r(0)}.
$$
Indeed, the inequality is trivially satisfied for $n =m$, and assuming the inequality holds for some $n \ge m$ implies
$$
\begin{aligned}
\| \f_{n+1,m}(z) - \varphi_{n+1,m}(z) \| = & \| \f_{n+1}\circ \f_{n,m}(z) - \f_{n+1} \circ \varphi_{n,m}(z) + \xi_{n+1}(\varphi_{n,m}(z))\|\\
\le & \sum_{j = m}^{n+1} (\prod_{k=j+1}^{n+1} e^{\frac{\re(\kappa)}{2k}}) \|\xi_j\|_{D_r(0)}.
\end{aligned}
$$
Note that $\sum_{j=m}^n (\prod_{k=j+1}^n e^{\frac{\re(\kappa)}{2k}}) \|\xi_j\|_{D_r(0)} \rightarrow 0$ as $n \rightarrow \infty$, hence the fact that the sequence $(\varphi_n)$ collapses implies that the sequence $(\f_n)$ collapses as well.
\end{proof}

Since the sequence $(\f_n)$ collapses, it follows immediately that the sequence $(f_n)$ collapses as well, concluding the case $\mathrm{Re}(\kappa) < 0$.

\begin{coro}
If $\mathrm{Re}(\kappa) > 0$ the sequence $\f_n$ is expulsive.
\end{coro}
\begin{proof}
Note that there are $r,n_0>0$, such that for every $z,w\in D_r(0)$ and every $n>n_0$ we have
$$
|\f_n(z)-\f_n(w)|=|z-w|\cdot|e^{\frac{\kappa}{n}}+\frac{1}{n}O(z,w)|> e^{\frac{\re(\kappa)}{2n}}|z-w|.
$$
Expulsion of all but one orbit follows immediately.
\end{proof}
Again it follows that $(f_n)$ is expulsive, completing the case $\mathrm{Re}(\kappa) > 0$.

\subsection{Rotationally linearizable case ($\re(\kappa)=0$)}

%We will prove Theorem \ref{thm:siegel} for maps of the form (\ref{form}),
%$$
%\f_n=\lambda \cdot e^{\kappa \cdot x_{n}}\cdot z+\hat{\f}_n(z)+\zeta_n(z).
%$$

Let us define
{\black
$$
L_n(z)=e^{\kappa \log n } \cdot z.
$$}
We obtain
\begin{align*}
g_n&=L^{-1}_{n+1}\circ \f_n\circ L_n\\
&=\lambda z + {\black e^{-\kappa \log(n+1) }} \frac{1}{n} \sum_{\ell=2}^{\infty} d_\ell {\black e^{\kappa \ell \log n}} z^\ell+ L^{-1}_{n+1}\circ \xi_n\circ L_n\\
&\asymp\lambda z + \frac{1}{n}\sum_{\ell=2}^{\infty} d_\ell e^{\kappa(\ell-1)\log n} z^\ell.
\end{align*}

%Let us define
%$$
%L_n(z)=e^{\kappa \sum_{k=1}^{n-1} \frac{1}{k}} \cdot z.
%$$
%We obtain
%\begin{align*}
%g_n&=L^{-1}_{n+1}\circ \f_n\circ L_n\\
%&=\lambda z + e^{-\kappa \sum_{k=1}^{n} \frac{1}{k}} \frac{1}{n} \sum_{\ell=2}^{\infty} d_\ell e^{\kappa \ell \sum_{k=1}^{n-1} \frac{1}{k}}z^\ell+ L^{-1}_{n+1}\circ \xi_n\circ L_n\\
%&\asymp\lambda z + \frac{1}{n}\sum_{\ell=2}^{\infty} d_\ell e^{\kappa(\ell-1)\log n} z^\ell,
%\end{align*}
%where the last equality used that $\sum_{k=1}^{n-1}\frac{1}{k}=\log n +O(\frac{1}{n^2})$.

Since $\mathrm{Re}(\kappa) = 0$, the maps $L_n$ are rotations, hence it is sufficient to prove that the sequence $(g_n)$ is rotationally linearizable.

By Lemma \ref{lemma:reduction} we may ignore the absolutely summable part of $g_n$, hence with slight abuse of notation we may assume that
$$
g_n=\lambda z + \frac{1}{n}\sum_{\ell=2}^{\infty} d_\ell e^{\kappa (\ell-1)\log n} z^\ell.
$$

Recall that $\lambda = e^{2\pi i \zeta}$, where $\zeta$ is Diophantine.

\begin{lem}\label{lemma:lambda} There exist constants $C,r>0$ such that for every integer $\ell\geq1$ and for every $0<m<N$ we have
$$
\left|\sum_{j=m}^{N}\lambda^{\ell j}\right|<C\ell^r.
$$
\end{lem}
\begin{proof}

Since $\zeta$ is assumed to be Diophantine, there exist $c,r>0$ such that $|\lambda^n-1|\geq cn^{-r}$ for all $n$. This gives the bound
$$
\left|\sum_{j=m}^{N}\lambda^{\ell j}\right|  = \left|\sum_{j=m}^{N}\frac{\lambda^{\ell (j+1)}-\lambda^{\ell j}}{\lambda^{\ell}-1}\right|=
\left|\frac{1}{\lambda^{\ell}-1}\sum_{j=m}^{N}(\lambda^{\ell(j+1)}-\lambda^{\ell j})\right|<\left|\frac{2}{\lambda^\ell-1}\right|<C \ell^r.
$$
\end{proof}

\begin{lem}There exist $\tilde{C},r>0$ such that for all integers $n,\ell>0$ we have
$$
\left|\sum_{k=n}^{\infty} \frac{e^{\kappa \ell \log k}}{k} \lambda^{k\ell}\right|<\frac{\tilde{C}\ell^{r+1}}{n}
$$
\end{lem}
\begin{proof}
Summation by parts gives
\begin{equation*}
\begin{aligned}
\sum_{k=n}^{N} \frac{e^{\kappa \ell \log k}}{k} \lambda^{k {\black \ell}}&= \frac{e^{\kappa \ell \log N}}{N}\sum_{k=n}^{N}\lambda^{k\ell }-\sum_{k=n}^{N-1}\left(\frac{e^{\kappa \ell \log (k+1)}}{k+1}-\frac{e^{\kappa \ell \log k}}{k}\right)\sum_{j=n}^{k}\lambda^{j \ell}\\
&=\frac{e^{\kappa \ell \log N}}{N}\sum_{k=n}^{N}\lambda^{k\ell}-\sum_{k=n}^{N-1}e^{\kappa \ell \log k}\left(\frac{1+\frac{\kappa l}{k}+O(\frac{1}{k^2})}{k+1}-\frac{1}{k}\right)\sum_{j=n}^{k}\lambda^{j \ell}.
\end{aligned}
\end{equation*}
Observe that $\frac{1+\frac{\kappa \ell }{k}+O(\frac{1}{k^2})}{k+1}-\frac{1}{k} =O(\frac{1}{k^2})$ is absolutely summable, hence using Lemma \ref{lemma:lambda} we obtain
\begin{equation*}
\begin{aligned}
\left|\sum_{k=n}^{N} \frac{e^{\kappa \ell  \log k}}{k} \lambda^{k\ell }\right|&< \frac{1}{N}\left|\sum_{k=n}^{N}\lambda^{k\ell }\right|+\sum_{k=n}^{N-1}\left|\frac{1+\frac{\kappa \ell }{k}+O(\frac{1}{k^2})}{k+1}-\frac{1}{k}\right|\left|\sum_{j=n}^{k}\lambda^{j\ell }\right| \\
&< \frac{C\ell^r}{N}+C\ell^r\sum_{k=n}^{N-1}\left|\frac{\ell \kappa-1}{k(k+1)}+O(\frac{1}{k^3})\right|\\
&< \frac{\tilde{C}\ell^{r+1}}{n}
\end{aligned}
\end{equation*}
\end{proof}

Let us introduce one more change of coordinates
$$
S_{n+1}(z)=z-\lambda^{-1}\sum_{\ell=2}^{\infty}\lambda^{(n+1)(1-\ell)} d_\ell z^\ell\sum_{k=n+1}^{\infty} \frac{e^{\kappa (\ell-1)\log k}}{k} \lambda^{k(\ell-1)}
$$
\begin{lem}\label{crucial-lemma}
Writing $S_n(z)=z+\hat{S}_n(z)$ we obtain
$$
\hat{S}_{n+1}(\lambda z)=\lambda \hat{S}_n(z)+ \hat{g}_n(z).
$$
\end{lem}
\begin{proof} Computing $\hat{S}_{n+1}(\lambda z)-\hat{g}_n(z)$ gives
\begin{align*}
&-\lambda^{-1}\sum_{\ell=2}^{\infty}\lambda^{(n+1)(1-\ell)} d_\ell \lambda^\ell z^\ell \sum_{k=n+1}^{\infty} \frac{e^{\kappa (\ell-1)\log k}}{k} \lambda^{k(\ell-1)}- \frac{1}{n}\sum_{\ell=2}^{\infty} e^{\kappa (\ell-1)\log n} d_\ell z^\ell\\
&=-\sum_{\ell=2}^{\infty}\lambda^{n(1-\ell)} d_\ell z^\ell\sum_{k=n+1}^{\infty} \frac{e^{\kappa (\ell-1)\log k}}{k} \lambda^{k(\ell-1)}- \sum_{\ell=2}^{\infty} \frac{e^{\kappa (\ell-1)\log n}}{n}\lambda^{n(\ell-1)}\lambda^{n(1-\ell)}d_\ell  z^\ell\\
&=-\sum_{\ell=2}^{\infty}\lambda^{n(1-\ell)} d_\ell z^\ell\sum_{k=n}^{\infty} \frac{e^{\kappa (\ell-1)\log k}}{k} \lambda^{k(\ell-1)} \;\; \quad = \lambda \hat{S}_n(z).
\end{align*}
\end{proof}

\begin{lem} The maps $S_n$ satisfy $S_{n}=z+O(\frac{1}{n})$, with uniform bounds.
\end{lem}
\begin{proof}
\begin{equation*}
\begin{aligned}
\left|\hat{S}_n(z)\right|&= \left|\lambda^{-1}\sum_{\ell=2}^{\infty}\lambda^{n(1-\ell)} d_\ell z^\ell \sum_{k=n}^{\infty} \frac{e^{\kappa (\ell-1)\log k}}{k} \lambda^{k(\ell-1)}\right|\\
&<\frac{\tilde{C}}{n}\sum_{\ell=2}^{\infty}|d_\ell z^\ell|(\ell-1)^{r+1}.
\end{aligned}
\end{equation*}
\end{proof}

Let us define
$$
h_n:=S^{-1}_{n+1}\circ g_n \circ S_n.
$$
\begin{lem} The maps $h_n$ are of the form
$$
h_n= \lambda z + O(n^{-2}).
$$
\end{lem}
\begin{proof}
The definition of $h_n$ immediately gives that $h_n(z) = \lambda z + O(\frac{1}{n})$.
$$
g_n \circ S_n = S_{n+1} \circ h_n
$$
and thus
$$
\lambda z + \lambda \hat{S}_n(z)+ \hat{g}_n(z +\hat{S}_n) = \lambda z+ \hat{h}_n(z)+ \hat{S}_{n+1}(\lambda z+\hat{h}_n),
$$
which gives
$$
\lambda \hat{S}_n(z)+ \hat{g}_n(z)+ \hat{g}_n'(z)\hat{S}_n(z) +O(\hat{S}_n^2) = \hat{h}_n(z)+ \hat{S}_{n+1}(\lambda z)+ \hat{S}'_{n+1}(\lambda z)\hat{h}_n+ O(\hat{h}^2_n).
$$

Hence by Lemma~\ref{crucial-lemma} we obtain
$$
\hat{g}_n'(z)\hat{S}_n(z) +O(\hat{S}_n^2) = \hat{h}_n(z)(1+ \hat{S}'_{n+1}(\lambda z))+ O(\hat{h}^2_n).
$$
Since $\hat{g}_n=O(\frac{1}{n})$ and $\hat{S}_n =O(\frac{1}{n})$ we get
$$
\hat{h}_n(z)(1+ \hat{S}'_{n+1}(\lambda z))+ O(\hat{h}^2_n) =O(\frac{1}{n^2}).
$$
Since $h_n(z) = \lambda z + O(\frac{1}{n})$ it follows that $\hat{h}_n(z)=O(\frac{1}{n^2})$.

\end{proof}

Lemma \ref{lemma:reduction} implies that the sequence $(h_n)$ is rotationally linearizable, hence the same holds for $(g_n)$, $(\f_n)$ and finally $(f_n)$, which completes the proof of Theorem \ref{thm:siegel-simple}.

\begin{rem}\label{remark:generalstatement} The proof of Theorem \ref{thm:siegel-simple} also works for more general perturbations, for example
 $$
f_n(z)\asymp f(z)+\frac{1}{n}h_1(z) +\frac{\log n}{n}h_2(z),
$$
 {\black where $h_1$ and $h_2$ are holomorphic around the origin.
 In this case we have two indexes $\kappa_j$, $j\in\{1,2\}$, that can be computed using the equation \eqref{kappa}, where constants $c_0$ and $c_1$ are the coefficients of the linear part of the Taylor series of $h_j$ at the origin.} The following is a general version of Theorem \ref{thm:siegel-simple}.
\begin{enumerate}
\item If  $\mathrm{Re}(\kappa_2)>0$  then the sequence $(f_n)$ is expulsive.
\item If  $\mathrm{Re}(\kappa_2)<0$  then the sequence $(f_n)$ is collapsing.

\item If $\mathrm{Re}(\kappa_2) = 0$ and
\begin{enumerate}
\item  $\mathrm{Re}(\kappa_1)>0$,  then the sequence $(f_n)$ is expulsive.
\item  $\mathrm{Re}(\kappa_1)<0$, then the sequence $(f_n)$ is collapsing.
\item $\mathrm{Re}(\kappa_1) = 0$, then the sequence $(f_n)$ is rotationally linearizable, hence all limit maps have rank $1$.
\end{enumerate}
\end{enumerate}
\end{rem}

\section{Existence of parabolic curves}\label{section:parabolic}
	
The purpose of this section is to prove the following proposition.

\begin{prop}\label{prop:curve}
	Let $P(z,w):=(f(z)+\frac{\pi^2}{4}w,g(w))$, with $f(z)=z+z^2+bz^3+O(z^4)$  and $g(w)=w-w^2+O(w^3)$.
	Then $P$ has at least 3 parabolic curves: one is contained in the invariant fiber ${\black w}=0$ and is an attracting petal for $f$;
	the other two are graphs over the same petal $\mathcal{P}$ in the parabolic basin $\bcal_g$. Moreover they are of the form
	$$\zeta^\pm(w)=\pm c_1 \sqrt{w} + c_2 w \pm  c_3 w^{3/2}+O(w^2),$$
	where $c_1=\frac{\pi i}{2}$ and $c_2=\frac{b\pi^2}{8}- \frac{1}{4}$.
\end{prop}

Proposition \ref{prop:curve}  gives a positive answer to a question posed in \cite{ABDPR}. {\black We note that the result does not follow from the results of Hakim~\cite{Hakim}, as the $2$ characteristic directions we consider are \emph{degenerate}, in the language used by Hakim}.
The existence of $3$ parabolic curves can be derived from the recent paper \cite{LH-R} by Lopez-Hernanz and Rosas. However, their proof gives no guarantee that the parabolic curves $\zeta^\pm$ are graphs over the same petal in $\bcal_g$, which is crucial for our purpose.

Let us start by observing that $P$ is semi-conjugate to a map $Q$, holomorphic near the origin, given by
$$
Q(z,\eps)=\left(f(z)+\frac{\pi^2}{4}\eps^2, \eps- \frac{\eps^3}{2}+O(\eps^5)\right)
$$
(with $\eps^2=w$). The map $Q$ has $3$ characteristic directions: ${\black \epsilon}=0$, $z=\frac{\pi i}{2}\eps$ and $z=-\frac{\pi i}{2}\eps$.
It is clear that there is a parabolic curve tangent to the characteristic direction ${\black \epsilon}=0$, namely the attracting petal for $f$ in the invariant
fiber $\{{\black \epsilon}=0\}$. We call this parabolic curve the trivial curve. For the existence of the two other parabolic curves we will use a graph transform argument.
	
Let us write $Q(z,\eps) = \left(f_{\eps}(z), \tilde{g}(\eps)\right)$, so that $f_{\eps}(z)=f(z)+\frac{\pi^2}{4}\eps^2$ and $\teps:=\tilde{g}(\eps)=\sqrt{g(\eps^2)}=\eps -\frac{\eps^3}{2}+ O(\eps^5)$. We are looking for parabolic curves of the form $\eps\rightarrow(\zeta(\eps),\eps)$, hence satisfying the equation
\begin{equation}\label{eq:funcequation}
Q(\zeta(\eps),\eps)=(\zeta(\teps),\teps).
\end{equation}
Equivalently we are looking for a function $\zeta$, defined for $\epsilon$ in a parabolic petal of $\tilde{g}$, satisfying the functional equation
$$
\zeta(\tilde{g}(\eps)) =f_{\eps}( \zeta(\eps)).
$$
	
We will prove that $Q$ has two parabolic curves $\zeta^\pm$, corresponding to the characteristic directions $z=\pm\frac{\pi i}{2}\eps$,  which are graphs over the same attracting petal of $\tilde{g}$ in the right half-plane. This will complete the proof of Proposition \ref{prop:curve}, since these two parabolic curves can be lifted to parabolic curves of $P$ satisfying the desired properties.

The key idea in proving the existence of $\zeta(\epsilon)$ is to start with sufficiently high order jets $\zeta_1(\epsilon)$ of the \emph{formal} solution to the equation \eqref{eq:funcequation}, and then apply a graph transform argument, starting with $\zeta_1$. By starting with higher order jets, we obtain higher order error estimates, but the constants in those estimates are likely to deteriorate. However, these estimates can be controlled by dropping the order of the error estimates by $1$, and working with $|{\black \epsilon}| < \delta$, with $\delta$ depending on the order of the jets. It turns out that starting with jets of order $20$ is sufficient to obtain convergence of the graph transforms. {\black We do not claim that $20$ is the minimal order for which convergence can be obtained, only that the order suffices for our purposes.}
	
\begin{lem}\label{lemma:1} For every integer $n>0$ there exists $\zeta_1(\eps)=c_1\eps+c_2\eps^2+c_3\eps^2+\cdots +c_n\eps^n$ and $\delta>0$  such that $|\zeta_1(\teps)-f_{\eps}(\zeta_1(\eps))|<|\eps|^n$ for all $|\eps|<\delta$. Moreover we have $c_1=\pm\frac{\pi i}{2}$ and $c_2 =\frac{b\pi^2}{8}-\frac{1}{4}$.
\end{lem}
	
\begin{proof}  Recall from \cite{ABDPR} that by choosing $\zeta_1(\epsilon) = c_1 \eps + c_2 \eps^2$, with
$c_1=\pm\frac{\pi i}{2}$ and $c_2 =\frac{b\pi^2}{8}-\frac{1}{4}$, we obtain
$$
|\zeta_1(\teps)-f_{\eps}(\zeta_1(\eps))| < O(|\eps|^4).
$$
Now suppose that $c_1, \ldots , c_n$ are found such that for $\zeta(\epsilon) = c_1\eps+ \cdots + c_n\eps^n$ we have
$$
|\zeta_1(\teps)-f_{\eps}(\zeta_1(\eps))| < O(|\eps|^{n+2}).
$$

Let $E_n(\eps) :=f_{\eps}(\zeta_1(\eps))-\zeta_1(\teps)$. For $c_{n+1} \in \C$, let $$E_{n+1}(\eps):=f_{\eps}(\zeta_1(\eps)+c_{n+1} \eps^{n+1})-\zeta_1(\teps)-c_{n+1} \teps^{n+1};$$ we shall prove that there exists some $c_{n+1}$ such that $E_{n+1}=O(\eps^{n+3})$. Indeed,
	\begin{align*}
		f_{\eps}(\zeta_1(\eps)+c_{n+1} \eps^{n+1})&=f_{\eps}(\zeta_1(\eps)) +f'(\zeta_n(\eps)) c_n \eps^{n+1}+O(\eps^{2n+2}) \\
		&=f_{\eps}(\zeta_1(\eps)) + (1+2c_1\eps) c_{n+1} \eps^{n+1}+O(\eps^{n+3} ).
	\end{align*}	
	On the other hand, we have $c_{n+1} \teps^{n+1} = c_{n+1} \eps^{n+1}+O(\eps^{n+3})$;
	so
	$$E_{n+1}(\eps) = E_n(\eps)+ 2c_1 c_{n+1} \eps^{n+2}+O(\eps^{n+3}).$$
	Since $E_n(\eps) = O(\eps^{n+2})$ (and $c_1 \neq 0$), we may therefore find some value of $c_{n+1}$
	for which $E_{n+1}(\eps)=O(\eps^{n+3})$.

We conclude that if $\delta$ is small enough then $|\zeta_1(\teps)-f_{\eps}(\zeta_1(\eps))|<|\eps|^n$ for all $|\eps|<\delta$.
\end{proof}

\begin{rem} The choice of parabolic curve is determined by the choice of $c_1$. From now on we will assume that $c_1=\frac{\pi i}{2}$; for the case $c_1=-\frac{\pi i}{2}$ the proofs are essentially the same.
\end{rem}
\medskip

{\black For $R \in \mathbb C$ we} write $\mathbb{H}_R=\{Z \in\C \mid \arg(Z-R)\in (-\frac{\pi}{2} -\epsilon_0,\frac{\pi}{2}+\epsilon_0)\}$ for some $\epsilon_0>0$,
and
$$
\mathcal{P}_{\delta}=\{\eps \in\C\mid \eps^{-2}\in\mathbb{H}_{\delta^{-2}}
 \text{ and } \re(\eps)>0 \}.
$$

For $\delta>0$ sufficiently small the petal $\mathcal{P}_{\delta}$ is forward invariant under $\tilde{g}$, i.e. $\tilde{g}(\mathcal{P}_{\delta})\subset \mathcal{P}_{\delta}$. Recall the existence of Fatou coordinates on $P_\delta$: the function $\tilde{g}$ is conjugate to the translation $T_1: Z \mapsto Z+1$ via a conjugation of the form
$$
Z = \frac{1}{\epsilon^2} + \alpha \log(\epsilon) + o(1),
$$
where the constant $\alpha$ depends on $g$. All forward orbits in $P_\delta$ converge to $0$ tangent to the positive real axis, and the conjugation gives the estimates

\begin{equation}\label{estimates}
|\mathrm{Re}(\tilde{g}^k(\eps))|<\frac{C}{\sqrt{k}} \; \; \mathrm{and} \; \; |\mathrm{Im}(\tilde{g}^k(\eps))|<\frac{C}{k},
\end{equation}
for a uniform $C>0$ depending on $\delta$. We note that by choosing $\delta$ sufficiently small, the constant $C$ can be chosen arbitrarily small as well.

\begin{lem}\label{lemma:2} Let $n > 0$ and $\zeta_1(\eps)$ be as in Lemma \ref{lemma:1}. There exist $\delta,A>0$  such that for every $|\eps|<\delta$ we have
$$
|f^{-1}(f(\zeta_1(\eps))+3\eps^4)-\zeta_1(\eps)| \le A|\eps|^4.
$$
\end{lem}
\begin{proof}
The Taylor series expansion of $f$ gives
$$
|f^{-1}(f(\zeta_1(\eps))+3\eps^4)-\zeta_1(\eps)| \le \sum_{i=1}^{\infty}\left|\frac{(f^{-1})^{(i)}(f(\zeta_1(\eps)))}{i!}\right|3^i|\eps^4|^i,
$$
and the desired estimate follows immediately.
\end{proof}

\begin{lem}\label{lemma:3} Let $n > 0$ and $\zeta_1(\eps)$ be as in Lemma \ref{lemma:1}, $A>0$ and $\delta>0$ sufficiently small. Let  $(\zeta_k(\eps))$ be any sequence of holomorphic functions defined on $\mathcal{P}_{\delta}$ and satisfying
$$
|\zeta_k(\eps)-\zeta_1(\eps)| < A|\eps|^4.
$$
Then there exists $C_1 > 0$, depending on $\zeta_1$, such that
$$
\left|\prod_{s={\black \ell}}^k f'(\zeta_s(\tilde{g}^{k+1-s}(\eps)))\right|^{-1}<C_1\cdot (k+1-{\black \ell}),
$$
for all $\eps\in \mathcal{P}_{\delta}$ and every $0<{\black \ell}\leq k$.
\end{lem}
	
\begin{proof} Let us write $x_k=\mathrm{Re}(\tilde{g}^k(\eps))>0$ and $y_k=\mathrm{Im}(\tilde{g}^k(\eps))$. Estimates (\ref{estimates}) imply

\begin{equation}\label{estimates2}
\sum_{k=0}^{\infty}|\tilde{g}^k(\eps)|^3<K<\infty \qquad \textrm{for all} \quad \eps\in \mathcal{P}_{\delta}.
\end{equation}
Since by assumption $|\zeta_s(\eps)-\zeta_1(\eps)| < A|\eps|^4$ for every $s\ge 1$, it follows that $\zeta_s(\eps)=c_1\eps+c_2\eps^2+O(\eps^3)$ and
$$
|f'(\zeta_s(\eps))-f'(\zeta_1(\eps))| < B|\eps^4|,
$$
where $B>0$  depends only on $\zeta_1$ and $A$.

Observe that $f'(z)=1+2z+3bz^2+O(z^3)=e^{2z+(3b-2)z^2+O(z^3)}$, hence we obtain
$$
f'(\zeta_s(\eps))=e^{\pi i \eps +\left(\frac{\pi^2(1-b)}{2}-\frac{1}{2}\right)\eps^2+O(\eps^3)},
$$
where the bound $O(\eps^3)$ is uniform with respect to $s$.

Therefore we can find $C_1>0$ such that
\begin{align*}
\left|\prod_{s=\ell}^k f'(\zeta_s(\tilde{g}^{k+1-s}(\eps)))\right|&>
\left|e^{\sum_{s=\ell}^k \mathrm{Re}\left(\pi i \tilde{g}^{k+1-s}(\eps) + \left(\frac{\pi^2(1-b)}{2}-\frac{1}{2}\right)(\tilde{g}^{k+1-s}(\eps))^2\right)+O((\tilde{g}^{k+1-s}(\eps))^3)}\right|\\			
&>\frac{1}{C_1}\left|e^{\sum_{s=\ell}^k-\pi y_{k+1-s} +\left(\frac{\pi^2(1- \mathrm{Re}(b))}{2}-\frac{1}{2}\right)x^2_{k+1-s}}\right|\\
&>\frac{1}{C_1}\left|e^{-\sum_{s=\ell}^k\frac{1}{k+1-s}}\right|\\
&>\frac{1}{C_1(k+1-\ell)}.
\end{align*}
In the first inequality we used the fact that $|e^z|=e^{\mathrm{Re}(z)}$. The second inequality follows from estimates (\ref{estimates}) and  (\ref{estimates2}). The third inequality depends on the constant $C$ from (\ref{estimates}) being sufficiently small, which can be guaranteed by taking sufficiently small $\delta$.
\end{proof}

\begin{rem}Note that the estimates in Lemmas \ref{lemma:1}, \ref{lemma:2} and \ref{lemma:3} hold regardless of the choice of $n$ in the definition of $\zeta_1$. If $n$ is increased, then all estimates hold, with the same constants, for $\delta$ sufficiently small. It turns out that it will be sufficient for us to work with $n = 20$, and we will work with this choice from now on.
\end{rem}

\begin{lem}\label{lemma:4} There exists sufficiently small $\delta >0$ such that for every $k\geq 2$  and every $\eps \in\mathcal{P}_{\delta}$ we have
\begin{equation*}\label{ineq}
|\tilde{g}^{k}(\eps)|^{19}k+ |\tilde{g}^{k}(\eps)|^{39}(k-1)+\sum_{\ell=2}^{k-1}\frac{|\tilde{g}^{k+1-\ell}(\eps)|^{23}(k-\ell)}{(\ell-1)^{4}}+\frac{|\tilde{g}(\eps)|^{23}}{(k-1)^ {4}}<\frac{4|\eps|^{12}}{k^2}%<\frac{|\eps|^4}{k^2}.
\end{equation*}
\end{lem}

\begin{proof} We will prove that each of the four terms in the left hand summation is bounded by $\frac{|\eps|^{12}}{k^2}$.  It follows from (\ref{estimates}) that for every $0\leq \ell \leq 19$ we have
$$
|\tilde{g}^{k}(\eps)|^{19}k<\frac{C^{19}k}{|k+\frac{1}{\eps^2}|^{\frac{19}{2}}}<\frac{C^{19}| \eps|^{\ell} k}{k^{\frac{19-\ell}{2}}}.
$$
If we choose $\ell=13$ and assume that $\delta$ is small enough, then we get
$$
|\tilde{g}^{k}(\eps)|^{19}k<\frac{|\eps|^{12}}{k^2}
$$
for $\eps\in \mathcal{P}_{\delta}$. The desired bound for a second term follows immediately from the inequality
$$
|\tilde{g}^{k}(\eps)|^{39}(k-1)<|\tilde{g}^{k}(\eps)|^{19}k.
$$

Next observe that for every $k\geq 2$ we have
$$
\frac{|\tilde{g}(\eps)|^{23}}{(k-1)^{4}}<\frac{2^2|\eps|^{23}}{(2(k-1)^2)^{2}}<\frac{|\eps|^{12}}{k^2},
$$
where the last inequality holds for sufficiently small $\delta$. Finally, for the third term in the summation we use (\ref{estimates}) to obtain
$$
\sum_{\ell=2}^{k-1}\frac{|\tilde{g}^{k+1-\ell}(\eps)|^{23}(k-\ell)}{(\ell-1)^{4}}<\sum_{\ell=2}^{k-1}\frac{C^{10}|\eps|^{13}(k-\ell)}{(k+1-\ell)^{5}(\ell-1)^{4}}<C^{10}|\eps|\sum_{\ell=2}^{k-1}\frac{|\eps|^{12}}{(k-\ell)^{4}(\ell-1)^{4}}.
$$

In order to obtain the desired bound it suffices to prove that
$$
\sum_{\ell=2}^{k-1}\frac{1}{(k-\ell)^{4}(\ell-1)^{4}}<\frac{4^4}{k^2}.
$$

First observe that
$$
\frac{1}{(\ell-1)(k-\ell)}\leq \frac{4}{k}
$$
for every $k\geq3$ and  $2\leq \ell\leq k-1$. To see this let us denote $s=\ell-1$ and $t=k-1$. The above inequality now translates to
$$\frac{1}{s(t-s)}\leq \frac{4}{t+1}$$
for $t\geq2$ and  $1\leq s\leq t-1$, and hence to
$$p_t(s):=4s^2-4ts+t+1\leq 0.$$
Observe that $p_t(1)< 0$ and that roots of $p_t(s)$ lie outside the closed interval $[1,t-1]$. Therefore we obtain
$$
\sum_{\ell=2}^{k-1}\frac{1}{(k-l)^{4}(\ell-1)^{4}}<\sum_{\ell=2}^{k-1}\frac{4^{4}}{k^{4}}<\frac{4^4}{k^{2}},
$$
and hence for $\delta$ sufficiently small
$$
\sum_{\ell=2}^{k-1}\frac{|\tilde{g}^{k+1-\ell}(\eps)|^{23}(k-\ell)}{(\ell-1)^{4}}<\frac{|\eps|^{12}}{k^2}.
$$

This completes the proof of Lemma \ref{lemma:4}.

\end{proof}
	
\begin{proof}[Proof of Proposition \ref{prop:curve}:] As we have remarked at the beginning of this section, it is enough to prove that $Q$ has two parabolic curves $\zeta^\pm$ corresponding to the characteristic directions $z=\pm\frac{\pi i}{2}\eps$, both curves graphs over the same attracting petal of $\tilde{g}$ in the right half-plane.
 By Lemma \ref{lemma:1} there exist $\delta>0$ and $\zeta_1(\eps)=c_1\eps+c_2\eps^2 +\ldots + c_{20}\eps^{20}$ such that $|\zeta_1(\teps)-f_{\eps}(\zeta_1(\eps))|<|\eps|^{20}$ for $|\eps|<\delta$. Let $A>0$ be as in Lemma \ref{lemma:2}, and let $C_1>0$ be the constant defined in Lemma \ref{lemma:3}.

We will show that the sequence of functions defined inductively by
$$
\zeta_{k+1}(\eps):=f_{\eps}^{-1}(\zeta_k(\teps))
$$
is convergent, and that the limit satisfies the functional equation \eqref{eq:funcequation}. Let us define
$$E_{k}(\eps):=\zeta_{k}(\teps)-f_{\eps}(\zeta_{k}(\eps)) $$
and observe that
$$\zeta_{k+1}(\eps)=f_{\eps}^{-1}(\zeta_k(\teps))=f_{\eps}^{-1}(f_{\eps}(\zeta_k(\eps))+E_k(\eps))$$
and hence
$$f_{\eps}(\zeta_{k+1}(\eps))=f_{\eps}(\zeta_{1}(\eps))+\sum_{{\black \ell}=1}^kE_{\black \ell}(\eps).$$
Note that we can replace $f_\epsilon$ by $f$ on both sides, giving
$$
\zeta_{k+1}(\eps)=f^{-1}\left(f(\zeta_1(\eps))+\sum_{{\black \ell}=1}^kE_{\black \ell}(\eps)\right),
$$
and hence
$$
\zeta_{k+1}(\eps)=\zeta_{1}(\eps)+\sum_{i=1}^{\infty}\frac{(f^{-1})^{(i)}(f(\zeta_1(\eps)))}{i!}\left(\sum_{{\black \ell}=1}^k E_{\black \ell}(\eps)\right)^i.
$$

We will prove that $|E_k(\eps)|<\frac{|\eps|^4}{|k-1|^2}$ for every $k\geq 2$ on some small petal $\mathcal{P}_{\delta}$.  This will imply that the sequence $\zeta_{k+1}$ converges to a parabolic curve $\zeta$ on $\mathcal{P}_{\delta}$ for sufficiently small $\delta$.

We claim that there exists $\delta>0$ such that for every $\eps\in\mathcal{P}_{\delta}$ and every $k>1$ the following two statements hold:
\begin{enumerate}
\item[$\bf I_k(1):$] $\; \; \; \; \; \; |\zeta_k(\eps)-\zeta_1(\eps)|<A|\eps|^4$, \; \; and
\item[$\bf I_k(2):$] $\; \; \; \; \; \; |E_k(\eps)|<\frac{4|\eps|^{12}}{|k-1|^2}<\frac{|\eps|^4}{|k-1|^2}$
\end{enumerate}
We will prove these two statements simultaneously by induction on $k$. %First recall that $|E_1(\eps)|=|\zeta_1(\teps)-f_{\eps}(\zeta_1(\eps))|<|\eps|^{20}$.

\medskip

\noindent {\bf Step 1:}  First we prove $I_2(1)$. By definition
$$\zeta_2=\zeta_{1}(\eps)+\sum_{i=1}^{\infty}\frac{(f^{-1})^{(i)}(f(\zeta_1(\eps)))}{i!}\left(E_1(\eps)\right)^i,$$
hence by Lemma \ref{lemma:2} we obtain the desired inequality.

Next we prove that $I_2(2)$. Observe that for sufficiently small $\delta$ we get
\begin{align*}
|E_2(\eps)|&<\left|\frac{E_1(\teps)}{f'(\zeta_1(\teps))}\right|+C_2|E_1(\teps)|^2\\
&<C_1|\tilde{g}(\eps)|^{20}+C_2|\eps||\tilde{g}(\eps)|^{40}\\
&< C_1|\eps||\tilde{g}(\eps)|^{19}+C_2|\eps||\tilde{g}(\eps)|^{39}\\
&<4|\eps|^{12}.
\end{align*}
Here $C_1$ is the constant introduced in Lemma \ref{lemma:3}.

\medskip

\noindent {\bf Step 2:} Now let us assume that $I_{\black \ell}(1)$ and $I_{\black \ell}(2)$ hold for every $2\leq {\black \ell}\leq k$. Observe that
$$
|\zeta_{k+1}(\eps)-\zeta_1(\eps)|<\sum_{i=1}^{\infty}\left|\frac{(f^{-1})^{(i)}(f(\zeta_1(\eps)))}{i!}\right|\left|\sum_{{\black \ell}=1}^k E_{\black \ell}(\eps)\right|^i.
$$

Since  $|E_{\black \ell}(\eps)|<\frac{|\eps|^4}{|{\black \ell}-1|^2}$ for ${\black \ell}\geq2$ and $|E_1(\eps)|<|\eps|^4$ we get
$$\left|\sum_{{\black \ell}=1}^k E_{\black \ell}(\eps)\right|<3|\eps|^4$$
hence by Lemma \ref{lemma:2} inequality $I_{k+1}(1)$ holds.

Observe that
\begin{align*}
E_{k+1}(\eps)&=\zeta_{k+1}(\teps)-f_{\eps}(\zeta_{k+1}(\eps)) \\
&=  \zeta_{k+1}(\teps)-\zeta_{k}(\teps)\\
&=f^{-1}_{\teps}(f_{\teps}(\zeta_k(\teps))+E_k(\teps))-\zeta_{k}(\teps)\\
& =  f^{-1}(f(\zeta_k(\teps))+E_k(\teps))-\zeta_{k}(\teps)\\
&=(f^{-1})'(f(\zeta_k(\teps)))\cdot E_k(\teps)+O(E_k(\teps)^2), &
\end{align*}
where the constant in the order can be chosen independently from $k$. It follows that there exists $C_2>0$ independent of $k$ such that

\begin{equation}\label{inequality1}
|E_{k+1}(\eps)|<\left|\frac{E_k(\teps)}{f'(\zeta_k(\teps))}\right|+C_2|E_k(\teps)|^2.
\end{equation}
Using the inequality (\ref{inequality1}) successively we obtain
\begin{equation}
\begin{aligned}
& |E_{k+1}(\eps)| < \left|\frac{E_k(\tilde{g}(\eps))}{f'(\zeta_k(\tilde{g}(\eps)))}\right|+C_2|E_k(\tilde{g}(\eps))|^2\\
	&<\left|\frac{E_{k-1}(\tilde{g}^2(\eps))}{f'(\zeta_{k-1}(\tilde{g}^2(\eps)))\cdot f'(\zeta_k(\tilde{g}(\eps)))}\right|+C_2\frac{|E_{k-1}(\tilde{g}^2(\eps))|^2}{|f'(\zeta_k(\tilde{g}(\eps)))|}+C_2|E_k(\tilde{g}(\eps))|^2\\
\label{triple_eq} &<\frac{|E_1(\tilde{g}^{k}(\eps))|}{\prod_{{\black \ell}=1}^{k}{|f'(\zeta_{\black \ell}(\tilde{g}^{k+1-{\black \ell}}(\eps)))|}}+C_2\sum_{{\black \ell}=1}^{k-1} \frac{|E_{\black \ell}(\tilde{g}^{k+1-{\black \ell}}(\eps))|^2}{\prod_{s={\black \ell}+1}^k|f'(\zeta_s(\tilde{g}^{k+1-s}(\eps)))|}+C_2|E_k(\teps)|^2
\end{aligned}
\end{equation}

Combining equation \ref{triple_eq} and Lemma \ref{lemma:3} gives
\begin {align*}
|E_{k+1}(\eps)|&< C_1|\tilde{g}^{k}(\eps)|^{20}k+C_1C_2|\tilde{g}^{k}(\eps)|^{40}(k-1)\\
 & +16C_1C_2\sum_{{\black \ell}=2}^{k-1}\frac{|\tilde{g}^{k+1-{\black \ell}}(\eps)|^{24}(k-{\black \ell})}{({\black \ell}-1)^{4}}+16C_1\frac{|\tilde{g}(\eps)|^{24}}{(k-1)^{4}}\\
&< C_1|\eps||\tilde{g}^{k}(\eps)|^{19}k +C_1C_2|\eps||\tilde{g}^{k}(\eps)|^{39}(k-1)\\
& +16C_1C_2|\eps|\sum_{{\black \ell}=2}^{k-1}\frac{|\tilde{g}^{k+1-{\black \ell}}(\eps)|^{23}(k-{\black \ell})}{({\black \ell}-1)^{4}}+16C_1|\eps|\frac{|\tilde{g}(\eps)|^{23}}{(k-1)^{4}}.
\end {align*}
If $\delta$ is sufficiently small this last inequality together with Lemma \ref{lemma:4} implies
$$
|E_{k+1}(\eps)|<\frac{4|\eps|^{12}}{k^2}<\frac{|\eps|^{4}}{k^2},
$$
completing the proof of $I_{k+1}(2)$ and thus the induction argument. We emphasize that throughout the proof $\delta$ can be chosen dependently of $k$.

To summarize, the equation
$$ \zeta_{k+1}(\eps)=f^{-1}\left(f(\zeta_1(\eps))+\sum_{{\black \ell}=1}^kE_{\black \ell}(\eps)\right)$$
implies that for sufficiently small $\delta$ the sequence $\zeta_k$ converges on $\mathcal{P}_{\delta}$ to a parabolic curve $\zeta$ satisfying
$\zeta(\teps)=f_{\eps}(\zeta(\eps))$. Recall that we have only proven the existence of parabolic curve for $c_1=\frac{\pi i}{2}$. For $c_1=-\frac{\pi i}{2}$ we can use  same arguments as above, but we might get a different value for $\delta$. Since the parabolic petals are nested and forward invariant, both parabolic curves are graphs over the petal with minimal $\delta$.
\end{proof}

From the proof it follows that
$$
\zeta^\pm(\eps)=\pm c_1 \eps + c_2 \eps^2 \pm  c_3 \eps^3+O(\eps^4),
$$
where $c_1=\frac{\pi i}{2}$ and $c_2=\frac{b\pi^2}{8}- \frac{1}{4}$.

\section{Estimates on Convergence towards Lavaurs map}\label{section:ConvergenceEstimates}

\subsection{Preliminaries}

The goal of this section is to obtain explicit estimates for one of the main objects to appear in our arguments: the functions $A(\epsilon,z)$ and $A_0(z)$, \black  {which measure how much the dynamics differ from a translation after
a certain change of coordinates. The key difference between this section and the corresponding computations in \cite{ABDPR} is that we now know that we have two exactly invariant parabolic curves $\zeta^\pm$, instead of invariant jets. This is used crucially in the proof of Proposition \ref{prop:phianalytic}.}

\begin{defi}
	Let $f_w(z):=f(z)+\frac{\pi^2}{4}w$,
	where $f(z)=z+z^2+z^3+O(z^4)$ is a degree $d$ polynomial. Let $g(w)=w-w^2+O(w^3)$ be a degree $d$ polynomial.
\end{defi}

In what follows, we set $\epsilon:=\sqrt{w}$, {\black working throughout with the branch that takes positive values on the positive real axis. We note that this branch is well-defined on the parabolic basin of the polynomial $g$}. Abusing notation, we write
$f_\epsilon(z):=f(z)+\frac{\pi^2}{4}\epsilon^2$ and $\zeta^\pm(\epsilon)=\pm i\frac{\pi}{2} \epsilon + c_2 \epsilon^2+ O(\epsilon^3)$, where $\zeta^\pm$ are the parabolic curves constructed in the preceding section. Let $\tilde g(\epsilon):=\sqrt{g(\epsilon^2)} = \epsilon - \frac{1}{2}\epsilon^3+O(\epsilon^5)$ ($\tilde g$ is analytic near $\epsilon=0$).

Let us first record here the following lemma for later use:

\begin{lem}\label{lem:epsj} {\black Let $w_0\in \mathcal{B}_g$ and let $\eps_j:=\sqrt{g^{n^2+j}(w_0)}$.}
	For $1 \leq j \leq n$, we have:
	$$\epsilon_j = \frac{1}{n}- \frac{j}{2n^3}-\frac{\phi^{{\black\iota}}_g (w_0)}{2n^3}+ o(\frac{1}{n^3})$$
\end{lem}

\begin{proof} {\black Let us write $w_{n^2+j}:=g^{n^2+j}(w_{\black 0})$.}
	We have
	$$\phi_g(w_{n^2+j})= \phi_g^{{\black \iota}}(w_{\black 0})+ n^2+j=\frac{1}{w_{n^2+j}}+o(1)$$
	(note that we assume here $g(w)=w-w^2+w^3+O(w^4)$).
	Therefore
	$$w_{n^2+j}=\frac{1}{n^2+j+\phi_g^{{\black\iota}}(w_{\black 0})+o(1)},$$
	and
	\begin{align*}
	\epsilon_j &= \sqrt{w_{n^2+j}} = \frac{1}{n} \left(1+ \frac{j+\phi_g^{{\black \iota}}(w_{\black 0})}{n^2}+o\left(\frac{1}{n^2}\right) \right)^{-\frac{1}{2}} \\
	&=\frac{1}{n} \left(1- \frac{1}{2}  \frac{j+\phi_g^{{\black \iota}}(w_{\black 0})}{n^2} + o\left(\frac{1}{n^2}\right)\right).
	\end{align*}
\end{proof}

\begin{defi}
	{\black Let $$
	\psi_\epsilon^{\iota} (z):=\frac{1}{i\pi} \log \left[\frac{\zeta^+(\epsilon)-z}{z-\zeta^-(\epsilon)} \right] + 1
	$$
	and
	 $$
	 \psi_\epsilon^{o} (z):=\frac{1}{i\pi} \log \left[\frac{\zeta^+(\epsilon)-z}{z-\zeta^-(\epsilon)} \right] - 1
	 $$
	where $\log$ is the principal branch of the logarithm.}
\end{defi}

Remark: with that choice of branch, $\psi$ is defined on $\C \backslash L_\epsilon$,
where $L_\epsilon$ is the real line through  $\zeta^+(\epsilon)$ and $\zeta^-(\epsilon)$ minus the segment
$[\zeta^-(\epsilon), \zeta^+(\epsilon)]$. In particular, $\psi_\epsilon^\iota$ and $\psi_\epsilon^o$ are both
defined in a disk centered at $z=0$ whose radius is of order $\epsilon$.

It will also be useful to note that
\begin{equation}\label{eq:psiinv}
	\left( \psi_{\epsilon}^{\iota/o} \right)^{-1} (Z) = \frac{\zeta^+(\eps)-e^{\pm i\pi Z} \zeta^-(\eps)}{1-e^{\pm i\pi Z}} = - \frac{\pi}{2} \eps \cot\left(\pm \frac{\pi Z}{2}\right) + O(\eps^2).
\end{equation}

\begin{defi}
	Let
	\begin{enumerate}
		\item $A(\epsilon,z):=\psi_{\tilde g(\epsilon)}^{\iota/o} \circ f_{\epsilon}(z)-\psi_\epsilon^{\iota/o}(z)-\epsilon$
		\item $A_0(z):=-\frac{1}{f(z)}+\frac{1}{z}-1$
	\end{enumerate}
\end{defi}

Note that the formula for $A(\epsilon,z)$ does not depend on whether the ingoing or outgoing coordinate $\psi_\epsilon$ is used, and is therefore well defined.

\begin{prop}\label{prop:phianalytic}
	We have
	\begin{enumerate}
		\item $A_0$ is analytic near zero.
		\item There exists $r>0$ such that for all $\epsilon \neq 0$ in a neighborhood of zero, $A(\epsilon, \cdot)$ is analytic on
		$\D(0,r)$.
	\end{enumerate}
\end{prop}

\begin{proof}
	(1) {\black A quick computation shows that $A_0(z)=\frac{f(z)-z-zf(z)}{zf(z)}=\frac{O(z^2)}{1+O(z)}$, from where the conclusion easily follows. }
	
	For (2), note that
	\begin{align*}
		A(\epsilon,z) &= \frac{1}{i\pi} \log \left(\frac{\zeta^+( \tilde g(\epsilon)) - f_{\epsilon}(z)}{f_{\epsilon}(z) - \zeta^-( \tilde g(\epsilon))} :
		\frac{\zeta^+(\epsilon)-z}{z-\zeta^-(\epsilon)}		\right) -\epsilon \\
		&= \frac{1}{i\pi} \log  \left(
		\frac{f_{\epsilon}( \zeta^+(\epsilon))- f_{\epsilon}(z)}{\zeta^+(\epsilon)-z}		
		 : 	\frac{f_{\epsilon}(z) - f_{\epsilon}( \zeta^-(\epsilon))}{z-\zeta^-(\epsilon)}
		\right) -\epsilon
	\end{align*}
	From the above expression we see that the singularities at $z=\zeta^\pm(\epsilon)$ are in fact
	removable, unless one of the points coincides with a critical point of $f$. The fact that these critical points are bounded away from zero completes the proof.
\end{proof}

\begin{lem}\label{prop:1/z}
	Let $K$ be a compact subset of $\C^*$. There exists $C=C_K>0$ such that for all $z \in K$,
	$$\left| \frac{z-\zeta^+(\epsilon)}{z-\zeta^-(\epsilon)} -  (1 - \frac{i\pi}{z} \epsilon - \frac{\pi^2}{2z^2} \epsilon^2) \right| \leq C \epsilon^3$$
\end{lem}

\begin{proof}
	For $z \in K$, we have:
	\begin{align*}
		\frac{\zeta^+(\epsilon)-z}{z-\zeta^-(\epsilon)} &= \frac{\zeta^+(\epsilon)}{z-\zeta^-(\epsilon)}
		- \frac{z}{z-\zeta^-(\epsilon)}\\
		&= \frac{\zeta^+(\epsilon)}{z} \left(  \frac{1}{1-\frac{\zeta^-(\epsilon)}{z}}     \right)
		- \frac{1}{1-\frac{\zeta^-(\epsilon)}{z}} \\
		&=\frac{\zeta^+(\epsilon)}{z} \left(1+\frac{\zeta^-(\epsilon)}{z}  +O(\epsilon^2) \right)
		- \left( 1+\frac{\zeta^-(\epsilon)}{z} + (\frac{\zeta^-(\epsilon)}{z})^2+O(\epsilon^3)  \right)\\
		&= \frac{c_1 \epsilon + c_2 \epsilon^2 - \frac{c_1^2}{z} \epsilon^2}{z}
		-1 - \frac{-c_1 \epsilon +c_2\epsilon^2 }{z} - \frac{c_1^2 \epsilon^2}{z^2} +O(\epsilon^3)\\
		&= -1 + \frac{2 c_1}{z} \epsilon - \frac{2c_1^2}{z^2} \epsilon^2 + O(\epsilon^3)\\
		&= -1 + \frac{i\pi}{z} \epsilon + \frac{\pi^2}{2z^2} \epsilon^2 + O(\epsilon^3).
	\end{align*}

\end{proof}

\begin{lem}\label{lem:1/z}
	Let $K$ be a compact subset of $\C^*$.
	Then
	$$\frac{f_\epsilon(z) - f_\epsilon(\zeta^+(\epsilon))}{f_\epsilon(z) - f_\epsilon(\zeta^-(\epsilon))}
	 = 1-\frac{i\pi}{f(z)}\epsilon -\frac{\pi^2}{2f(z)^2}\epsilon^2+O(\epsilon^3).$$
	 As in the previous lemma the constant in the $O$ depends on $K$.
\end{lem}

\begin{proof}
	The invariance of the parabolic curves gives
    $$
    \begin{aligned}
    \frac{f_\epsilon(z) - f_\epsilon(\zeta^+(\epsilon))}{f_\epsilon(z) - f_\epsilon(\zeta^-(\epsilon))}
	 & = \frac{f_\epsilon(z) - \zeta^+(\tilde{g}(\epsilon))}{f_\epsilon(z) - \zeta^-(\tilde{g}(\epsilon))}
     = 1 - \frac{i\pi}{f_\epsilon(z)} \tilde{g}(\epsilon) + \frac{\pi^2}{2 f_\epsilon(z)^2}\tilde{g}(\epsilon^2) + O(\epsilon^3)\\
     & = 1 - \frac{i \pi}{f(z)} \epsilon + \frac{\pi^2}{2f(z)^2} \epsilon^2 + O(\epsilon^3).
     \end{aligned}
     $$
     The last equality uses the fact that $\tilde{g}(\epsilon) = \epsilon + O(\epsilon^3)$.
\end{proof}

\begin{prop}\label{prop:estimateA}
	There exists a constant $C_0 \in \C$ (depending only on $f$ and
	$g$) such that:
	$$A(\epsilon,z)=\epsilon A_0(z) + \epsilon^3 C_0+
	O(\epsilon^4, \epsilon^3 z)$$
	where the constants in the $O$ are uniform for $(z,\epsilon) \in \C^2$ near $(0,0)$
	(with $\mathrm{Re}(\epsilon)>0$).
\end{prop}

\begin{proof}
	Let $K$ be a compact of $\C^*$. Then by the two previous lemmas, we have
	\begin{align*}
		A(\epsilon,z) &= \frac{1}{i\pi} \log \left(  \frac{z-\zeta^-(\epsilon)}{z-\zeta^+(\epsilon) }
		\cdot \frac{f_\epsilon(z) - f_\epsilon(\zeta^+(\epsilon))}{f_\epsilon(z)-f_\epsilon(\zeta^-(\epsilon))}
		\right) -\epsilon \\
		&=  -\frac{1}{i\pi} \log \left(  1 - \frac{i\pi}{z} \epsilon - \frac{\pi^2}{2z^2} \epsilon^2+O(\epsilon^3) \right) +\frac{1}{i\pi} \log \left(  1-\frac{i\pi}{f(z)}\epsilon -\frac{\pi^2}{2f(z)^2}\epsilon^2+O(\epsilon^3)		\right) -\epsilon\\
	&=  \frac{\epsilon}{z} - \frac{\epsilon}{f(z)} - \epsilon + O(\epsilon^3)   = \epsilon A_0(z) + O(\epsilon^3).
	\end{align*}
	Here the constant in the $O$ still depends on $K \subset \C^*$.
	Let $\phi_\epsilon(z):=\frac{A(\epsilon,z)-\epsilon A_0(z)}{\epsilon^3}$. By Proposition \ref{prop:phianalytic}  $\phi_\eps$ is holomorphic on $\D(0,r)$. We have proved that
	for all compact $K \subset \C^*$, for all $z \in K$, and for all small $\epsilon \neq 0$ with $\mathrm{Re}(\epsilon)>0$, we have
	$|\phi_\epsilon(z)| \leq C_K$. By taking $K=\{ |z|=\frac{r}{2}  \}$ we therefore obtain the same
	estimate $|\phi_\epsilon(z)| \leq C_K$ for all $|z| \leq \frac{r}{2}$ because of the maximum modulus principle.
	This gives the desired uniformity.
\end{proof}

\begin{lem}
	If $\zeta^\pm(\epsilon)=\pm \frac{i\pi}{2}\epsilon+c_2 \epsilon^2+ c_3^\pm \epsilon^3+O(\epsilon^4)$,
	and $f(z)=z+z^2+z^3+bz^4+O(z^5)$, then
	$$C_0 = \frac{-3b\pi^3+2\pi^3+12c_2\pi + 12i(c_3^- - c_3^+)}{12\pi}$$
\end{lem}

\begin{proof}
	By repeating the computations from Lemma \ref{prop:1/z} and Lemma \ref{lem:1/z} with one additional order of significance, one obtains
    $$
    \frac{z-\zeta^+(\epsilon)}{z-\zeta^-(\epsilon)} =  1 - \frac{i\pi}{z} \epsilon - \frac{\pi^2}{2z^2} \epsilon^2 + \left(\frac{c_3^- - c_3^+}{z} - \frac{i \pi c_2}{z^2} + \frac{i \pi^3}{4 z^3}\right) \epsilon^3 + O(\epsilon^3),
    $$
    and
   $$
    \frac{f_\epsilon(z) - f_\epsilon(\zeta^+(\epsilon))}{f_\epsilon(z) - f_\epsilon(\zeta^-(\epsilon))} =  1 - \frac{i\pi}{f(z)} \epsilon - \frac{\pi^2}{2f^2(z)} \epsilon^2 + \left(\frac{i\pi^3}{4f^2(z)}+\frac{c_3^- - c_3^+}{f(z)} - \frac{i \pi c_2}{f^2(z)} + \frac{i \pi^3}{4 f^3(z)}\right) \epsilon^3 + O(\epsilon^4).
    $$
	
    Plugging these two equations into the formula for $A(\epsilon,z)$, and using the power series expansions of $\frac{1}{f(z)^j}$, for $j = 1, \ldots , 3$, one notices again that all terms involving negative powers of $z$ cancel, either by the argument used in the proof of the previous proposition, or by lengthy computations using
    $$
    c_2 = \frac{\pi^2}{8} - \frac{1}{4}.
    $$
    Summing the terms that do not depend on $z$ gives the desired result.
\end{proof}

\begin{lem}
	We have
	$$c_3^+= -c_3^- = \frac{1}{i\pi} \left(\frac{3}{16} + \frac{5\pi^4}{64}-\frac{b\pi^4}{16}-\frac{\pi^2}{4}  \right).$$
\end{lem}
	We will omit the proof, which is a long but direct computation, starting from the functional
	equation $f_\epsilon \circ \zeta^\pm(\eps) = \zeta^\pm \circ \tilde{g}(\eps)$ and identifying
	coefficients in powers of $\eps$.

	In particular, it follows that
	$$
    \begin{aligned}C_0 & = -\frac{b\pi^2}{4}-\frac{1}{4}+\frac{7\pi^2}{24} + \left(\frac{b\pi^2}{8}+\frac{1}{2}-\frac{5\pi^2}{32}-\frac{3}{8\pi^2} \right)\\
	& =\frac{-b\pi^2}{8}+\frac{13\pi^2}{96}-\frac{3}{8\pi^2}+\frac{1}{4}.
    \end{aligned}
    $$

\subsection{Convergence result} %{\black Let $(z_0,w_0)\in K\times K'\subset B_f\times B_g$ and $n$ sufficiently large.}
{\black
 %In section 4 we have proven the existence of parabolic curves $\zeta^\pm$ which are graphs over some petal $\mathcal{P}\subset\mathcal{B}_g$.
 For the rest of the section 5 we fix a compact subset $K\times K' \subset \mathcal{B}_f\times \mathcal{B}_g$  and a point $(z_0,w_0)\in K\times K'$. Moreover we assume that $n$ is sufficiently large so that $g^{n^2}(K')$ is contained in a petal $\mathcal{P}$ from Proposition \ref{prop:curve}.
Unless otherwise stated, all the constants appearing in estimates depend only on the compact $K\times K'$, but not on the point $(z_0,w_0)$ nor the integer $n$.

  }

Let $f_j(z):=f(z)+\frac{\pi^2}{4}w_{n^2+j}$, where $w_{n^2+j}:=g^{n^2+j}(w_0)$.
Let $z_j:=f_j \circ f_{j-1} \circ \ldots \circ f_1(z_0)$.
Let $F_{m,p}:=f_m \circ \ldots \circ f_{p+1}$, and let $\epsilon_j:=\sqrt{w_{n^2+j}}$.

\black

The strategy of the proof of Theorem \ref{thm:lav} is as follows: we will use \emph{approximate Fatou coordinates}
$\phi_n^{\iota/o}$ and prove that on some appropriate domains, $\phi_n^{\iota/o}$ converges locally uniformly to $\phi_f^{\iota/o}$
(with a known error term of order $\frac{1}{n}$).
Moreover, we will compute $\phi_n^\iota(z_0)$ and $\phi_n^o(z_{2n+1})$, again at a precision of order $\frac{1}{n}$.
This will allow us to compare accurately $z_{2n+1}$ and $\lcal_f(z_0) = (\phi_f^o)^{-1} \circ \phi_f^\iota(z_0)$.
This approach differs from \cite{ABDPR} in that approximate Fatou coordinates in \cite{ABDPR} were only used
at small scale near $0$, while here they are defined on a whole petal: this simplifies the comparison with the
actual Fatou coordinates $\phi_f^{\iota/o}$. The approach used here is strongly inspired by \cite{BSU}.

\color{black}

\begin{defi}
	Let
	$$Z_j^{i/o}:=\psi_{\epsilon_j}^{\iota/o}(z_j)=\frac{1}{i\pi} \log \frac{\zeta^+(\epsilon_{j})-z_j}{z_j-\zeta^-(\epsilon_{j})}\pm 1$$
\end{defi}

Observe that by definition of $A(\eps,z)$,
\begin{equation}
A(\epsilon_j,z_j)=Z_{j+1}^i - Z_j^i - \epsilon_j.
\end{equation}

\begin{prop}\label{prop:error1}
	We have
	$$\psi_{\epsilon_j}^{\iota/o}(z_0)=-\frac{\epsilon_j}{z_0}+O\left(\frac{1}{n^3}\right).$$
\end{prop}

\begin{proof}
	This follows from computations similar to those appearing in the proof of Proposition \ref{prop:phianalytic}
	(recall as well that $\epsilon_j = O\left(\frac{1}{n}\right)$).
\end{proof}

We now introduce approximate incoming Fatou coordinates:

\begin{defi}
	Let
	$$\phi_{n}^{\iota}(z_0):=\frac{1}{\epsilon_{n}} Z_{n} - \frac{1}{\epsilon_n} \sum_{j=1}^{n-1} \epsilon_j$$
\end{defi}

\black

Let $D_\eps$ be the disk of radius $\frac{1}{2}|\zeta^+(\eps)-\zeta^-(\eps)|$ centered at  $\frac{1}{2}(\zeta^+(\eps)+\zeta^-(\eps))$. Let $\mathcal{S}(\eps,r)$ be the union of the two disks of radius $r$ that both contain the points $\zeta^+(\eps),\zeta^-(\eps)$ on their boundary. Here $r$ will be a sufficiently small number, to be fixed in the paragraph before Lemma \ref{lem:translation}. The definition of $\mathcal{S}(\eps,r)$ of course only makes sense when the distance between $\zeta^+(\eps)$ and $\zeta^-(\eps)$ is less than $2r$, which once $r$ is fixed will be satisfied for $\epsilon$ sufficiently small. We note that the choice of $r$ will depend on the map $f$, but not on $\epsilon$.

The line $L_\eps$ through  $\zeta^+(\epsilon)$ and $\zeta^-(\epsilon)$ cuts the complex plane into the left half plane $H^\iota_\eps$ and the right half plane $H^o_\eps$. We define $\mathcal{S}^{\iota/ o}(\epsilon,r):= \mathcal{S}(\epsilon,r) \cap H^{\iota/o}_\eps$. The map $\psi_\epsilon^\iota$ maps the disk $D_\epsilon$ to the  strip $[1/2,3/2]\times i\mathbb{R}$.  The image of $\mathcal{S}^{\iota}(\epsilon,r)$ is bounded by two vertical lines, intersecting the real line in a point of the form $0 + O(\eps)$ and in the point $1$, see Figure \ref{fig:conformalmaps}. In particular we can find
$0<\alpha< \beta$, such that
$$
[\beta \eps,1]\times i\mathbb{R}\subset \psi_\epsilon^\iota(\mathcal{S}^{\iota}(\epsilon,r))\subset[\alpha \eps,1]\times i\mathbb{R}
$$
for all $\eps$. We define $\mathcal{S}^{\iota}(0,r):=\mathbb{D}(- r,r)$.

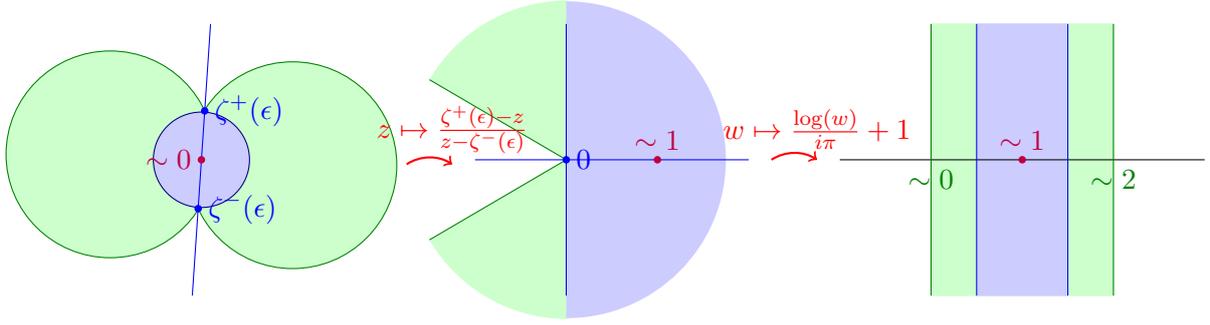
\begin{figure}[hbt!]
	\begin{center}
		\begin{tikzpicture}[scale=0.6]\label{fig:2}
		\filldraw[fill=green!20!white, draw=green!50!black] (-2,0.12) circle (65pt);
		\filldraw[fill=green!20!white, draw=green!50!black] (2,-0.12) circle (65pt);
		\filldraw[fill=blue!20!white, draw=blue!50!black] (0,0) circle (30pt);
		\draw[blue] (-0.2,-3) -- (0.2,3);
		\filldraw[purple] (0,0) circle (2pt) node[anchor=east] {$\sim 0$};
		\filldraw[blue] (0.072,1.08) circle (2pt) node[anchor=west] {$\zeta^+(\epsilon)$};
		\filldraw[blue] (-0.072,-1.08) circle (2pt) node[anchor=west] {$\zeta^-(\epsilon)$};

		\fill[green!20!white] (8,0) -- (5,1.77) arc (150:90:35mm) -- (8,0);
		\fill[green!20!white] (8,0) -- (5,-1.77) arc (-150:-90:35mm) -- (8,0);
		\fill[blue!20!white] (8,0) -- (8,-3.5) arc (-90:90:35mm) -- (8,0);
		\draw[green!50!black] (8,0) -- (5,1.77);
		\draw[green!50!black] (8,0) -- (5,-1.77);
		\draw[blue] (8,-3)  -- (8,3);
		\draw[blue] (6,0) -- (12,0);
		\filldraw[blue] (8,0) circle (2pt) node[anchor=west] {$0$};
		\filldraw[purple] (10,0) circle (2pt) node[anchor=south] {$\sim 1$};

		\filldraw[blue!20!white] (17,-3) rectangle (19,3);
		\filldraw[green!20!white] (16,-3) rectangle (17,3);
		\filldraw[green!20!white] (19,-3) rectangle (20,3);
		
		\draw[black] (14,0) -- (22,0);
		\draw[blue] (17,-3) -- (17,+3);
		\draw[blue] (19,-3) -- (19,+3);
		\draw[green!50!black] (16,-3) -- (16,+3);
		\draw[green!50!black] (20,-3) -- (20,+3);
		
		\filldraw[purple] (18,0) circle (2pt) node[anchor=south] {$\sim 1$};
		\filldraw[green!50!black] (16,0) circle (0pt) node[anchor=north] {$\sim 0$};
		\filldraw[green!50!black] (20,0) circle (0pt) node[anchor=north] {$\sim 2$};
		
		\draw[red, thick, ->] (12.5,0) .. controls (12.8,0.2)  and (13.2,0.2) .. (13.5,0) node[anchor=south] {$w \mapsto \frac{\log(w)}{i \pi} +1$};
		\draw[red, thick, ->] (4.5,-0.1) .. controls (4.8,0.1)  and (5.2,0.1) .. (5.5,-0.1) node[anchor=south] {$z \mapsto \frac{\zeta^+(\epsilon) - z}{z-\zeta^-(\epsilon)}$};
		\end{tikzpicture}
	\end{center}
	\caption{\emph{The sets $D_\epsilon \subset S(\epsilon,r)$ and their images under $\psi_\epsilon^\iota$.}}
\label{fig:conformalmaps}
\end{figure}

\medskip

Recall that $A(\eps,\cdot)$ is analytic on a small disk $\mathbb{D}(0,R)$ centered at the origin. Moreover there exists $R>0$ such that $|A(\eps,z)|\leq \frac{|\eps|}{12}$ for all $z\in \mathbb{D}(0,R)$ and $\eps$ in the petal $\mathcal{P}_{\delta}$ defined in Section 4. By taking smaller $R$ if necessary we my assume that $f$ is $1-$Lipschitz on $\mathbb{D}(- R,R)$.

Now let us assume that $r<<R$ is sufficiently small so that $\mathcal{S}(\eps,r)\subset \mathbb{D}(0,R)$ for all $\eps>0$, and note that for every compact set $K\subset \mathcal{B}_f$, there exist $n',\eps'>0$ so that $f^{n}(K)\subset  \mathcal{S}^{\iota}(\epsilon,r) \cap \mathbb{D}(- R,R)$ for all  $n\geq n'$ and all $\eps\geq\eps'$. We now fix this $r$.

\begin{lem}\label{lem:translation} Let $K\times K'\subset \mathcal{B}_f \times \mathcal{B}_g$ be a compact set. There exist $n_0,m_0>0$ such that for all $(z_0,w_0)\in K\times K'$ and all $n>n_0$ we have:
\begin{enumerate}
\item $z_j\in \mathcal{S}^{\iota}(\epsilon_j,r)\cup D_{\eps_j}$ for all $m_0\leq j\leq n-1$,
\item $z_j\in D_{\eps_j}$ for all $\frac{2n}{3}\leq j\leq n-1$,
\item If $z_k\in \mathcal{S}(\epsilon_k,r)$ for all $m_0\leq k\leq j$, then $|\im (\psi^\iota_{\eps_{j+1}}(z_{j+1}))|<1$,
\end{enumerate}
where $\epsilon_j:=\sqrt{w_{n^2+j}}=\frac{1}{n}+O\left(\frac{j}{n^3}\right)$.
\end{lem}
\begin{proof} There exists $m_0>0$ so that $f^{m_0}(K)\subset S^{\iota}(0,r)$. Let $n_0$ be sufficiently large so  that for all $(z_0,w_0)\in K\times K'$ we have:
\begin{enumerate}[label=(\roman*)]
\item  $\beta |\eps_{m_0}|<\re(\psi_{\eps_{m_0}}^\iota(z_{m_0}))<\frac{1}{6}$
\item $|\im(\psi_{\eps_{m_0}}^\iota(z_{m_0}))|<\frac{1}{2}$
\item $|\eps_j|<2\re(\eps_j)<R$ for $0\leq j\leq 2n+1$
\end{enumerate}
Indeed, (i) and (ii)  follow from the equality $\psi_\eps^\iota(z)=-\frac{\eps}{z}+O(\eps^3)$  and (iii) follows from the fact that $\epsilon_j=\frac{1}{n}+O\left(\frac{j}{n^3}\right)$. Note that the constants in $O$ depend only on the compact $K\times K'$ and not on $n$ or $j$.

Recall that by our assumption $\mathcal{S}^{\iota}(\epsilon_j,r)\subset \mathbb{D}(0,R)$ for all $j$ and that
 $|A(\eps_j,z)|\leq \frac{|\eps_j|}{12}$ for all $z\in\mathbb{D}(0,R)$. By (i) we have  $z_{m_0}\in\mathcal{S}^{\iota}(\epsilon_{m_0},r)$ and observe that for $z_{j}\in\mathcal{S}^{\iota}(\epsilon_{j},r)$ we have
$$
\frac{5}{6}\re(\eps_{j})<\re(\psi^{\iota}_{\eps_{j+1}}(z_{j+1})-\psi^{\iota}_{\eps_{j}}(z_{j}))<\frac{7}{6}\re(\eps_{j}).
$$

It follows that
\begin{equation}\label{translation}
\beta |\eps_{m_0}|+\frac{5}{6}\sum_{k=m_0}^{j-1}\re(\eps_k)<\re(\psi^\iota_{\eps_j}(z_j))< \frac{1}{6} +\frac{7}{6}\sum_{k=m_0}^{j-1}\re(\eps_k),
\end{equation}
and since $\re(\eps_k)=\frac{1}{n}+O\left(\frac{k}{n^3}\right)$ we have
$$
\beta |\eps_{m_0}| <\re(\psi^{\iota}_{\eps_j}(z_j))< \frac{3}{2}
$$
for all $m_0\leq j\leq n-1$ as long as $n$ is sufficiently large. This proves $(1)$.

For $(2)$ observe that
$$-\frac{1}{2}<-1 +\frac{5}{6}\sum_{k=0}^{j-1}\re(\eps_k)$$
for all $\frac{2}{3}n\leq j\leq n-1$ as long as n is sufficiently large.

Finally for $(3)$ observe that equation \eqref{translation} implies that $z_k\in \mathcal{S}(\epsilon_k,r)$ for all $0\leq k\leq j$ can only hold for some $j<3n$. By (ii) and (iii) we have
$$
\im(\eps_{j}) -\frac{1}{6}\re(\eps_{j})<\im(\psi^{\iota}_{\eps_{j+1}}(z_{j+1})-\psi^{\iota}_{\eps_{j}}(z_{j}))<\im(\eps_{j}) +\frac{1}{6}\re(\eps_{j}),
$$
for  $z_{j}\in\mathcal{S}^{\iota}(\epsilon_{j},r)$,
hence
$$
-\frac{1}{2} +\sum_{k=m_0}^{j-1}\im(\eps_{k}) -\frac{1}{6}\re(\eps_{k})<\im(\psi_{\eps_j}(z_j))<\frac{1}{2} +\sum_{k=m_0}^{j-1}\im(\eps_{k}) +\frac{1}{6}\re(\eps_{k}).
$$
Since $\im(\eps_k)=O\left(\frac{k}{n^3}\right)$ we can conclude that $(3)$ holds as long as $n$ is sufficiently large.
\end{proof}

\begin{lem}\label{lem:zj-yj} For $0\leq j\leq n-1$ we have
$$
z_j-f^j(z_0)=O\left(\frac{j}{n^2}\right).
$$
\end{lem}

\begin{proof}
Let $m_0$ be as in Lemma \ref{lem:translation}. Since $m_0$ is independent from $n$ it is easy to see that for all $0\leq j\leq m_0$ we have
$z_{j}-f^{j}(z_0)=\frac{\pi^2}{4}\sum_{k=0}^{j-1}\left(\eps_k^2+ O\left(\frac{1}{n^3}\right)\right)=O\left(\frac{j}{n^2}\right)$.

Let $V_\eps:=\{z\in \mathcal{S}^\iota(\eps,r)\mid |\im(\psi^{\iota}_{\eps}(z))|<1\}$ and observe that
$V_\eps\backslash D_\eps\subset \mathbb{D}(- R,R)$ for all sufficiently small $\eps$. Since by our assumption  $f$ is 1-Lipschitz on $\mathbb{D}(- R,R)$, it follows by (1) in Lemma \ref{lem:translation} that for all $m_0< j<\frac{2}{3}n$ we have
\begin{align*}
|z_j-f^j(z_0)|&<|z_{j-1}-f^{j-1}(z_0)|+\frac{\pi^2}{4}\eps_{j-1}^2 \\
&< |z_{m_0}-f^{m_0}(z_0)|+\frac{\pi^2}{4}\sum_{k=m_0}^{j-1}\eps_k^2 \\
&=\frac{\pi^2}{4}\sum_{k=0}^{j-1}\left(\frac{1}{n^2}+O\left(\frac{1}{n^3}\right)\right) =O\left(\frac{j}{n^2}\right).
\end{align*}
Finally for $\frac{2}{3}n\leq j\leq n-1$, by item $(2)$ of Lemma \ref{lem:translation}, we have
$z_j\in D_{\eps_j}$, and in particular $z_j=O\left(\frac{1}{n}\right)$.  It follows that
$$
z_j-f^j(z_0)=O\left(\frac{1}{n},\frac{1}{j}\right)=O\left(\frac{j}{n^2}\right),
$$
where the last equality follows from the fact that $\frac{2}{3}n\leq j\leq n-1$. This completes the proof.
\end{proof}

\color{black}

\begin{lem}\label{lem:sumA}
%	Let $\gamma_j := j+ \sum_{k=0}^{j-1} A_0(f^k(z_0))$.
	 We have:
	$$\sum_{j=0}^{n-1} A_0(z_j)-A_0(f^j(z_0))= (b-1)\sum_{j=0}^{n-1} z_j^2-f^j(z_0)^2 + O\left(\frac{\black \log n}{n^2} \right)
	$$
\end{lem}

\begin{proof}
	Recall that $f(z)=z+z^2+z^3+bz^4+O(z^5)$ and $A_0(z)=-\frac{1}{f(z)}+\frac{1}{z}-1$.
	An elementary computation gives $A_0(0)=A_0'(0)=0$, and
	$A_0''(0)=2(b-1).$

	To simplify the notations, let $y_j:=f^j(z_0)$. We have:
	\begin{align*}
	A_0(z_j)-A_0(y_j) &= A_0'(y_j)(z_j-y_j)+\frac{1}{2}A_0''(y_j)(z_j-y_j)^2 + O\left((z_j-y_j)^3 \right)\\
	&=(y_j A_0''(0)+O(y_j^2)) (z_j-y_j)+ \frac{1}{2}(A_0''(0)+O(y_j))(z_j-y_j)^2+ O\left((z_j-y_j)^3 \right).\\
	\end{align*}
	By Lemma \ref{lem:zj-yj} we have $z_j-y_j = O\left(\frac{j}{n^2}\right)$, hence
	\begin{align*}
	A_0(z_j)-A_0(y_j) &=y_j A_0''(0)(z_j-y_j)+\frac{1}{2}A_0''(0) (z_j-y_j)^2+
	O\left(\frac{1}{jn^2}, \frac{j}{n^4}, \frac{j^3}{n^6} \right) \\
	&=(b-1) \left(2  y_j (z_j-y_j) +  (z_j-y_j)^2 \right) + 	O\left(\frac{1}{jn^2}, \frac{j}{n^4}, \frac{j^3}{n^6} \right) \\
	&=(b-1) (z_j^2-y_j^2)  + 	O\left(\frac{1}{jn^2}, \frac{j}{n^4}, \frac{j^3}{n^6} \right) \\
	\end{align*}
	It follows that
	$$\sum_{j=0}^{n-1} A_0(z_j)-A_0(y_j)= (b-1)\sum_{j=0}^{n-1} z_j^2-y_j^2 + O\left(\frac{\log n}{n^2}. \right)
	$$
\end{proof}

\begin{lem}\label{lem:alphaj}
	For $0 \leq j \leq n-1$, let
	\begin{equation}
		\gamma_j:=j+\sum_{k=0}^{j-1} A_0(f^k(z_0)) \text{,\quad  and \quad } x_j:=\sum_{k=0}^{j-1} \eps_k + A(\eps_k,z_k).
	\end{equation}
	Then
	$$x_j = \frac{\gamma_j}{n}+ O\left(\frac{j^2}{n^3} \right) = \frac{j}{n}+O\left(\frac{1}{n}\right).$$
	In particular, there exists $k \in \N$ independent from $n$ such that for all $k \leq j \leq n-k$,
	$$\alpha_j:=\cot\left(\frac{\pi}{2}x_j\right) $$
	is well-defined and strictly positive.
\end{lem}

\begin{proof}
	According to Lemma \ref{lem:zj-yj}, for $0 \leq j \leq n-1$ we have
	that $z_j - f^j(z_0)=O\left(\frac{j}{n^2} \right)$. In particular, $z_j = O(1)$.
	By Proposition \ref{prop:estimateA}, we have for every $0 \leq k \leq n-1$:
	\begin{equation}
	A(\eps_k,z_k) = \eps_k A_0(z_k) + O\left(\eps_k^3, z_k \eps_k^3\right) = \eps_k A_0(z_k) + O\left(\frac{1}{n^3}\right)
	\end{equation}
	(indeed, by Lemma \ref{lem:zj-yj}, $z_k = f^k(z_0)+O(\frac{k}{n^2})$, so in particular
	$z_k=O(1)$).
	By Lemma \ref{lem:epsj}, $\eps_k=\frac{1}{n}+O\left(\frac{k}{n^3}\right)$, hence
	\begin{align*}
	x_j&=\sum_{k=0}^{j-1} \eps_k + A(\eps_k,z_k)\\
	&= \sum_{k=0}^{j-1} \frac{1}{n}+ \frac{1}{n} A_0(z_k)+O\left(\frac{k}{n^3}\right)\\
	&=\frac{j}{n}+\frac{1}{n}\sum_{k=0}^{j-1} A_0(f^k(z_0)) + O\left( \frac{k}{n^{\black 2}},A_0'(f^k(z_0)) (z_k-f^k(z_0) \right) \\
	&=\frac{j}{n}+\frac{1}{n}\sum_{k=0}^{j-1} A_0(f^k(z_0)) + O\left(\frac{k}{n^{\black 2}},\frac{1}{k}\cdot \frac{k}{n^{\black 2}} \right) \\
	&=\frac{\gamma_j}{n}+O\left(\frac{j^2}{n^3}\right).
	\end{align*}
	Since $\gamma_j=j+O(1)$, we also have
	$$\frac{\gamma_j}{n}=\frac{j}{n}+O\left(\frac{1}{n}\right).$$

	Finally, the last assertion follows from the preceding equality and the fact that for $x \in (0,\frac{\pi}{2})$,
	$\cot(x)>0$.
\end{proof}

\begin{lem}\label{lem:riemannsum}
	Let $u(x):=\frac{2}{\pi}\tan(\frac{\pi}{2}x)$,
	$\Phi(x)=\frac{x^2-u(x)^2}{x^2 u(x)^2}$ and $\beta_j:=\frac{2n}{\pi\alpha_j}$.
	We have :
	$$\frac{\gamma_j^2-\beta_j^2}{\gamma_j^2 \beta_j^2} = \frac{1}{n^2}  \Phi(x_j) +O\left(\frac{1}{jn^2 }\right).$$
\end{lem}

\begin{proof}
	We have
	\begin{align*}
		\frac{\gamma_j^2-\beta_j^2}{\gamma_j^2 \beta_j^2} &=\frac{n^2}{n^4} \frac{\frac{\gamma_j^2}{n^2}-u(x_j)^2}{u(x_j)^2 \frac{\gamma_j^2}{n^2}} \\
	\end{align*}
	Now recall that by Lemma \ref{lem:alphaj}, $\frac{\gamma_j}{n}=x_j+O(\frac{j^2}{n^3})$, so that
	$\frac{\gamma_j^2}{n^2}=x_j^2+O\left(\frac{j^3}{n^4} \right)$. So:
	\begin{align*}
		\frac{\gamma_j^2-\beta_j^2}{\gamma_j^2 \beta_j^2} &=\frac{1}{n^2}
		\frac{x_j^2 - u(x_j)^2+ O(\frac{j^3}{n^4})}{u(x_j^2)(x_j^2+O(\frac{j^3}{n^4}))}\\
		&=\frac{1}{n^2} \frac{x_j^2 - u(x_j)^2}{u(x_j^2)(x_j^2+O(\frac{j^3}{n^4}))}
		+ O\left(\frac{j^3}{x_j^2 u(x_j)^2 n^6}  \right)
	\end{align*}
	and note that
	\begin{align*}
		\frac{1}{x_j^2 u(x_j)^2}=O\left(\frac{1}{x_j^4}\right)=O\left(\frac{n^4}{j^4} \right).
	\end{align*}
	Therefore
	\begin{align*}
		\frac{\gamma_j^2-\beta_j^2}{\gamma_j^2 \beta_j^2} &=\frac{1}{n^2}\frac{x_j^2 - u(x_j)^2}{u(x_j^2)(x_j^2+O(\frac{j^3}{n^4}))} + O\left(\frac{1}{jn^2} \right)\\
		&=\frac{1}{n^2} \frac{x_j^2 - u(x_j)^2}{u(x_j^2) x_j^2 (1+O(\frac{j^3}{n^4 x_j^2}))}
		+ O\left(\frac{1}{jn^2} \right)\\
		&=\frac{\Phi(x_j)}{n^2} \left(1+O\left(\frac{j}{n^2}\right)\right)+ O\left(\frac{1}{jn^2} \right)\\
		&=\frac{\Phi(x_j)}{n^2}+O\left(\frac{1}{jn^2}\right).
	\end{align*}
	Note that in the last line, we used the fact that $\Phi$ has only removable singularities at $x=0$
	and $x=1$, so that $\Phi(x_j)=O(1)$.
\end{proof}

\begin{prop}\label{prop:c1}
	There exists a universal constant $C_1 \in \R$ such that
	$$\sum_{j=0}^{n-1} A_0(z_j)-A_0(f^j(z)) = \frac{C_1(b-1)}{n}+O\left(\frac{\black \log n}{n^2}\right).$$
	More precisely,
	$C_1:=\int_0^{1} \Phi(x)dx=\frac{1}{4}(4-\pi^2).$
\end{prop}

\begin{proof}
	We have, for $0 \leq j \leq n-1$:
	\begin{equation*}
	z_j = \psi_{\epsilon_j}^{-1}(Z_j^\iota) = -\frac{\pi}{2n} \cot(\frac{\pi}{2}Z_j^\iota)+O\left(\frac{1}{n^2}\right)
	\end{equation*}
	and
	\begin{equation*}
	Z_j^\iota = Z_0^\iota + \sum_{k=0}^{j-1} \eps_k + A(\eps_k,z_k).
	\end{equation*}
	Recalling the notation $x_j:=\sum_{k=0}^{j-1} \eps_k + A(\eps_k,z_k) $ and
	$\alpha_j:=\cot(\frac{\pi}{2}x_j)$ from Lemma \ref{lem:alphaj}, and
	using the trigonometry formula $\cot(a+b)=\frac{\cot a \cot b - 1}{\cot a + \cot b}$, we therefore obtain:
	\begin{equation}
	z_j=-\frac{\pi}{2n} \frac{\cot(\frac{\pi}{2}Z_0^\iota) \alpha_j-1}{\alpha_j+\cot({\black \frac{\pi}{2}}Z_0^\iota)} + O\left(\frac{1}{n^2}\right)
	\end{equation}
	Let $k$ be as in Lemma \ref{lem:alphaj}, so that $\alpha_j>0$ for $k \leq j \leq n-k$.
	We have $\cot(\frac{\pi}{2}Z_0^\iota) = -\frac{2n}{\pi}z_0+ O\left(\frac{1}{n}\right) $, so that
	\begin{equation}
	z_j=-\frac{\pi}{2n} \frac{\cot(\frac{\pi}{2}Z_0^\iota) \alpha_j}{\alpha_j+\cot({\black \frac{\pi}{2}}Z_0^\iota)}
	+ O\left(\frac{1}{n^2}\right) = \frac{z_0 \alpha_j}{\alpha_j - z_0 \frac{2n}{\pi}}+ O\left(\frac{1}{n^2}\right)
	\end{equation}
	Finally, with $\beta_j:=\frac{2n}{\pi\alpha_j}$, we get
	\begin{equation}
	z_j = -\frac{1}{-\frac{1}{z_0}+\beta_j}+ O\left(\frac{1}{n^2}\right).
	\end{equation}
	On the other hand, from the definition of $A_0$ it follows that
	$\sum_{k=0}^{j-1} A_0(f^k(z_0))=\frac{1}{z_0}-\frac{1}{f^j(z_0)}-j$, which we may rewrite as
	\begin{equation}
		f^j(z_0) = -\frac{1}{-\frac{1}{z_0}+\gamma_j}.
	\end{equation}
	Therefore:
	\begin{equation}
	z_j-f^j(z_0) = \frac{\beta_j-\gamma_j}{\left(-\frac{1}{z_0}+\beta_j\right)\left(-\frac{1}{z_0}+\gamma_j\right)} + O\left(\frac{1}{n^2}\right)
	\end{equation}

	Now note that
	$j=O(\beta_j)$: indeed,  $\frac{1}{\beta_j}=O\left(\frac{\cot \left( \frac{j\pi}{n}+O(\frac{1}{n}) \right)}{n}\right)=O\left(\frac{n/j}{n}\right)=O(\frac{1}{j})$. Therefore:
	\begin{align*}
	\frac{1}{(-\frac{1}{z_0}+\gamma_j)(-\frac{1}{z_0}+\beta_j)}&= \frac{1}{(\gamma_j+O(1))(\beta_j+O(1))}\\
	&= \frac{1}{\gamma_j\beta_j + O(\beta_j)}\\
	&= \frac{1}{\gamma_j\beta_j} + O\left( \frac{1}{\gamma_j^2\beta_j}\right). \\
	\end{align*}

	Therefore, setting $y_j:=f^j(z_0)$:
	\begin{align*}
		z_j^2 - y_j^2 &= (z_j-y_j)(z_j+y_j)\\
		&= \left(\frac{\beta_j-\gamma_j}{\beta_j \gamma_j}+O\left(\frac{\beta_j-\gamma_j}{\gamma_j^2 \beta_j} \right) \right) \left(-\frac{\beta_j+\gamma_j}{\beta_j \gamma_j}+O\left(\frac{\beta_j+\gamma_j}{\gamma_j^2 \beta_j} \right) \right)\\
		&= \frac{\gamma_j^2-\beta_j^2}{\beta_j^2 \gamma_j^2} + O\left(\frac{\beta_j^2 - \gamma_j^2}{\beta_j^2 \gamma_j^3} \right)\\
%		&= \frac{\gamma_j^2-\beta_j^2}{\beta_j^2 \gamma_j^2}  + O\left(\frac{(z_j-y_j)^2}{j} \right)\\
%		&= \frac{\gamma_j^2-\beta_j^2}{\beta_j^2 \gamma_j^2} + O\left(\frac{j}{n^4} \right)
%		\text{ by Lemma \ref{lem:zj-yj}} \\
		&=\frac{1}{n^2} \Phi(x_j) + O\left(\frac{1}{jn^2} \right) \text{ by Lemma \ref{lem:riemannsum}}.
	\end{align*}
	Therefore by Lemma \ref{lem:sumA},
	\begin{align*}
		\sum_{j=0}^{n-1} A_0(z_j)-A_0(y_j)&=\left(\frac{b-1}{n^2} \sum_{j=0}^{n-1} \Phi(x_j) \right) +
		O\left(\frac{\black \log n}{n^2} \right) \\
		&=\frac{b-1}{n} \int_0^1 \Phi(x) dx+ O\left(\frac{\black \log n}{n^2}\right).
	\end{align*}
	In the last equality, we recognize a Riemann sum with subdivision $(x_j)_{0 \leq j \leq n-1}$.
	Finally, we have
	
\begin{equation*}
	\int_0^1 \Phi(x) dx=\frac{\pi}{2}\int_0^{\frac{\pi}{2}} \cot^2{t}-\frac{1}{t^2} dt=-\frac{\pi}{2}\left[ \left(\cot{t}-\frac{1}{t}\right)+t\right]_{0}^{\frac{\pi}{2}} = 1-\frac{\pi^2}{4}.
\end{equation*}	

\end{proof}

\subsubsection{Incoming part}

The following error estimate is one of the two crucial estimates that we will obtain in this section:
it measures accurately how close $\phi_n^\iota$ is to the incoming Fatou coordinate $\phi^\iota_{\black f}$.
\black  {This estimate differs from those obtained in \cite{ABDPR}, in that we compare $\phi_n^\iota$ with
	$\phi_f^\iota$ on a definite  region of $\bcal_f$ (independent from $n$), instead of comparing the two at small scale near the origin, compare with \cite[Property 1 p. 10]{ABDPR}. Moreover, the point of Proposition \ref{prop:incoming} is to push the precision of the estimate further and obtain the first error term $\frac{E^\iota(z_0)}{n}$, which cannot be easily obtained from the computations in \cite{ABDPR}.}

\begin{prop}\label{prop:incoming}
	We have
	$$\phi_n^{\iota}(z_0)=\phi^{\iota}_{\black f}(z_0)+\frac{E^{\iota}(z_0)}{n}+O\left(\frac{\black \log n}{n^2}\right)$$
	where $E^\iota(z) :=  C_0 + (C_1-1)(b-1)+\frac{1}{2}\phi^\iota(z_0)$.
\end{prop}

\begin{proof}
	Recall that by definition,
	$$\phi^\iota_{\black f}(z)=\lim_{n \to \infty} -\frac{1}{f^n(z)}-n =\limn -\frac{1}{z}+\sum_{j=0}^{n-1} A_0(f^j(z)).$$
	
	Similarly, we have:
	\begin{align*}
		\sum_{j=0}^{n-1} A(\epsilon_j, z_j) = \sum_{j=0}^{n-1} Z_{j+1}^{\iota}-Z_j^\iota- \epsilon_j = Z_n^\iota - Z_0^\iota - \sum_{j=0}^{n-1} \epsilon_j,
	\end{align*}
	and thus
	$$\phi_n^\iota(z_0)=\frac{Z_n^\iota}{\epsilon_n}- \frac{1}{\epsilon_n} \sum_{j=0}^{n-1} \epsilon_j= \frac{Z_0^\iota}{\epsilon_n} + \frac{1}{\epsilon_n}\sum_{j=0}^{n-1} A(\epsilon_j,z_j). $$
	Therefore:
	\begin{equation}
		\phi_n^\iota(z_0)-\phi^\iota_{\black f}(z_0) = E_1 + E_2 + E_3,
	\end{equation}
	where
	\begin{align*}
		E_1&:=\frac{Z_0^\iota}{\epsilon_n} + \frac{1}{z_0}\\
		E_2&:=\frac{1}{\epsilon_n}\sum_{j=0}^{n-1} A(\epsilon_j,z_j) - \sum_{j=0}^{n-1} A_0(f^j(z_0))\\
		E_3&:=-\sum_{j=n}^\infty A_0(f^j(z_0))\\
	\end{align*}

	We will now estimate each of the error terms $E_i$ separately. For $j \in \N$, we set
	$y_j:=f^j(z_0)$.

\begin{lem}
	We have $E_1=-\frac{1}{2nz_0}+O\left(\frac{1}{n^2}\right)$.
\end{lem}

\begin{proof}[Proof of Lemma]
	We have
	\begin{align*}
		\frac{Z_0^\iota}{\epsilon_n}&=\frac{1}{\epsilon_n}\psi_{\epsilon_0}(z_0)\\
		&=-\frac{\epsilon_0}{\epsilon_n z_0} + O\left(\frac{\epsilon_0^3}{\epsilon_n}\right) \quad \text{ by Prop. \ref{prop:error1}}\\
		&=-\frac{1}{z_0}\sqrt{ \frac{n^2+n+O(1)}{n^2+O(1)} } + O\left(\frac{1}{n^2}\right)\\
		&=-\frac{1}{z_0}-\frac{1}{2nz_0}+O\left(\frac{1}{n^2}\right).
	\end{align*}

\end{proof}

\begin{lem}
	We have $E_2=\frac{1}{n} \left(\frac{1}{2z_0}+\frac{1}{2}\phi^\iota_{\black f}(z_0)+ C_0+C_1(b-1) \right)+ O\left( \frac{\black \log n}{n^2}\right).$
\end{lem}

\begin{proof}[Proof of Lemma]
	Recall that we have
	$$A(\epsilon,z)=\epsilon A_0(z) + C_0 \epsilon^3 + O(z\epsilon^3, \epsilon^4),$$
	so that
	\begin{align*}
		E_2 &= \frac{1}{\epsilon_n}\sum_{j=0}^{n-1} A(\epsilon_j,z_j) - \sum_{j=0}^{n-1} A_0(y_j)\\
		&= \frac{1}{\epsilon_n}\sum_{j=0}^{n-1} \epsilon_j A_0(z_j) + C_0 \epsilon_j^3 + O(\frac{z_j}{n^3})
		- \epsilon_n A_0(y_j).
	\end{align*}	
	Therefore:
	\begin{align*}
		E_2&=\left( \sum_{j=0}^{n-1} \frac{\epsilon_j}{\epsilon_n} A_0(z_j) -  A_0(y_j) \right)
		+ \left(\sum_{j=0}^{n-1} C_0 \frac{\epsilon_j^3}{\epsilon_n} + O(\frac{z_j}{n^2})   \right),\\
	\end{align*}
	
	and
	\begin{align*}
		\sum_{j=0}^{n-1} C_0 \frac{\epsilon_j^3}{\epsilon_n} + O\left(\frac{z_j}{n^2}\right)&= C_0 \sum_{j=0}^{n-1} \frac{1}{n^2}+O\left(\frac{1}{n^2 j}\right)\\
		&=\frac{C_0}{n}+O\left(\frac{\log n}{n^2}\right).
	\end{align*}

	On the other hand, we have
	\begin{equation}\label{eq:e21}
	  \sum_{j=0}^{n-1} \frac{\epsilon_j}{\epsilon_n} A_0(z_j) -  A_0(y_j) = \sum_{j=0}^{n-1} \left(\frac{\epsilon_j}{\epsilon_n}-1\right) A_0(z_j) +\sum_{j=0}^{n-1} A_0(z_j)-A_0(y_j).
	\end{equation}	
	Now note that
	\begin{align*}
	\sum_{j=0}^{n-1} \left(\frac{\epsilon_j}{\epsilon_n}-1\right) A_0(z_j)&=
	\sum_{j=0}^{n-1} \left(\frac{\epsilon_j}{\epsilon_n}-1\right) A_0(y_j)
	+ \sum_{j=0}^{n-1} \left(\frac{\epsilon_j}{\epsilon_n}-1\right) \left(A_0(z_j)-A_0(y_j)\right),
	\end{align*}
	and that
	\begin{align*}
	\left| \sum_{j=0}^{n-1}\left(\frac{\epsilon_j}{\epsilon_n}-1\right) (A_0(z_j)-A_0(y_j)) \right|
	&\leq \max_{0 \leq j \leq n-1} \left|1-\frac{\eps_j}{\eps_n} \right| \cdot \sum_{j=0}^{n-1} \left| A_0(z_j)-A_0(y_j)\right|\\
	&=O\left(\frac{1}{n^2} \right),
	\end{align*}
	by Lemmas  \ref{lem:epsj} and Proposition \ref{prop:c1}.
	Another consequence of Lemma \ref{lem:epsj} is that 	
	\begin{equation}\label{eq:e22}
		\frac{\eps_j}{\eps_n}-1 = \frac{1}{2n} \left(1-\frac{j}{n}+O\left(\frac{1}{n}\right)\right).
	\end{equation}
	Therefore, by \eqref{eq:e21}, \eqref{eq:e22} and Proposition \ref{prop:c1}:
	\begin{align*}
	 \sum_{j=0}^{n-1} \frac{\epsilon_j}{\epsilon_n} A_0(z_j) -  A_0(y_j)
	 &=\frac{C_1(b-1)}{n}+\frac{1}{2n} \sum_{j=0}^{n-1} A_0(y_j) + O\left(\frac{\black \log n}{n^2}\right)\\
	&=\frac{C_1 (b-1)}{n}+\left( \dfrac{1/z_0+\phi_{{\black f}}^{{\black \iota}}(z_0)}{2n}\right) + O\left(\frac{\black \log n}{n^2}\right).\\
	\end{align*}

	Therefore, as announced, we have:
	$$E_2=\frac{1}{n} \left(\frac{1}{2z_0}+\frac{1}{2}\phi^\iota_f(z_0)+  C_0 + C_1 (b-1) \right)+ O\left( \frac{\black \log n}{n^2}\right).$$
\end{proof}

\begin{lem}
	We have $E_3=\frac{1-b}{n}+O(\frac{1}{n^2})$.
\end{lem}

\begin{proof}
	By explicit computations we have $A_0(z)=(b-1)z^2+O(z^3)$,
	so that $A_0(y_j) = (b-1)j^{-2}+O(j^{-3})$.
	Therefore:
	$$E_3 = (1-b)\sum_{j=n}^\infty j^{-2}+O(j^{-3})$$
	and $\sum_{j=n}^\infty j^{-3} =O( \int_n^\infty \frac{dx}{x^3})=O(\frac{1}{n^2})$.
	Similarly, $\sum_{j=n}^\infty j^{-2} \sim \int_n^\infty \frac{dx}{x^2} = \frac{1}{n}$,
	so that $E_3=\frac{1-b}{n} + O(\frac{1}{n^2})$.
\end{proof}

Finally, putting together the three preceding lemmas, the proof of Proposition \ref{prop:incoming} is finished.

\end{proof}

\subsubsection{Outgoing part}

We will now work to obtain estimates for the outgoing part of the orbit, that is,
for $n \leq j \leq 2n+1$. The method is largely similar to the incoming case.
Recall that the estimates we obtain only depend on the chosen compact set $K\subset \mathcal{B}_f$.

We will first need a rough preliminary estimate on boundedness of $z_{2n+1}$. Of course, by \cite{ABDPR}, we know that $z_{2n+1}$ converges
to $\mathcal{L}(z_0)$, and we could deduce this preliminary estimate from there. However, we prefer to present here a direct argument, so that the proof of Theorem \ref{thm:lav} remains self-contained.

\begin{prop}\label{prop:z2n+1bounded}
	There exists $k \in \N$ (independent from $n$) such that $z_{2n+1-k}$ belongs to a repelling petal
	$\D(r,r)$ for $f$. In particular,
	$z_{2n+1}=O(1)$.
\end{prop}

\begin{proof}
	Recall that by Proposition \ref{prop:incoming}, we have that
	$$\phi_n^\iota(z):=\frac{Z_n^\iota}{\eps_n}-\frac{1}{\eps_n} \sum_{j=0}^{n-1} \eps_j = \phi^\iota(z_0)+o(1)=O(1).$$
	In particular,
	$$Z_n^\iota = \left(\sum_{j=0}^{n-1} \eps_j \right) + O(\eps_n)=1+O\left(\frac{1}{n}\right)$$
	and therefore
	$Z_n^o=-1+O\left(\frac{1}{n}\right).$
	
	Let $R_n$ denote the rectangle defined by the conditions $-1-\frac{C}{n}\leq \re(Z) \leq -\frac{3}{n}$
	and $-1 \leq \im(Z) \leq 1$, where $C>0$ is a constant chosen large enough that $Z_n^o \in R_n$.
	Let
	\begin{equation}
		j_n:=\max\{k \leq 2n+1 : Z_k^o \in R_n  \}.
	\end{equation}
	
	Recall that for $j \leq 2n$, we have $Z_{j+1}^o = Z_j^o + A(\eps_j,z_j)$, and
	that
	by Proposition \ref{prop:estimateA}, we have
	\begin{equation}
		A(\eps_k, z_k) = \eps_k A_0(z_k)+O\left(\eps_k^3, \eps_k^3 z_k\right) = O\left( \eps_k z_k^2 \right).
	\end{equation}
	Moreover, for $n \leq j \leq j_n$, we have $Z_j^o = -\frac{\pi}{2n} \cot\left(\frac{\pi}{2}Z_j^o\right)
	+O\left( \frac{1}{n^2}\right)$, and therefore there exists a constant $C>0$ such that for all $n \leq j \leq j_n$,
	\begin{equation}
		|A(\eps_j, z_j)| \leq \frac{C'}{n^3} \left| \cot \left(\frac{\pi}{2}Z_j^o\right)\right|^2 \leq \frac{C}{|Z_j|^2 n^3},
	\end{equation}
	and thus
	\begin{equation}\label{eq:controleZo}
		\left|Z_j^o - Z_n^o - \sum_{k=n}^{j-1} \eps_k \right| \leq \frac{C}{n^3} \sum_{k=n}^{j-1} \frac{1}{|Z_k|^2}.
	\end{equation}
	From \eqref{eq:controleZo}, we can prove inductively on $j$ that for $n \leq j \leq j_n$,
	$\left|Z_j^o - Z_n^o - \sum_{k=n}^{j-1} \eps_k \right| = O\left(\frac{1}{n}\right)$ and hence
	$j_n=2n+O(1)$.

	Let $r>0$ small enough such that $\D(r,r)$ is a repelling petal for $f$.
	By the argument above and the definition of $R_n$, we have that $Z_{j_n}^o = O\left(\frac{1}{n}\right)$, so that
	$$z_{2n+1-k} = -\frac{\pi}{2} \cot\left(\frac{\pi}{2} Z_{2n+1-k}^o\right) + O\left(\frac{1}{n^2}\right)  = \frac{1}{k+O(1)}.$$
	Therefore, we can find some $k$ bounded independently from $n$ such that $z_{2n+1-k} \in \D(r,r)$.

\end{proof}

We now introduce approximate outgoing Fatou coordinates:

\begin{defi}
	For $n \leq m \leq 2n+1$, let
	$$\phi_n^o(z_{m}):=\frac{Z_n^o}{\epsilon_n} + \frac{1}{\epsilon_n} \sum_{j=n}^{m-1} \epsilon_j.$$
\end{defi}

\begin{lem}
	We have
	$$\phi_n^o(z_{m})=\frac{Z_{m}^o}{\epsilon_n} - \frac{1}{\epsilon_n}\sum_{j=n}^{m-1} A(\epsilon_j, z_j).$$
\end{lem}

\begin{proof}
	We have
	\begin{align*}
		\sum_{j=n}^{m-1} A(\epsilon_j,z_j)&= \sum_{j=n}^{m-1} Z_{j+1}^o - Z_j^o - \epsilon_j\\
		&=Z_{m}^o - Z_n^o - \sum_{j=n}^{m-1} \epsilon_j
	\end{align*}
	so that
	$$\frac{Z_n^o}{\epsilon_n}+\frac{1}{\epsilon_n} \sum_{j=n}^{m-1} \epsilon_j= \phi_n^o(z_{m})
	= \frac{Z_{m}^o}{\epsilon_n} - \frac{1}{\epsilon_n}\sum_{j=n}^{m-1} A(\epsilon_j, z_j).$$
\end{proof}

\begin{prop}\label{prop:c1o}
	Let $k \in \N$ be the integer from Prop. \ref{prop:z2n+1bounded}.
	Let $y_{2n+1-k}:=z_{2n+1-k}$ and $y_{2n+1}=f^k(y_{2n+1-k})$.  For $n \leq j \leq 2n$ we define:
	$$y_j:=f^{-(2n+1-j)}(y_{2n+1}),$$
	where $f^{-1}$ is the local inverse of $f$ fixing $0$: $f^{-1}(z)=z-z^2+z^3-bz^4+O(z^5)$.
	We have
	$$\sum_{j=n}^{2n} A_0(z_j)-A_0(y_j) = \frac{C_1 (b-1)}{n}+O\left( \frac{\black \log n}{n^2}\right).$$
\end{prop}

\begin{proof}
	The proof mirrors the incoming case, so we will only sktech it and leave the details to the reader.
	Recall that $y_{2n+1}                                                                                                                                                                                                                                                                                                                                                                                                                                                                                                                                                                                                                                                                                                                                                                                                                                                                                                                                                                                                                                                                                                                                                                                                                                                                                                                                                                                                                                                                                                                                                                                                                                                                                                                                                                                                                                                                                                                                                                                                                                                                                                                                                                                                                                                                                                                                                                                                                                                                                                                                                                                                                                                                                                                                                                                                                                                                                                                                                                                                                                                                                                                                                                                                                                                                                                                                                                                                                                                                                                                                                                                                                                                                                                                                                                                                                                                                                                                                                                                                                                                                                                                                                                                                                                                                                                                                                                                                                                                                                                                                                                                                                                                                                                                                                                                                                                                                                                                                                                                                                                                                                                                                                                                                                                                                                                                                                                                                                                                                                                                                                                                                                                                                                                                                                                                                                                                                                                                                                                                                                                                                                                                                                                                                                                                                                                                                                                                                                                                                                                                                                                                                                                                                                                                                                                                                                                                                                                                                                                                                                                                                                                                                                                                                                                                                                                                                                                                                                                                                                                                                                                                                                                                                                                                                                                                                                                                                                                                                                                                                                                                                                                                                                                                                                                                                                                                                                                                                                                                                                                                                                                                                                                                                                                                                                                                                                                                                                                                                                                                                                                                                                                                                                                                                                                                                                                                                                                                                                                                                                                                                                                                                                                                                                                                                                                                                                                                                                                                                                                                                                                                                                                                                                                                                                                                                                                                                                                                                                                                                                                                                                                                                                                                                                                                                                                                                                                                                                                                                                                                                                                                                                                                                                                                                                                                                                                                                                                                                                                                                                                                                                                                                                                                                                                                                                                                                                                                                                                                                                                                                                                                                                                                                                                                                                                                                                                                                                                                                                                                                                                                                                                                                                                                                                                                                                                                                                                                                                                                                                                                                                                                                                                                                                                                                                                                                                                                                                                                                                                                                                                                                                                                                                                                                                                                                                                                                                                                                                                                                                                                                                                                                                                                                                                                                                                                                                                                                                                                                                                                                                                                                                                                                                                                                                                                                                                                                                                                                                                                                                                                                                        =O(1)$ by Proposition \ref{prop:z2n+1bounded}, and that $z_{2n+1-k}$
	belongs to a repelling petal for $f$ for some $k \in \N$ independent from $n$, so that the $(y_j)_{n \leq j \leq 2n+1}$
	are well-defined.
	
	By a straightforward adaptation of Lemma \ref{lem:zj-yj}, $z_j - y_j=O\left(\frac{2n+1-j}{n^2}\right)$ for $n \leq j \leq 2n+1$. More precisely, this applies for $n \leq j \leq 2n+1-k$; but it is clear from the
	definition of the $y_j$ that for $2n+1-k \leq j \leq 2n+1$, we have $z_j-y_j=O(\frac{1}{n^2})$.
	Therefore the proof of Lemma \ref{lem:sumA} can be repeated to yield that
	\begin{equation}
	\sum_{j=n}^{2n} A_0(z_j)-A_0(y_j)  = (b-1) \sum_{j=n}^{2n} z_j^2-y_j^2 + O\left( \frac{\black \log n}{n^2}\right).
	\end{equation}
	Next, we have, for $n \leq j \leq 2n$:
	\begin{align*}
	z_j &= \left( \psi_{\eps_j}^o\right)^{-1}(Z_j^o)=   \left( \psi_{\eps_j}^o\right)^{-1}\left(Z_{2n+1}^o- \sum_{k=j}^{2n} \eps_k+ A(\eps_k,z_k) \right) \\
	&= -\frac{\pi}{2n} \cot\left(\frac{\pi}{2}Z_{2n+1}^o -\frac{\pi}{2} \sum_{k=j}^{2n} \eps_k+ A(\eps_k,z_k)  \right) + O\left(\frac{1}{n^2}\right).
	\end{align*}
	Through similar computations as those appearing in the proof of Proposition \ref{prop:c1},
	we deduce that
	\begin{equation}
	z_j=-\frac{1}{-\frac{1}{z_{2n+1}}-\beta_j} + O\left(\frac{1}{n^2}\right),
	\end{equation}
	with $\beta_j:=\frac{2n}{\pi}\tan(\frac{\pi}{2}x_j)=\frac{2n}{\pi}\tan(\frac{\pi}{2}  \sum_{k=j}^{2n} \eps_k+A(\eps_k,z_k)) $. On the other hand,
	$$-\frac{1}{y_{j}}=-\frac{1}{y_{2n+1}}-\sum_{k=j}^{2n} A_0(y_j),$$
	from which it follows that
	$
	y_j = -\frac{1}{-\frac{1}{y_{2n+1}}-\gamma_j},
	$
	with $\gamma_j:=\sum_{k=j}^{2n} A_0(y_j)$. Then, again, similar computations show that
	$$z_j^2-y_j^2 = \frac{1}{n^2} \Phi(x_j)+O\left(\frac{1}{n^2(2n+1-j)}\right),$$
	and $x_j = \frac{2n-j+O(1)}{n}$ for $n \leq j \leq 2n$.
	Therefore, we finally obtain:
	$$\sum_{j=n}^{2n} A_0(z_j)-A_0(y_j) =\frac{b-1}{n} \int_0^1 \Phi(x) dx+O\left( \frac{\black \log n}{n^2}\right)
	=\frac{C_1 (b-1)}{n}+O\left( \frac{\black \log n}{n^2}\right). $$
\end{proof}

In what follows, a slight technical complication comes from the fact that
the expected endpoint of the orbit, $z_{2n+1}$, needs not lie in a small enough repelling petal in which $\phi^o_f$ is well-defined. In order to overcome this issue,
we stop a few iterations short and work instead with $z_{2n+1-k}$.

We now come to the main proposition of this subsection:

\begin{prop}\label{prop:outgoing}
	We have:
	$$\phi_n^o(z_{2n+1-k})=\phi^o_{\black f}(z_{2n+1-k})+\frac{E^o(z_{2n+1-k})}{n}+O\left(\frac{\log n}{n^2}\right),$$
	where $E^o(z)=-\frac{1}{2} \phi^o_{\black f}(z) - C_0 - (C_1-1)(b-1)$.
\end{prop}

\begin{proof}
	We proceed similarly to the proof of Proposition \ref{prop:incoming}.
	We have, for $z$ in a small enough repelling petal:
	\begin{equation}
		\phi^o_{\black f}(z)=-\frac{1}{z}-\sum_{j=1}^\infty A_0( f^{-j}(z) ),
	\end{equation}
	where $f^{-1}$ is the inverse branch of $f$ fixing $0$. With the same notations as in
	Proposition \ref{prop:c1o}, we set:
	$y_j:=f^{j-(2n+1-k)}(z_{2n+1-k})$.

	We have:
	\begin{align}
		\phi_n^o(z_{2n+1-k}) - \phi^o_{\black f}(z_{2n+1-k}) &= \frac{Z_{2n+1-k}^o}{\epsilon_n} + \frac{1}{z_{2n+1-k}}
		+ \sum_{j=n}^{2n-k} -\frac{1}{\epsilon_n}A(\epsilon_j,z_j) + A_0(y_j)
		+\sum_{j=-\infty}^{n-1} A_0(y_j) \\
		&= E_1 + E_2+ E_3,
	\end{align}
	where
	\begin{align*}
		E_1 &= \frac{Z_{2n+1-k}^o}{\epsilon_n} + \frac{1}{z_{2n+1-k}}\\
		E_2&= \sum_{j=n}^{2n-k}- \frac{1}{\epsilon_n}A(\epsilon_j,z_j) + A_0(y_j)\\
		E_3&=  \sum_{j=-\infty}^{n-1} A_0(y_j)
	\end{align*}

\begin{lem}
	We have $E_1=\frac{1}{n} \frac{1}{2z_{2n+1-k}} + O\left(\frac{1}{n^2}\right).$
\end{lem}	
	
\begin{proof}[Proof of the lemma]
	By Proposition \ref{prop:error1}, we have
	$Z_{2n+1-k}^o = -\frac{\epsilon_{2n+1-k}}{z_{2n+1-k}}+O\left(\frac{1}{n^3}\right)$
	so that
	\begin{align*}
		E_1 &= \frac{1}{z_{2n+1-k}} - \frac{\epsilon_{2n+1-k}}{\epsilon_n} \frac{1}{z_{2n+1-k}} + O\left(\frac{1}{n^2}\right)\\
		&= \frac{1}{z_{2n+1-k}} \left(1-\sqrt{\frac{n^2+n+O(1)}{n^2+2n+O(1)}}\right)+O\left(\frac{1}{n^2}\right)\\
		&=  \frac{1}{n} \frac{1}{2z_{2n+1-k}} + O\left(\frac{1}{n^2}\right).
	\end{align*}
\end{proof}

\begin{lem}
	We have
	$$E_2=\frac{1}{n} \left(-\frac{1}{2z_{2n+1-k}}-\frac{1}{2}\phi^o_{\black f}(z_{2n+1-k})-  C_0 - C_1(b-1) \right)+ O\left( \frac{\black \log n}{n^2}\right).$$
\end{lem}

\begin{proof}[Proof of the lemma]
	We have
	\begin{align*}
		E_2&= \sum_{j=n}^{2n-k} A_0(y_j) -\frac{1}{\epsilon_n} A(\epsilon_j,z_j)\\
		&= \left( \sum_{j=n}^{2n-k} A_0(y_j) - \frac{\epsilon_j}{\epsilon_n} A_0(z_j) \right)
		- \left( \sum_{j=n}^{2n-k} C_0 \epsilon_j^3 + O(z_j \epsilon_j^3) \right).
	\end{align*}
	Same as before, we have $\frac{1}{\epsilon_n}\sum_{j=n}^{2n-k} C_0 \epsilon_j^3 + O(z_j \epsilon_j^3) = \frac{C_0}{n}+O\left(\frac{1}{n^2}\right)$. On the other hand, we have
	\begin{align*}
		\sum_{j=n}^{2n-k} A_0(y_j) -\frac{\epsilon_j}{\epsilon_n}A_0(z_j) &= \sum_{j=n}^{2n-k} (1-\frac{\epsilon_j}{\epsilon_n}) A_0(z_j) +\sum_{j=n}^{2n-k} A_0(y_j)-A_0(z_j). \\
	\end{align*}	
	Now note that
	\begin{align*}
		 \sum_{j=n}^{2n-k} (1-\frac{\epsilon_j}{\epsilon_n}) A_0(z_j)&=
		  \sum_{j=n}^{2n-k} (1-\frac{\epsilon_j}{\epsilon_n}) A_0(y_j)
		  + \sum_{j=n}^{2n-k} (1-\frac{\epsilon_j}{\epsilon_n}) (A_0(z_j)-A_0(y_j)),
	\end{align*}
	and that
	\begin{align*}
		\left| \sum_{j=n}^{2n-k} (1-\frac{\epsilon_j}{\epsilon_n}) (A_0(z_j)-A_0(y_j)) \right|
		&\leq \max_{n \leq j \leq 2n-k} \left|1-\frac{\eps_j}{\eps_n} \right| \cdot \sum_{j=n}^{2n-k} \left| A_0(z_j)-A_0(y_j)\right|\\
		&=O\left(\frac{1}{n^2} \right),
	\end{align*}
	by Proposition \ref{prop:c1o}.	
	Therefore, as in the proof of Proposition \ref{prop:incoming}:
	\begin{align*}
	\sum_{j=n}^{2n-k} A_0(y_j) - \frac{\eps_j}{\eps_n} A_0(z_j)&=-\frac{C_1(b-1)}{n}+\frac{1}{2n} \sum_{j=n}^{2n-k} A_0(y_j) + O\left(\frac{\black \log n}{n^2}\right)\\
	&=- \frac{1}{n} \left(C_1(b-1)+\frac{1}{2z_{2n+1-k}} + \frac{1}{2}\phi^o_{\black f}(z_{2n+1-k}) \right)+O\left(\frac{\black \log n}{n^2}\right)
	\end{align*}
	from which the lemma follows.
\end{proof}

\begin{lem}
	We have $E_3 = \frac{b-1}{n}+O\left(\frac{1}{n^2}\right)$.
\end{lem}	

\begin{proof}[Proof of the lemma]
	The proof is the same as in the incoming case: it follows from the fact that $A_0(y) = (b-1)y^2+O(y^3)$ and $y_j=\frac{1}{2n-j}+O\left(\frac{1}{(2n-j)^2}\right)$.
\end{proof}

\end{proof}

\subsubsection{Conclusion}

\begin{prop}\label{prop:sumeps}
	We have:
	$$\frac{1}{\epsilon_n}\left( \left( \sum_{j=0}^{2n} \epsilon_j \right) -2 \right)= - \frac{1}{n} \left(\frac{1}{2}+\phi_g^{{\black \iota}}(w_{\black 0})\right)+O\left(\frac{1}{n^2}\right).$$
\end{prop}

\begin{proof}
	We have:
	\begin{align*}
		 \sum_{j=0}^{2n} \epsilon_j &=  \sum_{j=0}^{2n} \frac{1}{\sqrt{n^2+j+\phi_g^{{\black \iota}}(w_{\black 0})+o(1)}}\\
		 &= \sum_{j=0}^{2n} \frac{1}{n} \left( 1- \frac{j}{2n^2}-\frac{\phi_g^\iota(w_0)}{2n^2}+o(\frac{1}{n^2})\right)\\
		 &=2+\frac{1}{n} - \frac{\phi_g^{{\black \iota}}(w_{\black 0})}{n^2}+o(\frac{1}{n^2})- \frac{1}{2n^3}\sum_{j=0}^{2n} j\\
		 &=2+\frac{1}{n} - \frac{\phi_g^{{\black \iota}}(w_{\black 0})}{n^2}+o(\frac{1}{n^2})- \frac{1}{2n^3} \frac{2n(2n+1)}{2}\\
		 &=2 - \frac{\phi_g^{{\black \iota}}(w_{\black 0})}{n^2} - \frac{1}{2n^2}+o(\frac{1}{n^2}).
	\end{align*}
	On the other hand:
	\begin{align*}
		\frac{1}{\epsilon_n} &= \sqrt{n^2+n+O(1)}\\
		&=n \left(1 + \frac{1}{2n}+ O\left(\frac{1}{n^2}\right) \right)\\
		&=n + \frac{1}{2}+O\left(\frac{1}{n}\right),
	\end{align*}
	and therefore:
	\begin{align*}
		\frac{1}{\epsilon_n}\left( \left( \sum_{j=0}^{2n} \epsilon_j \right) -2 \right)
		&= \left( - \frac{\phi_g^{{\black \iota}}(w_{\black 0})}{n^2} - \frac{1}{2n^2}+o(\frac{1}{n^2}) \right)
		\left(n + \frac{1}{2}+O\left(\frac{1}{n}\right) \right)\\
		&=- \frac{1}{n} \left(\frac{1}{2}+\phi_g^{{\black \iota}}(w_{\black 0})\right)+O\left(\frac{1}{n^2}\right).
	\end{align*}
\end{proof}

We are now finally ready to prove the following theorem:

\begin{thm}[Lavaurs' theorem with an error estimate]\label{thm:lav} {\black Let $K\subset \mathcal{B}_f\times \mathcal{B}_g$ be a compact set. For all $(z_0,w_0)\in K$ and all sufficiently large $n$ we have:
	$$z_{2n+1}=\mathcal{L}_f(z_0) + \frac{h(z_0,w_0)}{n} 	+ O\left(\frac{\log n}{n^2}\right),$$
	where
	$$
	h(z,w)=\frac{\mathcal{L}'_f(z)}{(\phi^\iota_f)'(z)}     \left(
	2C_0+2(C_1-1)(b-1)- \frac{1}{2} +\phi^\iota_f(z) -\phi_g^\iota(w) \right)
	$$
	is holomorphic on $\mathcal{B}_f\times \mathcal{B}_g$ 	and the constant in $O\left(\frac{\log n}{n^2}\right)$ is independent of the point $(z_0,w_0)$ and the integer $n$.}
\end{thm}

\begin{proof}
	We have, by definition:
	\begin{align*}
		\phi_n^o(z_{2n+1-k})&=\frac{1}{\epsilon_n} Z_n^o + \frac{1}{\epsilon_n} \sum_{j=n}^{2n-k} \epsilon_j\\
		&=\frac{Z_n^\iota}{\epsilon_n}- \frac{2}{\epsilon_n}+  \frac{1}{\epsilon_n} \sum_{j=n}^{2n-k} \epsilon_j\ \\
		&=\phi_n^\iota(z_0)-\frac{2}{\epsilon_n}+ \frac{1}{\epsilon_n} \sum_{j=0}^{2n-k} \epsilon_j, \\
	\end{align*}
	and therefore
	$$\phi^o_{\black f}(z_{2n+1-k}) + \frac{E^o(z_{2n+1-k})}{n} = \phi^\iota_{\black f}(z_0) + \frac{E^\iota(z_0)}{n} - \frac{\phi_g^{{\black \iota}}(w_0)+1/2}{n} - \frac{1}{\eps_n} \sum_{j=2n-k}^{2n} \eps_j+  O\left(\frac{\black \log n}{n^2}\right),$$
	by Propositions \ref{prop:incoming}, \ref{prop:outgoing} and \ref{prop:sumeps}.

	On the other hand, we have:
	\begin{align*}
		\frac{1}{\eps_n} \sum_{j=2n-k}^{2n} \eps_j &=
		\frac{1}{\frac{1}{n}-\frac{1}{2n^2}+O(\frac{1}{n^3})} \left(\frac{k}{n}- k\frac{2n}{2n^3} + O(\frac{1}{n^3}) \right) \text{ by Lemma \ref{lem:epsj}}\\
		&=\left(1+\frac{1}{2n}+O(\frac{1}{n^2})\right) \left(k-\frac{k}{n}+O(\frac{1}{n^2})\right) \\
		&=k-\frac{k}{2n}+O(\frac{1}{n^2}).
	\end{align*}
	
	Therefore:
	$$
	\phi^o_{\black f}(z_{2n+1-k})+k + \frac{E^o(z_{2n+1-k})-k/2}{n} = \phi^\iota_{\black f}(z_0) + \frac{E^\iota(z_0)}{n} - \frac{\phi_g^{{\black \iota}}(w_0)+1/2}{n}+  O\left(\frac{\black \log n}{n^2}\right).
	$$
	Recall that the outgoing Fatou coordinate $\phi^o_{\black f}$ has a well-defined inverse $\psi_f: \C \to \C$
	satisfying the functional equation $\psi_f(Z+1)=f \circ \psi_f(Z)$.
	Observe that since $k=O(1)$, we have
	$$
	\psi_f \left(\phi^o_{\black f}(z_{2n+1-k})+k \right) = f^k(z_{2n+1-k}) + O\left(\frac{1}{n^2}\right) = z_{2n+1} + O\left(\frac{1}{n^2}\right).
	$$
	Therefore, composing on both sides by $\psi_f$ and setting $E^o(z_{2n+1}):=E^o(z_{2n+1-k})-\frac{k}{2}$, we get:

	\begin{align*}
		z_{2n+1} &= (\phi^o_{\black f})^{-1} \left(   \phi^\iota_{\black f}(z_0) +   \frac{E^\iota(z_0)- E^o(z_{2n+1})-1/2-\phi_g^{{\black \iota}}(w_0)}{n} + O\left(\frac{\black \log n}{n^2}\right)      \right)\\
		&= \mathcal{L}_{\black f}(z_0) +  {((\phi^o_{\black f})^{-1})}'(\phi^\iota_{\black f}(z_0)) \left(     \frac{E^\iota(z_0)- E^o(z_{2n+1})-1/2-\phi_g^{{\black \iota}}(w_0)}{n}    \right) + O\left(\frac{\black \log n}{n^2}\right) \\
		&=  \mathcal{L}_{\black f}(z_0) + \frac{\mathcal{L}_{{\black f}}'(z_0)}{(\phi^\iota_{\black f})'(z_0)}   \left(    \frac{E^\iota(z_0)- E^o(z_{2n+1})-1/2-\phi_g^{{\black \iota}}(w_0)}{n}   \right) + O\left(\frac{\black \log n}{n^2}\right).
	\end{align*}
	In particular, we have proved that $z_{2n+1}=\mathcal{L}_{\black f}(z_0)+O\left(\frac{1}{n}\right)$.
	From there, we deduce that $\phi^o_{\black f}(z_{2n+1-k})+k=\phi^o_{\black f}(z_0)+O\left(\frac{1}{n}\right)$.
	Plugging this into the expression for $E^o(z_{2n+1})$, we finally obtain:
	$$z_{2n+1}=\mathcal{L}_{\black f}(z_0) + \frac{1}{n} \frac{\mathcal{L}_{{\black f}}'(z_0)}{(\phi^\iota_{\black f})'(z_0)}     \left(
	2C_0+2(C_1-1)(b-1) - \frac{1}{2}+ \phi^\iota_{\black f}(z_0)-\phi_g^{{\black \iota}}(w_0) \right)
	+ O\left(\frac{\black \log n}{n^2}\right).$$
\end{proof}

\subsection{Choice of index}\label{subsection:choice of index}

Assume that ${\black z_0}$ is a Siegel fixed point for the Lavaurs map $\lcal_{\black f}$, and
let $\lambda$ be its multiplier. Denote by $\kappa_{\black z_0}$ the index from Theorem \ref{thm:siegel-simple}: it is given by the
formula
$$\kappa_{\black z_0} = \frac{2b_2 c_0}{\lambda(1-\lambda)}+\frac{c_1}{\lambda},$$ with $2b_2 = \lcal_{{\black f}}''({\black z_0})$,
$c_0=h({\black z_0})$, $c_1 = h'({\black z_0})$, and
\begin{equation}
	h(z):=\frac{\mathcal L_{{\black f}}'(z)}{(\phi^\iota_{\black f})'(z)}     \left(
	2C_0+2(C_1-1)(b-1)- \frac{1}{2} +\phi^\iota_{\black f}(z) -\phi_g^{{\black \iota}}(w_0) \right).
\end{equation}
The function $h$ is the error term computed in the previous section.

A straightforward computation gives us that
\begin{equation}
\kappa_{\black z_0}=1+\frac{
	C + \phi^\iota_{\black f}({\black z_0})-\phi_g^{{\black \iota}}(w_0) }{((\phi^\iota_{\black f})'({\black z_0}))^2}\left(\frac{\mathcal{L}_{{\black f}}''({\black z_0})(\phi^\iota_{\black f})'({\black z_0})}{\lambda(1-\lambda)}-(\phi^\iota_{\black f})''({\black z_0})\right),
\end{equation}
for some universal constant $C \in \R$.

Observe that $\re(\kappa_{\black z_0})$ is independent from $w_0$ if and only if
\begin{equation}\label{eq:conditionkappa}
\frac{\mathcal{L}_{{\black f}}''({\black z_0})(\phi^\iota_{\black f})'({\black z_0})}{\lambda(1-\lambda)}-(\phi^\iota_{\black f})''({\black z_0})=0.
\end{equation}

If condition \eqref{eq:conditionkappa} is satisfied, then $\kappa_{\black z_0}=1$, and accordingly, Theorem \ref{thm:siegel-simple}
implies that there are no wandering domains for $P$ converging to the bi-infinite orbit of $({\black z_0},0)$, since we are then in the expulsion scenario of the trichotomy.

On the other hand, if the equality \eqref{eq:conditionkappa} is not satisfied, then $w_0 \mapsto \kappa_{\black z_0}(w_0)$ is a non-constant holomorphic function (defined on the parabolic basin $\mathcal{B}_g$),
of the form $w_0 \mapsto a \phi_g^{{\black \iota}}(w_0)+b$, with $a,b \in \C$ (independent from $w_0$) and $a \neq 0$.
Therefore, the condition for $\re(\kappa_{\black z_0}(w_0))$ to be negative is equivalent to $\phi_g^{{\black \iota}}(w_0)$ belonging to some
half-plane; but $\phi_g^{{\black \iota}}(\mathcal{B}_g)$ contains a domain of the form $U:=\{ W \in \C :  \re(W) > R - k |\im(W)|   \}$, for some
$R>0$ and $k \in (0,1)$, see e.g. \cite[Proposition 2.2.1 p.330]{Shishikura}. Since $U$ intersects any open half-plane, if condition \eqref{eq:conditionkappa} is {\black not} satisfied, then there exists some open subset $U_0 \subset U$ for which
$\re(\kappa_{\black z_0}(w_0))<0$, and so by Theorem \ref{thm:siegel-simple} there is a wandering domain accumulating on $({\black z_0},0)$.

\section{A Lavaurs map with a Siegel disk}\label{section: proposition B'}

The goal of this section is to construct a polynomial $f$ of the form $f(z)=z+z^2+z^3+O(z^4)$
whose Lavaurs map has a Siegel fixed point with diophantine multiplier $\lambda$, which does not
satisfy equality \eqref{eq:conditionkappa}.
The outline of the argument is as follows:
\begin{itemize}
	\item We start by finding a degree 7 real polynomial whose Lavaurs map has a superattracting fixed point, and for which a suitable reformulation of \eqref{eq:conditionkappa} does not hold
	\item We perturb that polynomial to get an attracting but not superattracting fixed point, in a way
	that equality \eqref{eq:conditionkappa} still does not hold
	\item  We apply quasiconformal surgery to get a multiplier arbitrarily close to $1$
	\item We show that in the limit, we get a polynomial whose Lavaurs map has a parabolic fixed point
	that does not exit the parabolic basin
	\item We perturb that last polynomial to get a Siegel fixed point, leaving the family of real polynomials;
	we prove that condition \eqref{eq:conditionkappa} does not hold for that last polynomial.
\end{itemize}

Recall that in \cite{ABDPR}, there are two constructions of a Lavaurs map with an attracting fixed point.
One is based on a residue computation near infinity in the Ecalle cylinder, and makes use of the fact
that in the family $f_a(z):=z+z^2+az^3$, the multiplier of the horn map $e_a$ of $f_a$ at the ends of the Ecalle cylinder is a non-constant holomorphic function of $a$.
This method cannot be used in a family of polynomials of the form $f(z)=z+z^2+z^3+O(z^4)$, where
those fixed points for the horn map are persistently parabolic.
This is why we adapt the second strategy for the first two steps described above.
\medskip

{\black
{\bf Remark:} From now on we will be using slightly different notations then in previous sections. Namely we will drop the subscript $f$ from the Fatou coordinates and the Lavaurs map in order to have space for other indexes in the subscript. It will be clear from the context to which function the Fatou coordinates or Lavaurs maps correspond.}

\medskip

Let $\phi^\iota$ be the incoming Fatou coordinate, and $\psi^o$ the outgoing Fatou parametrization.
Recall that the Lavaurs map is given by $\lcal = \psi^o \circ \phi^\iota : \mathcal{B}_f \to \C$, the lifted horn map is
$\ecal = \phi^\iota \circ \psi^o : V \to \C$, with $V \subset \C$ containing $\{Z : |\im(Z)|>R  \}$ for $R$ large enough. We have $\ecal \circ \phi^\iota = \phi^\iota \circ \lcal$, and
$\ecal(Z+1)=\ecal(Z)+1$, so $\ecal$ descends to a self-map of $\C/\Z$. Conjugating by the isomorphism $Z \mapsto e^{2i\pi Z}$, we
obtain a map $e: U - \{0,\infty\} \to \C^*$, where $U$ is an open set containing $0$ and $\infty$.
The map extends to $U$, and fixes $0$ and $\infty$. Since we consider polynomials with
$f(z)=z+z^2+z^3+O(z^4)$, both of those fixed points have multiplier 1.

\subsection{Construction a polynomial}
%
%Let $f(z)=z+z^2+O(z^3)$ be a polynomial and $\zeta = \lcal(\zeta)$ a fixed point of its Lavaurs map, with multiplier $\lambda \notin \{0,1\}$.  We say that condition ($\star$) is satisfied if
%\begin{equation}\label{star}
%\frac{\lcal''(\zeta) (\phi^\iota)'(\zeta)}{\lambda(1- \lambda)} - (\phi^\iota)''(\zeta) = 0.
%\end{equation}

{\black 
Let $f(z)=z+z^2+O(z^3)$ be a polynomial and $\zeta = \lcal(\zeta)$ a fixed point of its Lavaurs map, with multiplier $\lambda$.

\begin{defi}\label{star} If  $\lambda \notin \{0,1\}$ we say that the pair $(f,\zeta)$ is degenerate if and only if
\begin{equation}
\frac{\lcal''(\zeta) (\phi^\iota)'(\zeta)}{\lambda(1- \lambda)} - (\phi^\iota)''(\zeta) = 0.
\end{equation}
\end{defi}

}

\vspace{1cm}

%For a polynomial $f(z)=z+z^2+O(z^3)$ and a fixed point $\zeta = \lcal(\zeta)$ of the Lavaurs map, with
%multiplier $\lambda \notin \{0,1\}$, we say that condition (*) is satisfied if
%$$\frac{\lcal''(\zeta) (\phi^\iota)'(\zeta)}{\lambda(1- \lambda)} - (\phi^\iota)''(\zeta) = 0.$$

\begin{lem}
	We have:
	\begin{equation}
		\frac{\lcal''(\zeta) (\phi^\iota)'(\zeta)}{\lambda(1- \lambda)} - (\phi^\iota)''(\zeta)
		=\frac{\lambda}{1-\lambda} \left[\frac{(\psi^o)''(\phi^\iota(\zeta)) (\phi^\iota)'(\zeta)}{(\psi^o)'(\phi^\iota(\zeta))^2}  + (\phi^\iota)''(\zeta)     \right].
	\end{equation}
	
\end{lem}

\begin{proof}
	Since $\lcal = \psi^o \circ \phi^\iota$ we obtain
	$$\lcal'(z)=(\psi^o)'(\phi^\iota(z)) \phi'(z)$$ and
	$$\lcal''(z) = (\psi^o)''(\phi^\iota(z)) \phi'(z)^2+(\psi^o)'(\phi^\iota(z)) (\phi^\iota)''(z).$$
	Recalling that $\lcal'(\zeta)=\lambda$ it follows that
	$$\frac{\phi'(\zeta)}{\lambda} = \frac{1}{(\psi^o)'(\phi^\iota(\zeta))},$$
	and so
	\begin{align*}
		\frac{\lcal''(\zeta)\phi'(\zeta)}{\lambda}&= \frac{(\psi^o)''(\phi^\iota(\zeta)) (\phi^\iota)'(\zeta)^2}{(\psi^o)'(\phi^\iota(z))}
		+ (\phi^\iota)''(\zeta).
	\end{align*}
	It follows that
	\begin{align*}
		\frac{\lcal''(\zeta) (\phi^\iota)'(\zeta)}{\lambda(1- \lambda)}-(\phi^\iota)''(\zeta)
		&= \frac{1}{1-\lambda} \left[   \frac{(\psi^o)''(\phi^\iota(\zeta)) (\phi^\iota)'(\zeta)^2}{(\psi^o)'(\phi^\iota(\zeta))}
		+ (\phi^\iota)''(\zeta)    \right] - (\phi^\iota)''(\zeta) \\
		&= \frac{(\psi^o)''(\phi^\iota(\zeta)) (\phi^\iota)'(\zeta)^2}{(1-\lambda)(\psi^o)'(\phi^\iota(\zeta))} + (\phi^\iota)''(\zeta) \frac{\lambda}{1-\lambda}\\
		&=\frac{\lambda}{1-\lambda} \left[\frac{(\psi^o)''(\phi^\iota(\zeta)) (\phi^\iota)'(\zeta)}{(\psi^o)'(\phi^\iota(\zeta))^2}  + (\phi^\iota)''(\zeta)     \right].
	\end{align*}
\end{proof}
For the rest of the paper we shall set
\begin{equation}
	\fcal(f,\zeta):=\frac{(\psi^o)''(\phi^\iota(\zeta)) (\phi^\iota)'(\zeta)}{(\psi^o)'(\phi^\iota(\zeta))^2}  + (\phi^\iota)''(\zeta),
\end{equation}
where $\psi^o$ and $\phi^\iota$ are the Fatou parametrization and coordinates associated to $f$.
Note that for $\lambda \notin \{0,1\}$, {\black the pair $(f,\zeta)$ is degenerate if and only if}
%condition (*) is equivalent to
$\fcal(f,\zeta)=0$.

\vspace{1cm}

We record here for later use the following lemma:

\begin{lem}\label{lem:phi''non0}
	Let $f(z)=z+z^2+az^3+O(z^4)$ and let
	$\phi^\iota$ denote its incoming Fatou coordinate. Let $c$ be a critical point in the parabolic basin of $f$.
	Then we have $(\phi^\iota)''(c) =  0$ if and only if either $c$ is multiple critical point of $f$, or if
	the orbit of $c$ meets another critical point of $f$.
\end{lem}

\begin{proof} {\black The sequence of functions $$\phi_n(z):=-\frac{1}{f^n(z)}-n-(1-a) \log n$$  converges
 locally uniformly on the parabolic basin  to
	
	%the following limit we have:
%	$$
%	\phi^\iota(z):=\limn -\frac{1}{f^n(z)}-n-(1-a) \log n,
%	$$
%the convergence being locally uniform on the parabolic basin.
	$$
	\phi^\iota(z):=\limn \phi_n(z),
	$$}
	
	Therefore $(\phi^\iota)''(c)$  equals $\limn \phi_n''(c)$. Moreover,
	$
	\phi_n'(z) = \dfrac{(f^n)'(z)}{\left[f^n(z) \right]^2}
	$
	and
	\begin{align*}
	\phi_n''(c)&= \frac{d}{dz}_{|z=c} \frac{(f^n)'(z)}{\left[f^n(z) \right]^2} \\
	&= \frac{(f^n)''(c)  \left[f^n(c) \right]^2 -  2\left[ (f^n)'(c)\right]^2 f^n(c) }{\left[f^n(c) \right]^4}\\
	&=\frac{(f^n)''(c)}{\left[f^n(c) \right]^2}\\
	&=f''(c) \frac{\prod_{k=1}^{n-1} f'(f^k(c))}{\left[f^n(c) \right]^2}.
	\end{align*}	
	
	For the third and fourth equalities we used the fact that $f'(c)=0$.
	Since $c$ is in the parabolic basin of $f$, we have $\left[ f^n(c) \right]^2  \sim \frac{1}{n^2}$.
	Moreover, for $k \geq k_0$ with $k_0$ large enough, $f'(f^k(c)) \neq 0$ and
	$$f'(f^k(c)) = 1-\frac{2}{k}+O\left(\frac{\log k}{k^2}\right) = \exp\left(-\frac{2}{k}+O\left(\frac{\log k}{k^2}\right) \right).$$
	Therefore:
	\begin{align*}
	\prod_{k=k_0}^{n-1} f'(f^k(z)) &= \prod_{k=k_0}^{n-1} \exp\left(-\frac{2}{k}+O\left(\frac{\log k}{k^2}\right) \right)=\frac{\exp(O(1))}{n^2}.  \\			
	\end{align*}
	In particular,
	$\limn \dfrac{\prod_{k=k_0}^{n-1} f'(f^k(c))}{\left[f^n(c) \right]^2} \neq 0$, so
	$(\phi^\iota)''(c) = 0$ if and only if $f''(c)=0$ or $(f^k)'(c) = 0$, which concludes
	the proof.
\end{proof}

For $t \in \R$, a real polynomial  $P(z)=z+z^2+z^3+O(z^4)$ and $n > \deg P$ odd, let
$$f_t(z) = P(z) - \frac{P'(t)}{n t^{n-1}}z^n$$
Note that $f_t'(t)=0$: the choice of this family ensures that we have a marked critical point in $\R$. By $\lcal_{t}$ we denote the Lavaurs map of phase $0$ for the polynomial $f_t$.

\begin{prop}\label{prop:extn}
	Assume that there exists $P,n$ and $t_\infty <0$ as above such that :
	\begin{enumerate}
		\item $f_{t_\infty}(t_\infty)=0$
		\item $\frac{d}{dt}_{|t=t_\infty} f_t(t) <0$
		\item  $f_{t_\infty}$ has negative leading coefficient
		\item there exists $x>0$ in the
		repelling petal of $f_{t_\infty}$ that escapes to infinity.
	\end{enumerate}
	Then there is a sequence $t_n \to t_\infty$ such that $\lcal_{t_n}(t_n)=t_n$.
\end{prop}

\begin{proof}
	We will rely on the following two claims:
	
	\begin{claim}
		For $t \in (t_\infty,t_\infty+\eps)$  with $\eps>0$ small enough, the critical point $t$ is in the parabolic basin of $f_t$.
	\end{claim}
	
	\begin{proof}[Proof of the claim]
		It is enough to show that there is $r>0$ such that $(-r,0)$ is in the
		parabolic basin of $f_t$ for all $t$ close enough to $t_\infty$.
		Indeed, by (1) and (2), we have
		that for all $r>0$ there exists $\eps>0$ such that $f_t(t) \in (-r,0)$ for all $t \in (t_\infty,t_\infty+\eps)$.
		Let
		$$r_t:=\sup \{r>0 : \forall y \in (-r,0), 0<\frac{f_t(y)}{y}<1 \}.$$
		 For all $y \in (-r_t,0)$,
		$t<f_t(y)<0$ hence $y$ is in the parabolic basin of $f_t$. Finally, $t \mapsto r_t$ is continuous and $r_{t_\infty}>0$.
	\end{proof}

	\begin{claim}
		There exists a sequence $\tilde t_n \to t_\infty$ (with $\tilde t_n>t_\infty$) such that $\lcal_{\tilde t_n}(\tilde t_n) = f_{\tilde t_n}^n(x)$.
	\end{claim}

	\begin{proof}[Proof of the claim]
		We adapt here the argument from \cite{ABDPR}.
		The desired equality $\lcal_{\tilde t_n}(\tilde t_n) = f_{\tilde t_n}^n(x)$ is equivalent to
		$\psi_{\tilde t_n}^o \circ \phi_{\tilde t_n}^\iota(\tilde t_n)= \psi_{\tilde t_n^o}(\phi_{\tilde t_n}^{o}(x)+n)$.

		In particular, it is enough to find $\tilde t_n$ such that $\phi_{\tilde t_n}^\iota(\tilde t_n)=\phi_{\tilde t_n}^{o}(x)+n$. We look for $\tilde t_n$
		under the form $\tilde t_n = t_\infty-\frac{\alpha}{n+u}$, with $\alpha = \frac{1}{\frac{d}{dc}_{|c=t_\infty} f_c(c)}$. By the preceding claim, it is in the parabolic basin
		for $n$ large enough.

		We have
		$\phi_{\tilde t_n}^{o}(x)+n = n+ \phi_{t_\infty}^{o}(x)+o(1)$ since the map $t \mapsto \phi_t^o$ is continuous.
		Additionally,
		\begin{align*}
			\phi_{\tilde t_n}^\iota(\tilde t_n)&=\phi_{\tilde t_n}^\iota(f_{\tilde t_n}(\tilde t_n))-1 \\
			&=-\frac{1}{f_{\tilde t_n}(\tilde t_n)}-1+o(1),\quad ( \text{according to the asymptotic expansion of }\phi^\iota)\\
			&=n+u-1+o(1).
		\end{align*}
	 	Therefore, we have reduced the problem to solving the equation
	 	$u-1+o(1) = \phi_{t_\infty}^{o}(x)$ for $u \in \R$, where the $o(1)$ term is a continuous function of $u$.
	 	By the {\black intermediate value theorem} there is a solution $u=u_n \in (\phi_{t_\infty}^{o}(x), \phi_{t_\infty}^{o}(x)+2)$. We can take $\tilde t_n = t_\infty - \frac{\alpha}{n+u_n}$, and
	 	since $(u_n)_{n \in \N}$ is bounded from below, the sequence $(t_n)$ is well-defined for $n$ large enough
	 	and converges to $t_\infty$.
	\end{proof}
	
	We now come back to the proof of Proposition \ref{prop:extn}. For $n$ large enough, $G_{\tilde t_n}(x)>0$ (by continuity of
	the Green function
	$G$). Therefore $\lcal_{\tilde t_n}(\tilde t_n) = f_{\tilde t_n}^n(x)$ tends to $\infty$, and more precisely, $+\infty$ or $-\infty$ depending on the parity of $n$, thanks to condition (3). Therefore the continuous function
	$F(t):=\lcal_t(t)-t$ alternates sign between two consecutive $\tilde t_n$, so by the intermediate value theorem must have a zero $t_n$ between them.
\end{proof}

\begin{prop}\label{prop:exlc=c}
	Let {\black $P(z):=z+z^2+z^3+\frac{23}{7}z^4+\frac{17}{7}z^5$}, let $t_\infty:=-1$
	and $n:=7$. Then $P$, $n$ and $t_\infty$ satisfy conditions $(1)-(4)$
	in Proposition \ref{prop:extn}.
\end{prop}

\begin{proof}
	%The constants $a$ and $b$  are chosen so
	{\black Observe} that $f_{t_\infty}(t_\infty)=0$ and $P'(t_\infty)=1$. That second property
	implies that $f_{t_\infty}$ has negative leading coefficient. Therefore, conditions $(1)$ and $(3)$
	are satisfied.
	
	Let us check that condition $(2)$ is also satisfied. We have:
	\begin{align*}
		\frac{d}{dt}_{|t=t_\infty} f_t(t) &= \frac{d}{dt}_{|t=t_\infty} P(t) - \frac{t}{n}P'(t)\\
		&=\frac{n-1}{n} P'(t_\infty)  - \frac{t_\infty}{n}P''(t_\infty)\\
		&=\frac{6}{7} + \frac{1}{7} P''(-1) = -\frac{50}{49}<0.
	\end{align*}
	
Finally, condition $(4)$ is satisfied for $x:=1$. Indeed, {\black recall here that if $f(z)=\sum_{k=0}^n a_k z^k$ is a complex polynomial and $R=\max\{1,\frac{1+|a_0|+\ldots+ |a_{n-1}|}{|a_{n}|}\}$ then for all $|z|>R$ we have $|f(z)|\geq \frac{|z|^n}{R}$ , hence if an orbit at any point leaves the disk of radius $R$, then it must converge to infinity. Observe that for our polynomial $f_{t_\infty}$ we have $R=68$  and that a straightforward computation yields $f_{t_\infty}(1)=\frac{60}{7}$ and $|f^2_{t_\infty}(1)|>68$.}

%{\black since $f_{t_\infty}$ is a polynomial, there exist $R>0$ such that $|f_{t_\infty}(z)|\geq \frac{|z|^7}{R}$ for all $|z|>R$, hence if an orbit at any point leaves the disk of radius $R$, then it must converge to infinity. Observe that for $f_{t_\infty}$ it suffices to take $R=68$  and that a straightforward computation yields $f_{t_\infty}(1)=\frac{60}{7}$ and $|f^2_{t_\infty}(1)|>R$.}
%using symbolic computation software,
%one can check that $f_{t_\infty}^3(1)$ is larger than the escape radius for $f_{t_\infty}$.
%Recall here that if $f(z)=\sum_{k=0}^n a_k z^k$ is a complex polynomial, then
%the filled-in Julia set $K_f$ is contained in $\D\left(0, \frac{1+\sum_{k=0}^{n-1} |a_k|}{|a_{n}|}\right)$.
This proves rigorously that $x:=1$ has unbounded orbit under $f_{t_\infty}$.
\end{proof}

\begin{lem}\label{lem:basepoint}
	For $\eps_0>0$ small enough, there exists $\tbold>-1$ such that the following properties hold
	for $f_{\tbold}$:
	\begin{enumerate}
		\item $\lcal_{\tbold}$ has a fixed point $x_{\tbold}$ 	with multiplier $\eps_0 \neq 0$
		\item $\fcal(f_{\tbold},x_{\tbold}) \neq 0$
		\item $f_{\tbold}$ has 4 real critical points, ordered from left to right : $c_1, c_2, c_3, c_4$, with $\tbold=c_2$; and two non-real
		complex conjugate critical points $c'$ and $\overline{c'}$.
		\item the critical points $c_1$ and $c_4$ lie in the basin of infinity; the critical points $c_2$ and $c_3$
		are in the  parabolic basin.
		\item there is a unique repelling fixed point $\xi \in (c_1, c_2)$, and the intersection of $\R$ and the immediate
		basin of attraction of $0$ is $(\xi, 0)$.
		\item there is a unique $y \in (\xi, c_2)$ such that $f_\tbold(y)=c_2$
	\end{enumerate}
	
\end{lem}

\begin{proof}
	We will find $\tbold$ by taking a perturbation of one of the $t_{n_0}$ constructed above, with $n_0$ large enough.

	First, note that properties (3)-(6) hold for $f:=f_{t_\infty} : z \mapsto z+z^2+z^3 + \frac{23}{7}z^4+\frac{17}{7}z^5-\frac{z^7}{7}$; we leave the details to the reader.
	Therefore, for $n_0$ large enough, properties (3)-(6) still hold for $f_{t_{n_0}}$, as these properties
	are clearly open (in $\R$) near $t=t_\infty$. To lighten the notations, we let $f:=f_{t_{n_0}}$
	and $c_2:=t_{n_0}$.

	We now claim that $\fcal$ is well defined at $(f,c_2)$, and that $\fcal(f,c_2) \neq 0$.
	According to Lemma \ref{lem:phi''non0}, since $f$ satisfies conditions (3)-(6),
	we have $(\phi^\iota)''(c) \neq 0$. Indeed, $c_2$ is a simple critical point of $f$; and we claim that the forward orbit of $c_2$ does not meet any other critical point of $f$.
	To see this, note that the critical point $c_2$ is simple for $f$, and real. Since $c'$ and $\overline{c'}$
	are not real, the orbit of $c_2$ cannot land on either of them.
	Since the critical points $c_1$ and $c_4$ do not belong to the parabolic basin, the orbit of $c_2$ cannot land on them either.
	Finally, since $f(c_2)>c_3$, and since $f(c_2)$ belongs to a small attracting petal in which the
	sequence of iterates $(f^n(c_2))_{n \in \N}$ is increasing, the orbit of $c_2$ cannot land on
	$c_3$ either.
	
	Now that we have proved that $(\phi^\iota)''(c_2) \neq 0$, it is sufficient to prove that $$\frac{(\psi^o)''(\phi^\iota(c_2)) (\phi^\iota)'(c_2)}{(\psi^o)'(\phi(c_2))^2}=0.$$ In fact, since $(\phi^\iota)'(c_2)=0$, it is suffices to prove that
	$(\psi^o)'(\phi^\iota(c_2)) \neq 0$. Recall that for any $Z \in \C$, $(\psi^o)'(Z)=0$ if and only if
	there exists $n \geq 1$ such that $(\psi^o)'(Z-n)$ is a critical point for $f$; here, $Z=\phi^\iota(c_2)$
	and $\psi^o \circ \phi^\iota(c_2)=c_2$,  so we must prove that for all $n \geq 1$ and any critical point
	$c_i$ of $f$, $f^n(c_i) \neq c_2$. Since $c_1$ and $c_4$ escape, neither of their orbit can land on $c_2$; and since $c_2$ is not periodic under $f$, its own orbit cannot land on itself either.
	Since  $c_3$ is in the immediate parabolic basin, the orbit $(f^n(c_3))_{n \in \N}$
	is increasing, and so does not contain $c_2$ since $c_3>c_2$.
	
	Finally, it remains to argue that the orbits of the two non-real critical points $c'$ and $\overline{c'}$
	do not eventually land on $c_2$. To see that it cannot be the case, note that since the horn map $e$ of $f$
	has two parabolic fixed points at $0$ and $\infty$ corresponding to the ends of the Ecalle cylinder,
	each of those fixed points must attract singular values of $e$ distinct from themselves, see \cite{ABDPR}. The singular values of
	$e$ are the fixed points at $0$ and $\infty$, as well as the $\pi(c_i)$, where $c_i$ are the critical points
	of $f$ in the parabolic basin and $\pi(z)=e^{2i\pi \phi^\iota(z)}$. I f$f^n(c')=c_2$ for some $n \geq 1$, then by real symmetry we would also have $f^n(\overline{c'})=c_2$, and so
	$\pi(c')=\pi(\overline{c'})=\pi(c_2)$; but then $\pi(c_3)$ would be the only non-fixed singular value
	of $e$, which is impossible.
	
	Therefore $f$ has no critical relation, and so $(\psi^o)'(\phi^\iota(c_2)) \neq 0$; and
	$\fcal(f,c_2) \neq 0$ as announced.

	To sum things up, we have proved that for $n_0$ large enough, the polynomial $f_{t_{n_0}}$ satisfies
	properties (2)-(6). Since $t_{n_0}$
	is a super-attracting fixed point of $\lcal_{t_{n_0}}$ but persistently fixed, for $\eps_0>0$
	small enough, there exists $\tbold$ close to $t_{n_0}$ such that $f_\tbold$ satisfies (1);
	and by openness, if $\eps_0$ is small enough, $f_\tbold$ still satisfies (2)-(6).
\end{proof}

The next step is to use quasiconformal deformations to construct an immersed disk $D$ in parameter space
passing through $f_\tbold$, made of polynomials $p_u$ whose Lavaurs map has an attracting fixed point
of multiplier $e^{2i\pi u}$, $u \in \H$. We use on purpose the notation $p_u$ instead of $f_t$ to emphasize the
fact that except for $f_\tbold$, the polynomials $p_u$ do not a priori belong to the family $(f_t)_{t \in \R^*}$.

\begin{prop}\label{prop:frho}
	Let $p:=f_\tbold$ and $\eps_0>0$ be as in  Lemma \ref{lem:basepoint}.
	There exists a holomorphic map $\Phi: \H \to \mathcal{P}_{7}$ such that
	\begin{enumerate}
		\item $\Phi(u_0)=p$, for some $u_0 \in \H$ with $e^{2i\pi u_0}=\eps_0$
		\item For all $u \in \H$, the Lavaurs map of  $\Phi(u)=:p_u$ has a fixed point $z_u$ of multiplier
		$e^{2i\pi u} \in \D^*$; and $u \mapsto z_u$ is holomorphic
		\item All the maps $p_u$ are quasiconformally conjugated to $p$, the conjugacy being holomorphic
		outside of the grand orbit under $p$ of the attracting basin of $z_{u_0}:=x_\tbold$
		\item If $e^{2i\pi u} \in (0,1)$, then the conjugacy preserves the real line
		\item The set $\Phi(\H)$ is relatively compact in $\mathcal{P}_7$.
	\end{enumerate}
\end{prop}

\begin{figure}
	\begin{center}
		\includegraphics[scale=0.6]{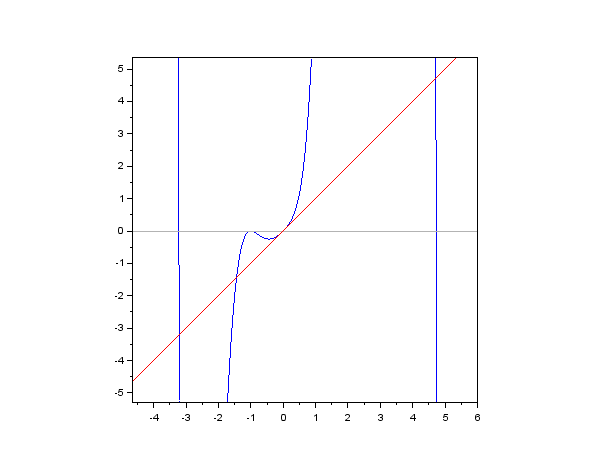}
		\caption{The graph of $f:=f_{t_\infty}$ (blue), with the line $y=x$ in red. We have $c_1 \approx -2.8$, $c_2=-1$, $c_3 \approx -0.4$,
			and $c_4 \approx 4$.
			The critical values $f(c_1)$ and $f(c_4)$ are out of the picture. }
		
	\end{center}
\end{figure}

\begin{proof}	
	Let $e: U \to \rs$ be the horn map of $g$; since $\lcal$ has an attracting fixed point $z_{u_0}:=x_\tbold$, so does
	$e$ (since they are semi-conjugated). Denote this attracting fixed point by  $x$.
	
	Let $u \in \H$, and $\mu$ be a Beltrami form invariant by $e$  (i.e. $e^*\mu=\mu$) such that the corresponding quasiconformal homeomorphism $h_\mu$
	conjugates $e$ to some holomophic map $e_\mu$ with an attracting fixed point of multiplier $e^{2i\pi u}$ :  $h_\mu \circ e = e_\mu \circ h_\mu$ and $e_\mu'(h_\mu(x) ) = e^{2i\pi u}$.
	We recall here briefly how to construct such a Beltrami form, and refer the reader to \cite{BF2014}
	for more details.
	If $\tau$ is a linearizing coordinate for the horn map $e$ near $x$, i.e. a holomorphic map
	defined near $p$ satisfying the functional equation $\tau \circ e = \eps_0 \tau$,
	we set:
	\begin{equation}
		\mu=\mu(u):=\tau^*\left(\frac{u-u_0}{u+u_0} \frac{z}{\overline{z}} \frac{d\overline{z}}{dz}\right)
	\end{equation}
	where $u_0 \in \H$ is any point such that $e^{2i\pi u_0}=\eps_0$. Notice that $u \mapsto \mu(u)$
	is holomorphic. In the rest of the proof, we fix $u \in \H$
	and just use the notation $\mu$ instead of $\mu(u)$.

	We choose the normalization
	of $h_\mu$ so that it fixes $0,1$ and $\infty$.
	Let $E(z):=e^{2i\pi z}$ and $T_1(z):=z+1$.
	We define:
	\begin{enumerate}
		\item $\nu:=E^* \mu$: so that $\nu = T_1^*\nu$, and $\nu = \ecal^*\nu$
		\item $\sigma:=\phi^*\nu$: so that $\sigma = g^*\sigma$ and $\sigma = \lcal^*\sigma$
		\item The quasiconformal homeomorphisms $h_\nu$ and $h_\sigma$  associated to $\nu, \sigma$
		respectively.
	\end{enumerate}
	Since $\nu=T_1^*\nu$, the map $h_\nu \circ T_1 \circ h_\nu^{-1} : \C \to \C$ is holomorphic;
	since it is conjugated to $T_1$, it is also a translation (distinct from the identity), and we choose the normalization of $h_\nu$
	so that  $h_\nu \circ T_1 \circ h_\nu^{-1} = T_1$ and $h_\nu(0)=0$. Similarly, since $\sigma=g^*\sigma$,
	the map $p_u:=h_\sigma \circ p \circ h_\sigma^{-1}$ is holomorphic, hence a polynomial (since it has same topological degree as $f$); it also
	has a parabolic fixed point with one attracting petal at the origin. We choose the unique normalization
	of $h_\sigma$ such that $p_u(z)=z+z^2+O(z^3)$. We set $\Phi(u):=p_u$; the holomorphic dependance
	$u \mapsto \mu(u)$ and the parametric version of Alhfors-Bers' Theorem imply that $\Phi$ is holomorphic
	on $\H$.
	
	We now define :
	\begin{enumerate}
		\item $\phi_\sigma:= h_\nu \circ \phi \circ h_{\sigma}^{-1} : h_\sigma(B) \to \C$,  where $B$ is
		the parabolic basin of $f$
		\item $\psi_{\nu}:=h_\sigma \circ \psi \circ h_\nu^{-1} : \C \to \C$
	\end{enumerate}

	\begin{lem}
		The map $\phi_{\sigma}$ is an incoming Fatou coordinate for $p_u$; and the map
		$\psi_\nu$ is an outgoing Fatou parametrization for $p_u$.
	\end{lem}

	\begin{proof}[Proof of the lemma]
	  	We start with $\phi_{\sigma}$. First, note that since $\sigma = \phi^*\nu$,
	  	the map $\phi_{\sigma}$ is holomorphic on $B_\sigma:=h_\sigma(B)$, which is exactly the parabolic
	  	basin of $p_u$. Then, note that :
	  	\begin{align*}
	  		\phi_{\sigma}\circ p_u &= h_\nu \circ \phi \circ h_{\sigma}^{-1}  \circ p_u \\
	  		&= h_\nu \circ \phi \circ g \circ h_\sigma^{-1} \\
	  		&=  h_\nu \circ T_1 \circ \phi \circ h_\sigma^{-1}\\
	  		&= T_1 \circ h_\nu  \circ \phi \circ h_\sigma^{-1} = T_1 \circ \phi_{\sigma}.\\
	  	\end{align*}
	  	So $\phi_\sigma$ conjugates $p_u$ on the whole parabolic basin to a translation, which
	  	means it is a Fatou coordinate.
	  	
	  	The proof is completely analoguous for $\psi_{\nu}$: first, to prove that $\psi_{\nu}$ is holomorphic,
	  	note that $\nu = \psi^*\sigma$. Indeed, $\nu = \ecal^*\nu = \psi^* \phi^* \nu = \psi^* \sigma$.
	  	To conclude, one can check directly that $\psi_{\nu} \circ T_1 = p_u \circ \psi_{\nu}$.

	\end{proof}

	As a consequence of the lemma, $\ecal_\nu:=h_\nu \circ \ecal \circ h_\nu^{-1}$ is a lifted horn map of $p_u$,
	and $\lcal_{\sigma}:=h_\sigma \circ \lcal \circ h_\sigma^{-1}$ is a Lavaurs map of $p_u$; and they have
	the same phase.  The phase could a priori be a non-zero, but we will prove that it is not the case. In order to do that, first we will prove that $E \circ \ecal_\nu = e_\mu \circ E$, i.e. that
	$e_\mu$ is a horn map that lifts to $\ecal_\nu$.
	
	Since $\nu = E^*\mu$, the map $E_\nu:=h_\mu \circ E \circ h_\nu^{-1} : \C \to \C^*$ is holomorphic.
	Moreover, since $E: \C \to \C^*$ is a universal cover, so is $E_\nu$. So $E_\nu$ is of the form
	$E_\nu(z)=\lambda e^{\alpha z}$, and with our choices of normalizations we find $E_\nu(z) = e^{2i\pi z} = E(z)$. So $E \circ h_\nu = h_\mu \circ E$.
	
	From this, we deduce the following:
	\begin{align*}
		E \circ \ecal_\nu &= E \circ h_\nu \circ \ecal \circ h_\nu^{-1} \\
		&= h_\mu \circ E \circ \ecal \circ h_\nu^{-1} \\
		&= h_\mu \circ e \circ E \circ h_\nu^{-1} \\
		&= h_\mu \circ e \circ h_\mu^{-1} \circ E \\
		&= e_\mu \circ E.
	\end{align*}
	Finally, it remains to observe that since $e_\mu$ is topologically conjugated to $e$, it also has
	two parabolic fixed points at $0$ and $\infty$ respectively, each of multiplier 1. Recall that
	the horn map of phase $0$ of a parabolic polynomial $f(z)=z+z^2+az^3+O(z^4)$ has multipliers at $0$ and
	$\infty$ both equal to $e^{2\pi^2(1-a)}$, and that the horn map  of phase $\varphi \in \C/\Z$ is obtained from the
	horn map $e$ of phase $0$ by multiplication by $e^{2i\pi \varphi}$. In
	particular, its multipliers at $0$ and $\infty$ are respectively $e^{2\pi^2(1-a)+2i\pi \varphi}$ and
	$e^{2\pi^2(1-a)-2i\pi \varphi}$. In this case, since both multipliers are equal to $1$, we must
	have $a=1$ and $\varphi=0$. Therefore, $\lcal_{\sigma}$ is the Lavaurs map of phase $0$ of $p_u$,
	and  $p_u(z)=z+z^2+z^3+O(z^4)$.
	
	Finally, if $\pi_\sigma(z):=e^{2i\pi \phi_\sigma(z)}$, then $\pi_\sigma \circ \lcal_\sigma = e_\mu \circ \pi_\sigma$, and
	$\pi_\sigma$ is locally invertible near $z_u:=h_\sigma(z_{u_0})$, and $\pi_\sigma(z_u) = h_\mu(x)$.
	Therefore, $z_u$ as a fixed point of $\lcal_{\sigma}$ has the same multiplier $e^{2i\pi u}$ as $h_\mu(x)$. This proves claims (1)-(3) of the proposition.

	To prove claim (4), note that if $e^{2i\pi u} \in (0,1)$ then the Beltrami form
	$\frac{u-u_0}{u+u_0} \frac{z}{\overline{z}} \frac{d\overline{z}}{dz}$ has real symmetry
	(since then $\frac{u-u_0}{u+u_0} \in \R$). We
	claim that this implies that $\sigma$ has real symmetry. Indeed,
	since $g(\R)=\R$, its Lavaurs map $\lcal$ maps a small interval $I \subset \R$ centered at $x_\tbold$ into itself.
	Moreover, the map $\tau \circ \pi$ semi-conjugates $\lcal$ to the multiplication by $\eps_0>0$; so
	$\tau \circ \pi$ maps $I$ into $\R$, which means that the holomorphic map $\tau \circ \pi$ is real: $\tau \circ \pi(\overline{z})=\overline{\tau \circ \pi(z)}$ for all $z$ in the parabolic basin of $g$.
	Therefore $\sigma = (\tau \circ \pi)^*\left(  \frac{u-u_0}{u+u_0} \frac{z}{\overline{z}} \frac{d\overline{z}}{dz} \right)$ has real symmetry, hence $h_\sigma$ restricts to a real homeomorphism.

Finally, $\Phi: \H \to \mathcal{P}_7$ is bounded in the space of polynomials of degree $7$. Indeed, by \cite[Prop. 4.4]{bassanelli2011distribution} the set of polynomials of given degree with given values of the Green function at the critical points is bounded, and since the conjugacy between the $p_u$ and $p$ is analytic outside of the parabolic basin, their Green functions have the same values at critical points.
\end{proof}

\begin{prop}\label{prop:existpar}
	With the same notations as before, there exists $p_0$ in the closure of $\Phi(\H)$
	such that the Lavaurs map of $p_0$ has a parabolic fixed point of multiplier 1.
\end{prop}

\begin{proof}
	Applying Proposition \ref{prop:frho}
	with $u_n = \frac{i}{n}$, we get a sequence of polynomials $p_{u_n}$ such that $p_{u_n}(z)=z+z^2+z^3+O(z^4)$,
	and the Lavaurs map $\lcal_n$ of $p_{u_n}$ has a fixed point $x_n$ of multiplier $e^{-2\pi/n}$.

	Each of the $p_{u_n}$ are quasiconformally conjugate to the real polynomial $f_\tbold$ from Lemma \ref{lem:basepoint} by a homeomorphism whose restriction to the real line is real and increasing, so the $p_{u_n}$ still satisfy the properties (3)-(6) from Lemma \ref{lem:basepoint}.
	
	By item (5) in Proposition \ref{prop:frho}, the sequence $(p_{u_n})_{n \in \N}$ is bounded
	in the space of degree 7 polynomials.
	So up to extracting, we may assume that
	\begin{enumerate}
		\item $p_{u_n}$ converges to a degree 7 polynomial $p_0$
		\item the critical points $c_{i,n}$ of $p_{u_n}$ converge to critical points $c_i$ of $p_0$
		\item the repelling fixed point $\xi_n$ converges
		to a non-attracting fixed point $\xi$ of $p_0$
		\item $x_n$ converges to $x \in \R$ and $y_n$ to  $y \in \R$.	
	\end{enumerate}
	
	We denote by $\lcal$ the Lavaurs map of $p_0$.
	If we can prove that $x$ lies in the parabolic basin of $p_0$, then we will get that $\lcal(x)=x$ and
	$\lcal'(x)=1$. To do that, it is enough to prove that $x \in (\xi, 0)$. But for all $n$, we have:
	$$\xi_n < y_n < x_n < c_{2,n} <0;$$
	hence 	$\xi < y \leq  x \leq  c_2 <0.$ The inequality $\xi < y$ is strict because as a limit of repelling fixed points,
	we have $|f'(\xi) | \geq 1$, so we can not have $y=\xi$ for otherwise we would have $\xi=f(\xi) = f(y)=c_2$ and so
	$f'(\xi)=0$, a contradiction. Similarly, we can not have $c_2=0$ since $p_0'(0)=1 \neq 0$. So $x \in (\xi,0)$ and
	$\xi$ is in the parabolic basin of $f$, and so $\lcal'(x)=1$ and $\lcal(x)=x$.
	Therefore $p_0$ has the desired property.
\end{proof}

\begin{prop}
	There exists a polynomial $g(z)=z+z^2+z^3+O(z^4)$ of degree 7 such that:
	\begin{enumerate}
		\item $\lcal$ has a Siegel fixed point $\zeta$ with diophantine multiplier, and
		\item {\black the pair $(g,\zeta)$  is non-degenerate.} %$(g,\zeta)$  does not satisfy condition (*).
	\end{enumerate}
\end{prop}

\begin{proof}
	Recall that $\mathcal{P}_7$ denotes the space of degree 7 polynomials of the form $f(z)=z+z^2+z^3+O(z^4)$,
	and let  $V = \{(f,\zeta) \in \psev \times \C : \zeta \in \mathcal{B}_f \}$: $V$ may be identified with an open set in
	$\C^5$. Finally, we consider $F:=\{(f,\zeta) \in V : \lcal(\zeta)=\zeta \}$, which
	is an analytic hypersurface of $V$.

	We consider the functions
	$\lambda : F \to \C$ and $\fcal: F \to \C$ defined as
	$\lambda(f,\zeta) = \lcal'(\zeta)$ and $\fcal(f,\zeta) =\frac{(\psi^o)''(\phi^\iota(\zeta)) (\phi^\iota)'(\zeta)}{(\psi^o)'(\phi^\iota(\zeta))^2}  + (\phi^\iota)''(\zeta)$, where $\phi^\iota$ and $\psi^o$ are the Fatou coordinate and parametrization of $f$. The function $\lambda$ is analytic on $F$,
	and $\fcal$ is meromorphic on $F$ and analytic on $\lambda^{-1}(\C^*)$, since $(\psi^o)'(\phi^\iota(z))=0$ implies
	that $\lcal'(z)=0$.

	Let $\Phi: \H \to \mathcal{P}_7$ be the map defined in Proposition \ref{prop:frho},
	and let $\tilde \Phi: \H \to F$ be the map given by $\Phi(u)=(p_u, z_u)$, where $z_u$
	is the fixed point of the Lavaurs map of $p_u$ with multiplier $e^{2i\pi u}$.
	Then $D:=\tilde \Phi(\H)$ is contained in one irreducible component $F_0$ of $F$.
	
		Let $p_0$ be the polynomial given by Proposition \ref{prop:existpar} such that its Lavaurs map
	has a parabolic fixed point $z_0$.
	By Proposition \ref{prop:existpar},
	 $(p_0, z_0)$ is in the closure of $D$ in $V$; therefore $(p_0, z_0) \in F_0$.

%	Assume for a contradiction that for all $(f,\zeta) \in F_0$ such that
%	$\lcal'(\zeta)$ has modulus one and diophantine argument, condition (*) holds.
{\black Assume for a contradiction that all pairs $(f,\zeta) \in F_0$ for which
	$\lcal'(\zeta)$ has modulus one and diophantine argument are degenerate.}
	Then by {\black the} density of diophantine numbers on the real line, we must have $\fcal(f,\zeta)=0$ on
	$\lambda^{-1}(S^1) \cap F_0$.
	Since for all $u \in \H$, $\lambda \circ \tilde \Phi(u)=e^{2i\pi u}$, the analytic map $\lambda$ is non-constant on
	$F_0$. In particular, $\lambda^{-1}(S^1)$ is a real analytic subset of $F_0$ of real codimension 1,
	non-empty since $\lambda(p_0,z_0)=1$.
	By Proposition \ref{prop:frho}, $D$ contains $(f_\tbold, x_\tbold)$, where $f_\tbold$ is the polynomial
	given by Lemma \ref{lem:basepoint}, and such that
	$\fcal(f_\tbold, z_\tbold) \neq 0$. So the analytic map $\fcal$ is not identically zero on
	$F_0$, and therefore it cannot vanish identically on $\lambda^{-1}(S^1) \cap F_0$, a contradiction.

\end{proof}

\bibliographystyle{abbrv}
\bibliography{Bibliography}

\end{document}